\documentclass[11pt]{amsart}
\usepackage{color}
\usepackage{amsmath}
\usepackage{amssymb}
\usepackage{amsfonts}
\usepackage{amsthm}
\usepackage{enumerate}
\usepackage[mathscr]{eucal}
\usepackage{verbatim}
\usepackage{amsthm}
\usepackage{amscd}

\newtheorem{theorem}{Theorem}[section]

\newtheorem{proposition}[theorem]{Proposition}

\newtheorem{lemma}[theorem]{Lemma}

\theoremstyle{definition}

\newtheorem{innercustomgeneric}{\customgenericname}
\providecommand{\customgenericname}{}

\newcommand{\newcustomtheorem}[2]{\newenvironment{#1}[1]
  {\renewcommand\customgenericname{#2}
   \renewcommand\theinnercustomgeneric{##1}\innercustomgeneric}{\endinnercustomgeneric}}

\newcustomtheorem{customthm}{Theorem}

\newcommand{\newcustomlemma}[2]{\newenvironment{#1}[1]
  {\renewcommand\customgenericname{#2}
   \renewcommand\theinnercustomgeneric{##1} \innercustomgeneric}{\endinnercustomgeneric}}

\newcustomlemma{customlemma}{Lemma}

\newcustomlemma{customproposition}{Proposition}

\numberwithin{equation}{section}

\newcommand{\rr}{\mathbb{R}}
\newcommand{\nn}{\mathbb{N}}

\newcommand{\rn}{\mathbb{R}^n}

\newcommand{\zz}{\mathbb{Z}}
\newcommand{\zn}{\mathbb{Z}^n}

\def\SS{{\mathscr{S}}}

\def\aaa{\vec{\boldsymbol{\alpha}}}

\def\xxxi{\vec{\boldsymbol{\xi}}}
\def\|{{\boldsymbol{|}}}

\def\fff{\vec{\boldsymbol{f}}}
\def\vv{\vec{\boldsymbol{v}}}

\def\yyy{\vec{\boldsymbol{y}}}

\def\ppp{\vec{\boldsymbol{p}}}

\def\sss{\vec{\boldsymbol{s}}}

\newcommand{\wt}{\widetilde}
\newcommand{\wh}{\widehat}

\begin{document}

\author{Loukas Grafakos}
\address{L. Grafakos, Department of Mathematics, University of Missouri, Columbia, MO 65211, USA} 
\email{grafakosl@missouri.edu}

\author{Bae Jun Park}
\address{B. Park, School of Mathematics, Korea Institute for Advanced Study, Seoul 02455, Republic of Korea}
\email{qkrqowns@kias.re.kr}

\thanks{The first author would like to acknowledge the support of  the Simons Foundation grant 624733. The second author is supported in part by NRF grant 2019R1F1A1044075 and by a KIAS Individual Grant MG070001 at the Korea Institute for Advanced Study}

\title[Multilinear multiplier theorem]{Characterization  of  
multilinear multipliers in terms of Sobolev space regularity}


\begin{abstract} 
We provide necessary and sufficient conditions for multilinear multiplier operators with symbols 
in $L^r$-based product-type Sobolev spaces uniformly over all annuli  to be bounded 
from products of Hardy spaces to a Lebesgue space. We   consider the case $1<r\le 2$ and
we   characterize   boundedness in terms of inequalities relating 
the Lebesgue indices (or Hardy indices), the dimension, and the regularity and integrability indices of the Sobolev space. The case $r>2$ cannot be handled by known techniques and remains open.  Our result not only  extends but also establishes  the sharpness of 
previous results of   Miyachi, Nguyen,   Tomita, and the first author  
\cite{Gr_Mi_Tom, Gr_Mi_Ng_Tom, Gr_Ng, Mi_Tom}, who only considered the case $r=2$.  
\end{abstract}

\maketitle

\section{Introduction}\label{intro}

Given a bounded function $\sigma$ on $\rn$, the linear Fourier multiplier operator $T_{\sigma}$ acting on a 
Schwartz function $f$ is given by
\begin{equation*}
T_{\sigma}f(x):=\int_{\rn}{\sigma(\xi)\wh{f}(\xi)e^{2\pi i\langle x,\xi\rangle}}d\xi ,
\end{equation*}
where $\wh{f}(\xi):=\int_{\rn}{f(x)e^{-2\pi i\langle x,\xi\rangle}}dx$ is the Fourier transform of $f$.
The classical Mikhlin multiplier theorem \cite{Mik} states that $T_{\sigma}$ admits an $L^p$-bounded extension for $1<p<\infty$ whenever
\begin{equation*}
\big| \partial_{\xi}^{\alpha}\sigma(\xi)\big|\lesssim_{\alpha}|\xi|^{-|\alpha|}, \quad \xi\not= 0
\end{equation*} for all multi-indices $\alpha$ with $|\alpha|\leq [n/2]+1$. H\"ormander \cite{Ho} refined this result by introducing the weaker condition
\begin{equation}\label{Hocondition}
\sup_{j\in\zz}{\big\Vert \sigma(2^j\cdot)\wh{\psi}\big\Vert_{L_s^2(\rn)}}<\infty
\end{equation} for $s>n/2$, where $L_s^2(\rn)$ denotes the standard fractional Sobolev space of order $s$ on $\rn$ and $\psi$ is a Schwartz function on $\rn$ whose Fourier transform is supported in the annulus $1/2<|\xi|<2$ and satisfies $\sum_{j\in\zz}{\wh{\psi}(\xi/2^j)}=1$ for $\xi\not= 0$.
Calder\'on and Torchinsky \cite{Ca_To} proved that if (\ref{Hocondition}) holds for $s>n/p-n/2$, then $T_{\sigma}$ is bounded on $H^p(\rn)$ for $0<p\leq 1$.
They also showed that $L_s^2$ in (\ref{Hocondition}) can be replaced by $L_s^r$ for the $L^p$-boundedness, using a complex interpolation method, and their assumptions were weakened by Grafakos, He, Honz\'ik, and Nguyen \cite{Gr_He_Ho_Ng}. 

The  multilinear counterparts of the Fourier multiplier theory have analogous formulations but substantially more complicated proofs. Let $m$ be a positive integer greater than $1$; this index will serve as the degree of the 
multilinearity of a Fourier multiplier. 
For a bounded function $\sigma$ on $(\rn)^m$ we define the corresponding $m$-linear multiplier operator $T_{\sigma}$ by
\begin{equation*}
T_{\sigma}\big(f_1,\dots,f_m \big)(x):={\int_{(\rn)^m}{\sigma(\xxxi)\Big(\prod_{j=1}^{m}\widehat{f_j}(\xi_j)\Big)e^{2\pi i\langle x,\sum_{j=1}^{n}{\xi_j} \rangle}}d\xxxi}
\end{equation*} for Schwartz functions $f_j$ on $\rn$, where $\xxxi:=(\xi_1,\dots,\xi_m)$ and $d\xxxi:=d\xi_1\cdots d\xi_m$.
As a multilinear extension of Mikhlin's result, Coifman and Meyer \cite{Co_Me2} proved that if $L$ is sufficiently large and $\sigma$ satisfies
\begin{equation*}
\big| \partial_{\xi_1}^{\alpha_1} \cdots\partial_{\xi_m}^{\alpha_m}\sigma(\xi_1,\dots,\xi_m)\big|\lesssim_{\alpha_1,\dots,\alpha_m}\big(|\xi_1|+\dots+|\xi_m|\big)^{-(|\alpha_1|+\dots +|\alpha_m|)}
\end{equation*} 
for multi-indices $\alpha_1,\dots,\alpha_m$ satisfying  $|\alpha_1|+\dots+|\alpha_m|\leq L$, then $T_{\sigma}$ is bounded from $L^{p_1}\times \cdots\times L^{p_m}$ to $L^p$ for all $1<p_1,\dots,p_m\leq \infty$ and $1<p<\infty$ with $1/p_1+\dots+1/p_m=1/p$. This result was extended to $p\leq 1$ by Kenig and Stein \cite{Ke_St} and Grafakos and Torres \cite{Gr_To}.

Let $\Psi^{(m)}$ be the $m$-linear counterpart of $\psi$. That is, $\Psi^{(m)}$ is a Schwartz function on $(\rn)^m$ having the properties:  
\begin{equation*}
\textup{Supp}(\widehat{\Psi^{(m)}})\subset \big\{\xxxi\in (\rn)^m: 1/2\leq |\xxxi|\leq 2 \big\}, \qquad \sum_{j\in\zz}{\widehat{\Psi^{(m)}}(\xxxi/2^j)}=1, \quad \xxxi\not= \vec{\boldsymbol{0}}.
\end{equation*}
Let $ (\vec{I}-\vec{\Delta})^{s/2}F=\Big( \big(1+4\pi^2( |\cdot_1|^2 +\dots+|\cdot_m|^2)\big)^{s/2}\wh{F}\Big)^{\vee}$ for 
a nice function on $  (\rn)^m$, where $F^{\vee}(\xxxi):=\wh{F}(-\xxxi)$ is the inverse Fourier transform of $F$. 
For $s\geq 0$ and $0<r<\infty$ we define the Sobolev space $L_s^r((\rn)^m)$ in terms of the finiteness of the norm: 
\begin{equation}\label{HHocondition}
\Vert F\Vert_{L_s^r((\rn)^m)}:=\big\Vert  (\vec{I}-\vec{\Delta})^{s/2}F\big\Vert_{L^r((\rn)^m)}. 
\end{equation}

Tomita \cite{Tom} was the first to obtain an $L^{p_1}\times\cdots\times L^{p_m}$ to $L^p$ boundedness for  $T_{\sigma}$ in the range  
 $1<p_1,\dots,p_m,p<\infty$,   under a condition analogous to (\ref{Hocondition}) for the 
 Sobolev space $L_s^r((\rn)^m)$. Grafakos and Si \cite{Gr_Si} extended this result to $p\leq 1$  using $L^r$-based Sobolev norms of $\sigma$ for $1<r\leq 2$:  
 
\begin{customthm}{A}(\cite{Gr_Si}) \label{knownresult1}
Let $1<r\leq 2$, $r\leq p_1,\dots,p_m<\infty$, $0<p<\infty$, and $1/p_1+\dots+1/p_m=1/p$. 
Suppose that
\begin{equation*}
s>mn/r.
\end{equation*}
If $\sigma$ satisfies 
\begin{equation}\label{Sob9988}
\sup_{j\in \zz}\big\Vert \sigma(2^j\cdot_1,\dots,2^j\cdot_m)\wh{\Psi^{(m)}}\big\Vert_{L^{r}_s((\rn)^m)}<\infty,
\end{equation} 
then we have
\begin{equation*}
\big\Vert T_{\sigma}\big(f_1,\dots,f_m \big) \big\Vert_{L^p(\rn)}\lesssim \sup_{j\in \zz}\big\Vert \sigma(2^j\cdot_1,\dots,2^j\cdot_m)\wh{\Psi^{(m)}}\big\Vert_{L^{r}_s((\rn)^m)} \prod_{i=1}^{m}{\Vert f_i\Vert_{L^{p_i}(\rn)}}
\end{equation*}
for   functions $f_1,\dots,f_m\in \SS(\rn)$.
\end{customthm} 

In the preceding  theorem and in the rest of this paper, 
$\SS(\rn)$ denotes   the space of all Schwartz functions on $\rn$.

The standard Sobolev space in \eqref{Sob9988} in many recent multiplier results is replaced by  
 a product type Sobolev space where the different powers of the Laplacian fall on different  variables
 $\xi_i\in \rn$.  For $s_1,\dots,s_m\geq 0$ and a function $F$ on $(\rn)^m$ let
\begin{equation*}
(I-\Delta_1)^{s_1/2}\cdots (I-\Delta_m)^{s_m/2}F:=\big((1+4\pi^2|\cdot_1|^2)^{s_1/2}\cdots (1+4\pi^2|\cdot_m|^2)^{s_m/2}\wh{F}\big)^{\vee}
\end{equation*} 
and for $0<r<\infty$ and $\sss:=(s_1,\dots,s_m)$,   define
\begin{equation*}
\Vert F\Vert_{L_{\sss}^{r}((\rn)^m)}:=\big\Vert (I-\Delta_1)^{s_1/2}\cdots (I-\Delta_m)^{s_m/2} F\big\Vert_{L^r((\rn)^m)}. 
\end{equation*} 
Here $\Delta_i$ is the Laplacian acting in the $i$th variable and $s_i\ge 0$. 
 Note that for $1<r<\infty$ and $s\geq s_1+\dots+s_m$, we have
\begin{equation}\label{pgcompare}
\big\Vert F \big\Vert_{L_{(s_1,\dots,s_m)}^{r}((\rn)^m)}\lesssim \big\Vert F \big\Vert_{L_s^{r}((\rn)^m)}. 
\end{equation}
When $r=2$, this is an immediate consequence of the pointwise estimate
\begin{equation*}
(1+4\pi^2|\xi_1|^2)^{s_1/2}\cdots (1+4\pi^2|\xi_m|^2)^{s_m/2}\leq (1+4\pi^2(|\xi_1|^2+\dots+|\xi_m|^2))^{s/2}
\end{equation*} for $s\geq s_1+\dots+s_m$.
In general, it follows from the fact that for $s\geq s_1+\dots+s_m$
\begin{equation*}
\mathcal{N}(\xi_1,\dots,\xi_m):=\frac{(1+4\pi^2|\xi_1|^2)^{s_1/2}\cdots (1+4\pi^2|\xi_m|^2)^{s_m/2}}{(1+4\pi^2(|\xi_1|^2+\dots+|\xi_m|^2))^{s/2}}
\end{equation*}
is an $L^r((\rn)^m)$-multiplier for $1<r<\infty$ due to the Marcinkiewicz multiplier theorem, which proves 
the inequality below
\begin{align*}
\Vert F\Vert_{L_{(s_1,\dots,s_m)}^r((\rn)^m)}&=\big\Vert (I_1-\Delta_1)^{s_1/2}\cdots (I_m-\Delta_m)^{s_m/2}F\big\Vert_{L^r((\rn)^m)}\\
&=\big\Vert T_{\mathcal{N}}(\vec{I}-\vec{\Delta})^{s/2}F\big\Vert_{L^r((\rn)^m)}\\
&\lesssim \big\Vert    (\vec{I}-\vec{\Delta})^{s/2}F  \big\Vert_{L^r((\rn)^m)}=\Vert F\Vert_{F_s^r((\rn)^m)},
\end{align*} 
where $T_{\mathcal{N}}$ is the multiplier operator associated with $\mathcal{N}$.

For a function $\sigma $ on $(\rn)^m$, 
throughout this work we will  use the notation: 
\begin{equation*}
\mathcal{L}_{\sss}^{r,\Psi^{(m)}}[\sigma]:=\sup_{j\in\zz }{\big\Vert \sigma(2^j\cdot_{1},\dots,2^j\cdot_{m})\wh{\Psi^{(m)}}\big\Vert_{L_{\sss}^r((\rn)^m)}}.
\end{equation*}
Research work has also  focused on boundedness properties of
 $T_\sigma$ under the assumption    
$\mathcal{L}_{\sss}^{r,\Psi^{(m)}}[\sigma]<\infty$ for given $\sss$. Under this assumption with $r=2$, 
Fujita and Tomita \cite{Fu_Tom1} provided weighted estimates for $T_\sigma$. Miyachi and Tomita \cite{Mi_Tom} obtained boundedness for bilinear multipliers (i.e., $m=2$) in the full range of indices $0<p,p_1,p_2\leq \infty$  extending a result of Calder\'on and Torchinsky \cite{Ca_To} to the bilinear setting; here Lebesgue spaces in the domain are replaced by Hardy spaces when $p_i\le 1$.  Multilinear extensions were later  provided by Grafakos, Miyachi, and Tomita \cite{Gr_Mi_Tom}, Grafakos and Nguyen \cite{Gr_Ng}, Grafakos, Miyachi, Nguyen, and Tomita \cite{Gr_Mi_Ng_Tom}, but all these results were proved only in the case $r=2$. 
We review most of these results in one formulation:

\begin{customthm}{B}(\cite{Gr_Mi_Tom, Gr_Mi_Ng_Tom, Gr_Ng, Mi_Tom})\label{knownresult2}
Let $0<p_1,\dots,p_m\leq \infty$, $0<p<\infty$, and $1/p_1+\dots+1/p_m=1/p$. Suppose that
\begin{equation}\label{minimal0}
s_1,\dots,s_m>n/2,\qquad \sum_{k\in J}\big({s_k}/{n}-{1}/{p_k} \big)>-{1}/{2}
\end{equation}
for every nonempty subset $J$ of $ \{1,\dots,m\}$.
If $\sigma$ satisfies $\mathcal{L}_{\sss}^{2,\Psi^{(m)}}[\sigma]<\infty$, then we have
\begin{equation}\label{boundresult0}
\big\Vert T_{\sigma}\big(f_1,\dots,f_m \big) \big\Vert_{L^p(\rn)}\lesssim \mathcal{L}_{\sss}^{2,\Psi^{(m)}}[\sigma] \prod_{i=1}^{m}{\Vert f_i\Vert_{H^{p_i}(\rn)}}
\end{equation}
for Schwartz functions $f_1,\dots,f_m\in \SS(\rn)$.
\end{customthm} 
Here and in the sequel,  $H^p(\rn)$ denotes the classical real Hardy space of Fefferman and Stein \cite{Fe_St2}. 
This space is defined for $ 0<p\leq \infty$ and coincides with $L^p(\rn)$ for $1<p\leq \infty$. \\

The optimality of (\ref{minimal0}) was also studied in \cite{Gr_Ng, Gr_Mi_Ng_Tom, Mi_Tom} and indeed, if (\ref{boundresult0}) holds, then we must necessarily have 
\begin{equation*}
s_1,\dots,s_m\geq n/2,\qquad \sum_{k\in J}\big({s_k}/{n}-{1}/{p_k} \big)\geq -{1}/{2}
\end{equation*} for every nonempty subset $J$ of $\{1,\dots,m\}$. However, this does not guarantee the validity of (\ref{boundresult0}) in the critical case
\begin{equation}\label{criticalcase}
\min{(s_1,\dots,s_m)}=n/2 \quad\text{or}\quad \sum_{k\in J}\big({s_k}/{n}-{1}/{p_k} \big)= -{1}/{2} \quad \text{for some } ~J\subset \{1,\dots,m\}
\end{equation} and recently, it was proved in Park \cite{Park4} that (\ref{boundresult0}) fails in the case (\ref{criticalcase}) as well. 
\begin{customthm}{C}(\cite{Park4})\label{knownresult3}
Let $0<p_1,\dots,p_m\leq \infty$ and $0<p<\infty$ with $1/p_1+\dots+1/p_m=1/p$.  
 If $\sss=(s_1,\dots,s_m)$ satisfies
\begin{equation*}
\min{(s_1,\dots,s_m)}\leq n/2 \quad\text{or}\quad \sum_{k\in J}{\big({s_k}/{n}-{1}/{p_k} \big)\leq -{1}/{2}} \quad \text{for some }~ J\subset \{1,\dots,m\},
\end{equation*}
then there exists a function $\sigma$ on $(\rn)^m$ such that $\mathcal{L}_{\sss}^{2,\Psi^{(m)}}[\sigma]<\infty$, but (\ref{boundresult0}) does not hold.
\end{customthm}

Therefore  (\ref{minimal0}) is a necessary and sufficient condition for (\ref{boundresult0}) to hold.\\

In this paper, we focus on the case $1<r\le 2$ and we prove necessary and sufficient conditions for bounded functions $\sigma $ on 
$(\mathbb R^{n})^m$ 
that satisfy the H\"ormander  
condition $\mathcal{L}_{\sss}^{r,\Psi^{(m)}}[\sigma]<\infty$ to be bounded multilinear multipliers.  
The case $r<2$ was also considered in \cite{Gr_He_Ng_Yan} but the results obtained there were
 non optimal. 
The  characterization we provide is given in terms of explicit inequalities relating different relevant indices and provides  generalizations   for Theorems  \ref{knownresult2} and \ref{knownresult3}, and  an extension of Theorem \ref{knownresult1} in view of (\ref{pgcompare}). 
The main result of this article is the following: 
\begin{theorem}\label{main}
Let $1<r\leq 2$, $s_1,\dots,s_m\geq 0$, $0<p_1,\dots,p_m\leq \infty$, $0<p<\infty$, and $1/p_1+\dots+1/p_m=1/p$.
Then the conditions 
\begin{equation}\label{minimal}
s_1,\dots,s_m>n/r  \qquad \textup{and} \qquad\sum_{k\in J}{\big({s_k}/{n}-{1}/{p_k}\big)}>-{1}/{r'} 
\end{equation} 
hold for every nonempty subset $J$ of $ \{1,2,\dots,m\}$ if and only if  every $T_{\sigma}$ with $ \mathcal{L}_{\sss}^{r,\Psi^{(m)}}[\sigma]<\infty$ satisfies 
\begin{equation}\label{boundresult}
\Vert T_{\sigma}(f_1,\dots,f_m)\Vert_{ L^p(\rn)}\lesssim \mathcal{L}_{\sss}^{r,\Psi^{(m)}}[\sigma]\,\,
\prod_{i=1}^{m}\Vert f_i\Vert_{H^{p_i}(\rn)}
\end{equation}
for $f_1,\dots,f_m\in \SS(\rn)$.
\end{theorem}

The implicit constant in \eqref{boundresult} depends only on the dimension $n$, the degree of multilinearity $m$, and the indices $p_j$, $s_j$, and $r$. Here $r'=r/(r-1)$. 
We remark that, when $r=2$, Theorem \ref{main} coincides with Theorem \ref{knownresult2} and \ref{knownresult3}. Moreover, since 
\begin{equation*}
s_1,\dots,s_m>n/r ~ \text{ implies }~ \sum_{k\in J}\big({s_k}/{n}-{1}/{p_k} \big)>-{1}/{r'}  ~\text{ for all }~ J \quad \text{ when }~ r\leq p_1,\dots,p_m,
\end{equation*}
and
\begin{equation*}
\mathcal{L}_{\sss}^{r,\Psi^{(m)}}[\sigma]\leq \sup_{j\in\zz}{\big\Vert \sigma(2^j\cdot_1,\dots,2^j\cdot_m)\wh{\Psi^{(m)}}\big\Vert_{L_s^r((\rn)^m)}} \quad \text{for }~s\geq s_1+\dots+s_m,
\end{equation*}
Theorem \ref{main} also covers Theorem \ref{knownresult1} and  extends  its   range 
of indices to $0<p_1,\dots,p_m\leq \infty$.

\subsection{Necessary condition}
In order to prove the direction $(\ref{boundresult})\Rightarrow (\ref{minimal})$ in Theorem \ref{main}, two different multipliers will be constructed based on an  idea contained in \cite{Park4}. However, the methods in \cite{Park4} essentially rely on   Plancherel's theorem to obtain the upper bound of
\begin{equation*}
\mathcal{L}_{\sss}^{2,\Psi^{(m)}}[\sigma]=\sup_{j\in\zz}{\Big\Vert \Big(\prod_{k=1}^{m}(1+4\pi^2|\cdot_k|^2)^{s_k/2}\Big)\big( \sigma(2^j\cdot_1,\dots,2^j\cdot_m)\wh{\Psi^{(m)}}\big)^{\vee}\Big\Vert_{L^2((\rn)^m)}}
\end{equation*} 
and this cannot be applied in the case   $1<r<2$ anymore.

To overcome this difficulty, we  benefit from a recent calculation of Grafakos and Park \cite{Gr_Park}
concerning a variant of the Bessel potentials that involves a logarithmic term.
For any $0<t,\gamma<\infty$ we define
\begin{equation}\label{hdefinition}
\mathcal{H}_{(t,\gamma)}(x):=\frac{1}{(1+4\pi^2|x|^2)^{t/2}}\frac{1}{(1+\ln{(1+4\pi^2|x|^2)})^{\gamma/2}}.
\end{equation}
We first observe that for any $t,\gamma>0$
\begin{equation}\label{hproperty1}
\mathcal{H}_{(t,\gamma)}(x-y)\geq \mathcal{H}_{(t,\gamma)}(x)\mathcal{H}_{(t,\gamma)}(y)
\end{equation} and
\begin{equation}\label{hproperty2}
\Vert \mathcal{H}_{(t,\gamma)}\Vert_{L^p(\rn)}<\infty \quad \text{if and only if}\quad  t>n/p \quad \text{or} \quad t=n/p, \gamma>2/p.
\end{equation}
Moreover, it was shown   in \cite{Gr_Park} that
\begin{equation*}
\big| \wh{\mathcal{H}_{(t,\gamma)}}(\xi)\big|\lesssim_{t,\gamma,n}e^{-|\xi|/2} \quad \text{for }~ |\xi|>1
\end{equation*} and when $0<t<n$,
\begin{equation*}
\big| \wh{\mathcal{H}_{(t,\gamma)}}(\xi)\big|\approx_{t,\gamma,n} |\xi|^{-(n-t)}(1+2\ln|\xi|^{-1})^{-\gamma/2}\quad \text{for }~ |\xi|\leq 1.
\end{equation*}
The estimates imply that
\begin{equation}\label{hproperty4}
\big\Vert \wh{\mathcal{H}_{(t,\gamma)}}\big\Vert_{L^p(\rn)}<\infty \quad \text{if and only if} \quad t>n-n/p \quad \text{or}\quad t=n-n/p, \gamma>2/p.
\end{equation}

These   properties provide us with tools that allow us to prove the following two propositions: 
\begin{proposition}\label{main21}
Let $1<r<\infty$, $0<p_1,\dots,p_m\leq \infty$, $0<p<\infty$, and $1/p_1+\dots+1/p_m=1/p$. 
Suppose that 
\begin{equation*}
s_1\leq s_2,\dots,s_m\qquad \text{and} \qquad  s_1\leq n/r.
\end{equation*}
Then there exists a function $\sigma$ on $(\rn)^m$ such that $\mathcal{L}_{\sss}^{r,\Psi^{(m)}}[\sigma]<\infty$, but 
$$
\Vert T_{\sigma}\Vert_{H^{p_1}\times\cdots\times H^{p_m}\to L^p}=\infty .
$$
\end{proposition}

\begin{proposition}\label{main22}
Let $1<r<\infty$, $0<p_1,\dots,p_m\leq \infty$, $0<p<\infty$, and  $1/p_1+\dots+1/p_m=1/p$.
Let $1\leq l\leq m$.
Suppose that $s_1,\dots,s_m>n/r$  and
\begin{equation*}
\sum_{k=1}^{l}{\big({s_k}/{n}-{1}/{p_k}\big)}\leq -{1}/{r'}.
\end{equation*}
Then there exists a function $\sigma$ on $(\rn)^m$ such that $\mathcal{L}_{\sss}^{r,\Psi^{(m)}}[\sigma]<\infty$, but 
$$
\Vert T_{\sigma}\Vert_{H^{p_1}\times\cdots\times H^{p_m}\to L^p}=\infty .
$$
\end{proposition}

The necessity part of Theorem \ref{main} is a consequence of the preceding   two propositions along with a rearrangement argument.\\

\subsection{Sufficiency condition}

The sufficiency condition part in Theorem \ref{main} is a consequence of the following four propositions 
combined with a rearrangement argument.
\begin{proposition}\label{propo1}
Let $1<r\leq 2$, $r\leq p_1,\dots,p_m\leq \infty$,  and $1/p=1/p_1+\dots+1/p_m$.
Suppose that 
\begin{equation}\label{conditionsss}
s_1,\dots,s_m>n/r.
\end{equation} 
If $\sigma$ satisfies $\mathcal{L}_{\sss}^{r,\Psi^{(m)}}[\sigma]<\infty$, then (\ref{boundresult}) holds.

\end{proposition}

\begin{proposition}\label{propo2}
Let $1<r\leq 2$, $1\leq l\leq m$, $0<p_1,\dots,p_l\leq 1$, $p_{l+1},\dots,p_m=\infty$, and $1/p=1/p_1+\dots+1/p_l$.
Suppose that 
\begin{equation}\label{conditionss}
s_{l+1},\dots,s_m>n/r,\qquad \sum_{k\in J}{\big({s_k}/{n}-{1}/{p_k} \Big)}>-{1}/{r'}
\end{equation}
for every nonempty subset $J\subset \{1,\dots,l\}$.
If $\sigma$ satisfies $\mathcal{L}_{\sss}^{r,\Psi^{(m)}}[\sigma]<\infty$, then (\ref{boundresult}) holds.
\end{proposition}

\begin{proposition}\label{propo3}
Let $1<r\leq 2$, $1\leq l<\rho\leq m$, $0<p_1,\dots,p_l\leq 1$, $r\leq p_{l+1},\dots,p_{\rho}<\infty$, $p_{\rho+1},\dots,p_m=\infty$, and $1/p=1/p_1+\dots+1/p_{\rho}$.
Suppose that (\ref{conditionss}) holds for every nonempty subset $J\subset \{1,\dots,l\}$.
If $\sigma$ satisfies $\mathcal{L}_{\sss}^{r,\Psi^{(m)}}[\sigma]<\infty$, then (\ref{boundresult}) holds.
\end{proposition}

\begin{proposition}\label{propo4}
Let $1<r\leq 2$ and $1\leq l\leq m$.
Suppose that $\mathfrak{L}$ is a subset of $\{1,\dots,m\}$ with $|\mathfrak{L}|=l$, and 
$$1<p_i<r \qquad \text{ for }~ i\in\mathfrak{L}$$
and 
$$0<p_i\leq 1 \quad \text{or} \quad r\leq p_i\leq \infty \qquad \text{ for }~i\in \{1,\dots,m\}\setminus\mathfrak{L}.$$
Suppose that (\ref{minimal}) holds for every nonempty subset $J$ of $ \{1,\dots,m\}$. If the function 
$\sigma$ satisfies $\mathcal{L}_{\sss}^{r,\Psi^{(m)}}[\sigma]<\infty$, then (\ref{boundresult}) holds.
\end{proposition}

The statements in the above propositions can be thought of as extensions   of Theorems \ref{knownresult1} and \ref{knownresult2} from $r=2$ to $1<r\leq 2$.
However,  the ingredients of their proofs are significantly more involved than in the case $r=2$, 
  in view of the lack of Plancherel's   identity.
The proofs we employ depend  on the Littlewood-Paley theory for the Hardy space $H^p$, but this certainly does not work for $H^{\infty}=L^{\infty}$ or $BMO$, and 
this  is the reason   the case $p_i=\infty$ was excluded in Theorem \ref{knownresult1}.
It was addressed in the proof of Theorem \ref{knownresult2} by applying a modified version of the Carleson measure estimate related to $BMO$ functions, which is contained in \cite{Gr_Mi_Tom}. We provide a new method to deal with this issue, using a generalization of Peetre's maximal function, saying $\mathfrak{M}_{\sigma,2^j}^tf$, introduced by Park \cite{Park3}. As we have an $L^{\infty}(\ell^2)$ characterization of $BMO$ with this maximal function,   stated in Lemma \ref{equivalence}, we may still utilize the Littlewood-Paley theory to obtain $H^{p_i}$ bounds for all $0<p_i\leq \infty$.

The proof of Proposition \ref{propo1} is based on that of Theorem \ref{knownresult1} for which the pointwise estimate in Lemma \ref{keyestilemma} below is essential. In 
Propositions \ref{propo2} and \ref{propo3}   at least one index $p_i$ satisfies $0<p_i\leq 1$ and  the $H^{p_i}$ atomic decomposition is very useful. 
In this case we need to employ an approximation argument for $\sigma$ as we don't know that we can interchange infinite sums of atoms and the action of the operator as in
\begin{equation*}
T_{\sigma}\big(f_1,\dots,f_m\big)=\sum_{k_1=1}^{\infty}\cdots\sum_{k_m=1}^{\infty}\lambda_{1,k_1},\cdots \lambda_{l,k_l}T_{\sigma}\big(a_{1,k_l},\dots,a_{l,k_l},f_{l+1},\dots,f_m\big)
\end{equation*} 
for functions $f_i\in H^{p_i}(\rn)$ with atomic representation 
$f_i=\sum_{k_i=1}^{\infty}{\lambda_{i,k_i}a_{i,k_i}}$, $1\leq i\leq l$. This regularization of the multiplier was also used  in \cite{Gr_Ng} but here it is stated in Lemma \ref{switchsum}. Afterwards, we   apply the method of Grafakos and Kalton \cite{Gr_Ka} and a  pointwise estimate of the form
\begin{equation}\label{keypointest}
\big| T_{\sigma}\big(a_{1,k_1},\dots,a_{l,k_l},f_{l+1},\dots,f_m\big)(x)\big|\lesssim \mathcal{L}_{\sss}^{r,\Psi^{(m)}}[\sigma]b_1(x)\cdots b_l(x) F_{l+1}(x)\cdots F_{m}(x)
\end{equation} where $\Vert b_i\Vert_{L^{p_i}(\rn)}\lesssim 1$  and $\Vert F_i\Vert_{L^{p_i}(\rn)}\lesssim \Vert f_i\Vert_{L^{p_i}(\rn)}$.
Since the above estimate separates the left-hand side to $m$ functions of $x$, we may now apply   
H\"older's inequality with exponents $1/p=1/p_1+\dots+1/p_m$.
The main idea in the proof of Proposition \ref{propo4} is a multilinear extension of the complex interpolation method of Calder\'on \cite{Ca} and Calder\'on and Torchinsky \cite{Ca_To}. Specifically, we   apply the interpolation to Propositions \ref{propo1}, \ref{propo2}, and \ref{propo3} to obtain (\ref{boundresult}) in the entire  range $0<p_1,\dots,p_m\leq \infty$.\\

{\bf Remark.}
The direction $(\ref{boundresult})\Rightarrow (\ref{minimal})$ is valid even for   $2<r<\infty$, in view of Propositions~\ref{main21} and ~\ref{main22}. Thus, 
 under the assumption $\mathcal{L}_{\sss}^{r,\Psi^{(m)}}[\sigma]<\infty$
   conditions $(\ref{minimal})$  are necessary for the 
boundedness of $T_\sigma$ 
  for all $r$ in the range $1<r<\infty$. However the sufficiency of $(\ref{minimal})$ for the boundedness of $T_\sigma$, i.e., the
  direction $(\ref{minimal}) \Rightarrow (\ref{boundresult})$
  is missing in the case $r>2$. It seems that our techniques are not applicable in this case. We hope to address
  this problem in the future but we welcome interested researchers to investigate this topic as well. \\

{\bf Organization.}
Section \ref{preliminary} contains some preliminary facts that are crucial in the proofs of the preceding  propositions. The proofs of Propositions \ref{main21} - \ref{propo4} are given in Sections \ref{proofpropo1} - \ref{proofpropo6}. Some key lemmas that appear in the proofs of the propositions are contained in the last section. \\

{\bf Notation.}
We denote by $\nn$ and $\zz$ the sets of   natural numbers and   integers, respectively. We   use the symbol $A\lesssim B$ to indicate that $A\leq CB$ for some constant $C>0$ independent of the variable quantities $A$ and $B$, and $A\approx B$ if $A\lesssim B$ and $B\lesssim A$ hold simultaneously.
The set of all dyadic cubes in $\rn$ is denoted by   $\mathcal{D}$, and for each $j\in\mathbb{Z}$ we 
designate    $\mathcal{D}_{j}$ to be  the subset of  $\mathcal{D}$ consisting of dyadic cubes with side length $2^{-j}$. For each $Q\in\mathcal{D}$, $\chi_Q$ denotes the characteristic function of $Q$.
We also use the notation $\fff:=(f_1,\dots,f_m)$, $\vv:=(v_1,\dots,v_m)$, and $\langle x\rangle:=(1+4\pi^2|x|^2)^{1/2}$.

\section{Preliminaries}\label{preliminary}

Let $\phi$ be a Schwartz function on $\rn$ with $\wh{\phi}(0)= 1$. For $0<p\leq \infty$ the Hardy space $H^p(\rn)$ contains all tempered distributions $f$ on $\rn$ which satisfy 
\begin{equation*}
\Vert f\Vert_{H^p(\rn)}:=\big\Vert \sup_{j\in\zz}{|\phi_j\ast f|} \big\Vert_{L^p(\rn)}<\infty
\end{equation*} where $\phi_j:=2^{jn}\phi(2^j\cdot)$.
According to \cite{Fr_Ja2, Tr}, different choices of   Schwartz function $\phi$ with $\wh{\phi}(0)= 1$ provide equivalent
Hardy space norms.
In this paper we fix a 
Schwartz function $\psi$    on $\rn$ whose Fourier transform is supported in the annulus $1/2<|\xi|<2$ and satisfies $\sum_{j\in\zz}{\wh{\psi}(\xi/2^j)}=1$ for $\xi\not= 0$. Set $\wh{\psi}(\cdot /2^j)= \wh{\psi_j} $. 
Then we define a
  function  $\phi \in \SS(\rn)$  by    
\begin{equation}\label{phidef}
\wh{\phi}(\xi):=\begin{cases}
\sum_{j\leq 0}{\wh{\psi_j}(\xi)}, & \xi\not= 0\\
1,& \xi=0,
\end{cases}
\end{equation}
 and let $\phi_j:=2^{jn}\phi(2^j\cdot)$ so that $\wh{ \phi_j} =\wh{ \phi } (\cdot/2^j)$. 
Note that $H^p(\rn)=L^p(\rn)$ for all $1<p\leq \infty$. A nice feature of the Hardy spaces $H^p$ for $0<p\leq 1$ is their atomic decomposition. More precisely, when $N$ is a positive integer greater or equal to $[n/p-n]+1$, every $f$ in $H^p(\rn)$, $0<p\leq 1$, can be written as $ \sum_{k=1}^{\infty}{\lambda_k a_k}$, where 
$\lambda_k$ are coefficients satisfying $\big( \sum_{k=1}^{\infty}{|\lambda_k|^{p}}\big)^{1/p}\lesssim \Vert f\Vert_{H^p(\rn)}$ and 
$a_k$ are $L^{\infty}$-atoms for $H^p$; this means that there exist cubes $Q_k$ such that $\textup{Supp}(a_k)\subset Q_k$, $\Vert a_k\Vert_{L^{\infty}(\rn)}\leq |Q_k|^{-1/p}$, and $\int_{Q_k}{x^{\alpha}a_k(x)}dx=0$ for all multi-indices $\alpha$ with $|\alpha|\leq N$.

The Hardy space $H^p$ can be characterized in terms of 
 Littlewood-Paley theory. For $0<p<\infty$  we have
\begin{equation}\label{littlewood}
\Vert f\Vert_{H^p(\rn)}\approx \Big\Vert \Big(\sum_{j\in\zz}{\big| \psi_j\ast f\big|^2} \Big)^{1/2}\Big\Vert_{L^{p}(\rn)}
\end{equation} 
where $\psi_j$ is a Littlewood-Paley function defined as above. This property is also independent of the choice of functions $\psi_j$ because of the Calder\'on reproducing formula and the Fefferman-Stein vector-valued maximal inequality   \cite{Fe_St} which states that
\begin{equation}\label{maximal1}
\big\Vert \big\{ \mathcal{M}_tf_j\big\}_{j\in\zz} \big\Vert_{L^p(\ell^q)}\lesssim \Vert \{f_j\}_{j\in\zz}\Vert_{L^p(\ell^q)} \qquad \text{ for }~ t<p,q<\infty
\end{equation} 
where $\mathcal{M}f(x):=\sup_{Q:x\in Q}{|Q|^{-1}\int_Q{|f(y)|}dy}$ is the Hardy-Littlwood maximal functions and $\mathcal{M}_tf(x):=\big(\mathcal{M}(|f|^t)\big)^{1/t}$ for $0<t<\infty$.
Note that (\ref{maximal1}) also holds for $0<p<\infty$, $q=\infty$ or for $p=q=\infty$.

For $j\in \zz$, $s>0$, and $0<t\leq \infty$ we now define
\begin{equation*}
\mathfrak{M}_{s,2^j}^{t}f(x):=2^{jn/t}\Big\Vert \frac{f(x-\cdot)}{(1+2^j|\cdot|)^{s}}\Big\Vert_{L^t(\rn)},
\end{equation*}
which is a generalization of the Peetre's maximal function $\mathfrak{M}_{s,2^j}f(x):=\mathfrak{M}_{s,2^j}^{\infty}f(x)$. 
It is easy to verify that if $0<t<\infty$ and $s>n/t$, then 
\begin{equation}\label{maximalcompare}
\mathfrak{M}_{s,2^j}^{t}f(x)\lesssim \mathcal{M}_tf(x), \quad \text{ uniformly in }~j\in\zz.
\end{equation}  
Moreover, for $s>0$, $0<t\leq r\leq \infty$, and $j\in \mathbb{Z}$, we have
\begin{equation}\label{mcomposition}
\mathfrak{M}_{s,2^j}^{r}\mathfrak{M}_{s,2^j}^tf(x)\lesssim \mathfrak{M}_{s,2^j}^tf(x).
\end{equation}
See \cite{Park3} for more details. 

Elementary considerations reveal that for $s>0$ and $Q\in\mathcal{D}_j$
\begin{equation*}
 \sup_{y\in Q}{|f(y)|}\lesssim \inf_{y\in Q}{\mathfrak{M}_{s,2^j}f(y)}
\end{equation*}
and then it follows from (\ref{mcomposition}) that for $0<t< \infty$
\begin{equation}\label{infmax2}
\sup_{y\in Q}{\mathfrak{M}_{s,2^j}^tf(y)}\lesssim \inf_{y\in Q}{\mathfrak{M}_{s,2^j}\mathfrak{M}_{s,2^j}^{t}f(y)}\lesssim \inf_{y\in Q}{\mathfrak{M}_{s,2^j}^tf(y)}.
\end{equation} 
In addition, the following maximal inequality holds. 
\begin{lemma}[\cite{Park3}]\label{maximal2}
Let $0<p,q,t \leq \infty$ and $s>n/\min{(p,q,t)}$. Suppose that the Fourier transform of $f_j$ is supported in a ball of radius $A2^j$ for some $A>0$.  
\begin{enumerate}
\item For $0<p<\infty$ and $0<q\le \infty$ or for $p=q=\infty$, we have
\begin{equation*}
 \big\Vert \big\{ \mathfrak{M}_{s,2^j}^tf_j\big\}_{j\in\mathbb{Z}}\big\Vert_{L^p(\ell^q)}\lesssim \big\Vert \{f_j\}_{j\in\mathbb{Z}}\big\Vert_{L^p(\ell^q)}.
\end{equation*}
\item For $p=\infty$, $0<q<\infty$, and $\mu\in\mathbb{Z}$, we have
\begin{equation*}
\sup_{P\in\mathcal{D}_{\mu}}{\Big( \frac{1}{|P|}\int_P{\sum_{j=\mu}^{\infty}{\big(\mathfrak{M}_{s,2^j}^tf_j(x) \big)^q}}dx\Big)^{1/q}}\lesssim \sup_{P\in\mathcal{D}_{\mu}}{\Big( \frac{1}{|P|}\int_P{\sum_{j=\mu}^{\infty}{|f_j(x) |^q}}dx\Big)^{1/q}}
\end{equation*} where the constant in the inequality is independent of $\mu$.
\end{enumerate}
\end{lemma}
Using Lemma \ref{maximal2}  we can prove the following result.
\begin{lemma}[\cite{Park3}]\label{equivalence}
Let $0<p\leq \infty$, $0<t\leq \infty$, $0<\gamma<1$, and $s >{n}/{\min{(p,2,t)}}$.
 Then for any dyadic cubes $Q\in\mathcal{D}$, there exists a proper measurable subset $S_Q$ of $Q$, depending on $\gamma, s,t,f$, such that
$|S_Q|>\gamma |Q|$ and
\begin{equation*}
\Vert f\Vert_{X^p}\approx \Big\Vert \Big\{ \sum_{Q\in\mathcal{D}_j}{\Big( \inf_{y\in Q}{\mathfrak{M}_{s ,2^j}^t\big(\psi_j\ast f\big)(y) }\Big)\chi_{S_Q}}\Big\}_{j\in\mathbb{Z}}\Big\Vert_{L^{p}(\ell^2)}
\end{equation*} 
where $X^p=H^p$ for $0<p<\infty$ and $X^{\infty}=BMO$.
\end{lemma}

We observe that if $S_Q$ is a   measurable subset of $Q\in\mathcal{D}$ with $|S_Q|>\gamma|Q|$ for some $0<\gamma<1$,
then  we have
\begin{equation}\label{chamaximal}
\chi_{Q}(x)\lesssim_{\tau,\gamma} \mathcal{M}_{\tau}(\chi_{S_Q})(x),
\end{equation}
which is due to the fact that for $x\in Q$
\begin{equation}\label{inversechi}
1<\frac{1}{\gamma^{1/{\tau}}}\frac{|S_Q|^{1/{\tau}}}{|Q|^{1/{\tau}}}=\frac{1}{\gamma^{1/{\tau}}}\Big(\frac{1}{|Q|}\int_Q{\chi_{S_Q}(y)}dy \Big)^{1/{\tau}}\leq \gamma^{-1/{\tau}}\mathcal{M}_{\tau}(\chi_{S_Q})(x).
\end{equation}

Based on the $L^{\infty}(\ell^2)$ characterization of $BMO$  from Lemma \ref{equivalence}, we have the following lemma, which will be essential in obtaining $L^{\infty}$ bounds in the proof of our main theorem. 
\begin{lemma}\label{bmoboundlemma}
Let $0<p,t<\infty$, $N\geq 3$, and
\begin{equation*}
s_1>{n}/{\min{(p,t)}}, \qquad s_i>{n}/{\min{(2,t)}} \quad 2\leq i\leq N.
\end{equation*}
Let $\varphi_j,\vartheta_j\in \SS(\rn)$, $j\in\zz$, satisfy $\textup{Supp}(\wh{\varphi_j})\subset \{\xi\in\rn: |\xi|\leq C2^j\}$ and $\textup{Supp}(\wh{\vartheta_j})\subset \{\xi\in\rn: D^{-1}2^j\leq |\xi|\leq D2^j\}$
for some $C,D>1$.
Suppose that $T^1$ and $T^2$ are the bilinear operators and $T^3$ is the $N$-linear operator, defined by
\begin{align*}
T^1(f_1,f_2)(x)&:=\Big[ \sum_{j\in\zz}{\big( \mathfrak{M}_{s_1,2^j}^t\big(\varphi_j\ast f_1\big)(x)\big)^2\big( \mathfrak{M}_{s_2,2^j}^{t}\big(\vartheta_j\ast f_2\big)(x)\big)^2}\Big]^{1/2},\\
T^2(f_1,f_2)(x)&:= \sum_{j\in\zz} \mathfrak{M}_{s_1,2^j}^t\big(\vartheta_j\ast f_1\big)(x) \mathfrak{M}_{s_2,2^j}^{t}\big(\vartheta_j\ast f_2\big)(x),\\
T^3(f_1,\dots,f_N)(x)&:=\sum_{j\in\zz}\mathfrak{M}_{s_1,2^j}^t\big(\varphi_j\ast f_1\big)(x)\prod_{i=2}^{N}\mathfrak{M}_{s_i,2^j}^{t}\big(\vartheta_j\ast f_i\big)(x)
\end{align*} for $f_1,\dots,f_N\in \SS(\rn)$ and $x\in \rn$.
Then we have
\begin{align}
\big\Vert T^1(f_1,f_2)\big\Vert_{L^p(\rn)} &\lesssim \Vert f_1\Vert_{H^p(\rn)}\Vert f_2\Vert_{BMO(\rn)} \label{t1est} \\
\big\Vert T^2(f_1,f_2)\big\Vert_{L^p(\rn)} &\lesssim \Vert f_1\Vert_{H^p(\rn)}\Vert f_2\Vert_{BMO(\rn)}  \label{t2est}  \\
\big\Vert T^3(f_1,\dots,f_N)\big\Vert_{L^p(\rn)} &\lesssim \Vert f_1\Vert_{H^p(\rn)}\prod_{i=2}^{N}\Vert f_i\Vert_{BMO(\rn)}. \label{t3est}
\end{align}
\end{lemma}

\begin{proof}
We will  only be concerned with (\ref{t1est}) and (\ref{t3est}) as the proof of (\ref{t2est}) is very similar to that of (\ref{t3est}) with $N=3$.

Since dyadic cubes with the same side length are pairwise disjoint,
the left-hand side of (\ref{t1est}) can be written as
\begin{equation*}
\Big\Vert \Big( \sum_{j\in\zz}\sum_{Q\in\mathcal{D}_j}  \big( \mathfrak{M}_{s_1,2^j}^t\big(\varphi_j\ast f_1\big)\big)^2\big( \mathfrak{M}_{s_2,2^j}^{t}\big(\vartheta_j\ast f_2\big)\big)^2 \chi_Q  \Big)^{1/2}\Big\Vert_{L^p(\rn)}
\end{equation*} and the estimate (\ref{infmax2}) implies that the preceding expression is dominated by a constant multiple of
\begin{equation}\label{squarecharacter}
\Big\Vert \Big( \sum_{j\in\zz}\sum_{Q\in\mathcal{D}_j}  \Big( \inf_{y\in Q}\mathfrak{M}_{s_1,2^j}^t\big(\varphi_j\ast f_1\big)(y)\Big)^2\Big( \inf_{y\in Q}\mathfrak{M}_{s_2,2^j}^{t}\big(\vartheta_j\ast f_2\big)(y)\Big)^2 \chi_Q  \Big)^{1/2}\Big\Vert_{L^p(\rn)}.
\end{equation}
According to Lemma \ref{equivalence}, for each $Q\in \mathcal{D}$ we can choose a proper measurable subset $S_Q$ of $Q$ such that $|S_Q|>\frac{1}{2}|Q|$ and 
\begin{equation}\label{BMOcha}
\Vert f_2\Vert_{BMO}\approx \Big\Vert \Big\{ \sum_{Q\in\mathcal{D}_j}{\Big( \inf_{y\in Q}{\mathfrak{M}_{s_2,2^j}^t\big(\vartheta_j\ast f_2\big)(y) }\Big)\chi_{S_Q}}\Big\}_{j\in\mathbb{Z}}\Big\Vert_{L^{\infty}(\ell^2)}.
\end{equation} 
Here, we may use $\vartheta_j$, instead of $\psi_j$, because of the Calder\'on reproducing formula and (\ref{mcomposition}).
Now, using (\ref{chamaximal}) with $\tau<\min{(p,2)}$ and the vector-valued maximal inequality (\ref{maximal1}) of $\mathcal{M}_{\tau}$ with the index set $\{Q\}_{Q\in\mathcal{D}}$, $\chi_Q$ can be replaced by $\chi_{S_Q}$ in (\ref{squarecharacter}) and then   H\"older's inequality yields that (\ref{squarecharacter}) is less than a constant times
\begin{align*}
\Big\Vert \sup_{j\in\zz}{\mathfrak{M}_{s_1,2^j}^t\big(\varphi_j\ast f_1\big)}\Big\Vert_{L^p(\rn)}\Big\Vert \Big\{ \sum_{Q\in\mathcal{D}_j}{\Big( \inf_{y\in Q}{\mathfrak{M}_{s_2,2^j}^t\big(\vartheta_j\ast f_2\big)(y) }\Big)\chi_{S_Q}}\Big\}_{j\in\mathbb{Z}}\Big\Vert_{L^{\infty}(\ell^2)}.
\end{align*}
The second term is definitely comparable to $\Vert f_2\Vert_{BMO}$ due to (\ref{BMOcha}) and the first one can be estimated by
\begin{equation*}
\Big\Vert \sup_{j\in\zz}{\big| \varphi_j\ast f_1\big|}\Big\Vert_{L^p(\rn)}\approx \Vert f_1\Vert_{H^{p}(\rn)}
\end{equation*} 
in view of Lemma \ref{maximal2}. This proves (\ref{t1est}).

Similarly, for each $Q\in\mathcal{D}$ we choose proper measurable subsets $S_{Q}^{2}$ and $S_Q^{3}$ of $Q$ such that $|S_Q^{2}|, |S_Q^{3}|>\frac{3}{4} |Q|$ and 
\begin{equation*}
\Vert f_k\Vert_{BMO}\approx \Big\Vert \Big\{ \sum_{Q\in\mathcal{D}_j}{\Big( \inf_{y\in Q}{\mathfrak{M}_{s_k,2^j}^t\big(\vartheta_j\ast f_k\big)(y) }\Big)\chi_{S_Q^k}}\Big\}_{j\in\mathbb{Z}}\Big\Vert_{L^{\infty}(\ell^2)}, \quad k=2,3.
\end{equation*} 
We note that $|S_Q^2\cap S_Q^3|>\frac{1}{2}|Q|$ and thus (\ref{chamaximal}) implies 
\begin{equation*}
\chi_Q(x)\lesssim_{\tau}\mathcal{M}_{\tau}\big(\chi_{S_Q^2\cap S_Q^3}\big)(x), \quad \text{ for all }~ 0<\tau<\infty.
\end{equation*}
Choose $\tau<\min{(1,p)}$. Then we can prove that the left-hand side of (\ref{t3est}) is smaller than a constant times
\begin{align*}
&\Big\Vert \sum_{j\in\zz}\sum_{Q\in\mathcal{D}_j} \Big( \inf_{y\in Q}{\mathfrak{M}_{s_1,2^j}^t\big(\varphi_j\ast f_1\big)(y)}\Big) \prod_{k=2}^{N}\Big( \inf_{y\in Q}{\mathfrak{M}_{s_k,2^j}^{t}\big(\vartheta_j\ast f_k \big)(y)}\Big) \chi_{S_Q^2}\chi_{S_Q^{3}}   \Big\Vert_{L^p(\rn)}\\
&\leq \Big\Vert \sup_{j\in\zz}{\mathfrak{M}_{s_1,2^j}^t\big(\varphi_j\ast f_1\big)}\Big\Vert_{L^p(\rn)} \prod_{k=2}^{3}\Big\Vert \Big\{ \sum_{Q\in\mathcal{D}_j}{\Big( \inf_{y\in Q}{\mathfrak{M}_{s_k,2^j}^t\big(\vartheta_j\ast f_k\big)(y) }\Big)\chi_{S_Q^k}}\Big\}_{j\in\mathbb{Z}}\Big\Vert_{L^{\infty}(\ell^2)}\\
&\qquad \qquad \qquad \qquad \qquad \qquad \qquad  \qquad \qquad  \times \prod_{k=4}^{N}\big\Vert \big\{ \mathfrak{M}_{s_k,2^j}^t\big(\vartheta_j\ast f_k\big)\big\}_{j\in\zz}\big\Vert_{L^{\infty}(\ell^{\infty})} \\
&\lesssim \Vert f_1\Vert_{H^{p_1}(\rn)}\prod_{k=2}^{N}\Vert f_k\Vert_{BMO}
\end{align*} as desired.
Here, we used the fact that for $4\leq k\leq N$,
\begin{align*}
\big\Vert \big\{ \mathfrak{M}_{s_k,2^j}^t\big(\vartheta_j\ast f_k\big)\big\}_{j\in\zz}\big\Vert_{L^{\infty}(\ell^{\infty})}\lesssim \big\Vert  \{\vartheta_j\ast f_k\}_{j\in\zz}\big\Vert_{L^{\infty}(\ell^{\infty})}\approx \Vert f_k\Vert_{\dot{F}_{\infty}^{0,\infty}}\lesssim \Vert f_k\Vert_{\dot{F}_{\infty}^{0,2}}\approx \Vert f_k\Vert_{BMO}
\end{align*} where $\dot{F}_p^{0,q}$ is the homogeneous Triebel-Lizorkin space, and Lemma \ref{maximal2}, the embedding $\dot{F}_{\infty}^{0,2}\hookrightarrow \dot{F}_{\infty}^{0,\infty}$, and the characterization $BMO=\dot{F}_{\infty}^{0,2}$ are applied. We refer to \cite{Park3} for more details.
\end{proof}

The following lemma is the main tool used to derive  pointwise estimates like (\ref{keypointest}). In fact, similar results can be found in \cite{Gr_Mi_Tom, Gr_Mi_Ng_Tom, Gr_Ng, Gr_Si, Mi_Tom} with the maximal function $\mathcal{M}_t$, but here we replace $\mathcal{M}_t$ by $\mathfrak{M}_{s_k,2^j}^t$ in order to apply the arguments in Lemmas \ref{equivalence} and \ref{bmoboundlemma}.
\begin{lemma}\label{keyestilemma}
Let $1<t \leq 2$ and $s_1,\dots,s_m>n/t$. Suppose that $\sigma$ is a bounded function with a compact support in $(\rn)^m$.
Then we have
\begin{equation*}
\big| T_{\sigma}\fff(x)\big|\lesssim \big\Vert \sigma(2^j\cdot)\big\Vert_{L_{\sss}^{t}((\rn)^m)}\prod_{k=1}^{m}{\mathfrak{M}_{s_k,2^j}^tf_k(x)}, \qquad \text{ uniformly in }~ j\in\zz.
\end{equation*}
\end{lemma}
\begin{proof}
Using the H\"older inequality, we obtain
\begin{align*}
\big| T_{\sigma}\fff(x)\big|&= \Big|\int_{(\rn)^m}{\sigma^{\vee}(\vv)\prod_{k=1}^{m}f_k(x-v_k)}d\vv \Big|\\
&\leq 2^{-jmn/t}\bigg[ \int_{(\rn)^m}{\Big( \prod_{k=1}^{m}\langle 2^jv_k\rangle^{s_kt'}\Big)|\sigma^{\vee}(\vv)|^{t'}}d\vv \bigg]^{1/t'}\prod_{k=1}^{m}\mathfrak{M}_{s_k,2^j}^{t}f_k(x)
\end{align*}
where we applied the simple inequality that
\begin{equation*}
\big\Vert f(x-\cdot)\langle 2^j\cdot\rangle^{-s_k}\big\Vert_{L^t(\rn)}\lesssim 2^{-jn/t}\mathfrak{M}_{s_k,2^j}^tf(x).
\end{equation*}
Then the Hausdorff Young inequality with $1<t\leq 2$ yields that
\begin{equation*}
\bigg[ \int_{(\rn)^m}{\Big( \prod_{k=1}^{m}\langle 2^jv_k\rangle^{s_kt'}\Big)|\sigma^{\vee}(\vv)|^{t'}}d\vv \bigg]^{1/t'}\lesssim 2^{jmn/t} \big\Vert \sigma(2^j\cdot )\big\Vert_{L_{\sss}^t((\rn)^m)}
\end{equation*} and this completes the proof.
\end{proof}

The next lemma is a multi-parameter inequality of Kato-Ponce type. 
\begin{lemma}\label{katoponce}
Let $1< t<\infty$ and $s_1,\dots,s_m\geq 0$. 
Suppose that $g$ is a function in $L_{\sss}^{t}((\rn)^m)$ and $\Xi\in \SS((\rn)^m)$. Then we have
\begin{equation}\label{Katoestimate}
\big\Vert \Xi \cdot g \big\Vert_{L_{\sss}^{t}((\rn)^m)}\lesssim_{\Xi} \Vert g \Vert_{L_{\sss}^{t}((\rn)^m)}.
\end{equation}
\end{lemma}
The above lemma is clear when $s_1,\dots,s_m$ are even integers as the derivatives of $\Xi$ are bounded functions, using the embedding $L^{t}_{\sss^{(1)}}\hookrightarrow L_{\sss^{(2)}}^{t}$ for $\sss^{(2)}:=(s_1^{(2)},\dots,s_m^{(2)}) \leq \sss^{(1)}:=(s_1^{(1)},\dots,s_m^{(1)})$, which means $s_k^{(2)}\leq s_k^{(1)}$ for each $1\leq k\leq m$. 
 To be specific,
$$\big\Vert \Xi \cdot g\Vert_{L_{\sss}^t((\rn)^m)}\lesssim \Big(\max_{1\le j\le m}\max_{1\le l_j\le s_j/2} \big\Vert (I-\Delta_{1})^{l_1}\cdots(I-\Delta_m)^{l_m}\Xi \; \big\Vert_{L^{\infty}((\rn)^m)} \Big)\Vert g\Vert_{L^t_{\sss}((\rn)^m)}$$
for even integers $s_1,\dots,s_m$. 
Now suppose, by induction, that (\ref{Katoestimate}) holds for all even integers $s_1,\dots,s_l$ and all nonnegative numbers $s_{l+1},\dots, s_m$. Let $r_1,\dots,r_{l-1}$ be even integers and $r_l,\dots,r_m$ be nonnegative numbers. Choose an integer $N_l$ with $2N_l\le r_l<2(N_l+1)$. Then we obtain 
$$\big\Vert \Xi \cdot g \big\Vert_{L_{(r_1,\dots,r_m)}^{t}((\rn)^m)}\lesssim_{\Xi} \Vert g \Vert_{L_{(r_1,\dots,r_m)}^{t}((\rn)^m)}$$
by interpolating the two estimates (\ref{Katoestimate}) with
$$(s_1,\dots,s_m)=(r_1,\dots,r_{l-1},2N_l,r_{l+1},\dots,r_m)$$ 
and 
$$(s_1,\dots,s_m)=(r_1,\dots,r_{l-1},2(N_l+1),r_{l+1},\dots,r_m),$$ which follow from the induction hypothesis. By repeating this process $(m-1)$ times, we obtain (\ref{Katoestimate}) for any $s_1,\dots,s_m\ge 0$.
We refer to \cite[Section 5]{Gr2} for a similar interpolation technique.


We now discuss a regularization of multipliers.
\begin{lemma}\label{regularization}
Let $1<r\leq 2$ and $\sigma$ satisfy $\mathcal{L}_{\sss}^{r,\Psi^{(m)}}[\sigma]<\infty$ for $s_k>n/r$, $1\leq k\leq m$. Then there exists a family of Schwartz functions $\{\sigma^{\epsilon}\}_{0<\epsilon<1/2}$ such that $\wh{\sigma^{\epsilon}}$ has a compact support in $(\rn)^m$,
\begin{equation}\label{sigmaapprox}
\sup_{0<\epsilon<1/2}\mathcal{L}_{\sss}^{r,\Psi^{(m)}}[\sigma^{\epsilon}]\lesssim \mathcal{L}_{\sss}^{r,\Psi^{(m)}}[\sigma],
\end{equation}  and 
\begin{equation}\label{l2approx}
\lim_{\epsilon\to 0}\big\Vert T_{\sigma}\fff-T_{\sigma^{\epsilon}}\fff\big\Vert_{L^2(\rn)}=0
\end{equation}
for Schwartz functions $f_1,\dots,f_m$ on $\rn$.
\end{lemma}

The above lemma can be verified with a very similar argument as described in \cite[Theorem 3.1]{Gr_Ng}, by using Lemma \ref{katoponce} and just replacing $L_{\sss}^2$ by $L_{\sss}^r$.
Therefore, the proof will not be pursued here.
As shown in \cite{Gr_Ng}, the $L^2$ convergence in (\ref{l2approx}) implies the existence of a sequence of positive numbers $\{\epsilon_j\}_{j\in\nn}$, converging to $0$ as $j\to \infty$, such that
\begin{equation*}
\lim_{j\to \infty}{T_{\sigma^{\epsilon_j}}\fff(x)}=T_{\sigma}\fff(x) \qquad \text{ a.e. } ~x\in\rn.
\end{equation*} 
Then   Fatou's lemma and (\ref{sigmaapprox}) yield that
\begin{align*}
\big\Vert T_{\sigma}\fff\big\Vert_{L^p(\rn)}& \leq \liminf_{j\to\infty}{\big\Vert T_{\sigma^{\epsilon_j}}\fff\big\Vert_{L^p(\rn)}}\lesssim \sup_{0<\epsilon<1/2}\mathcal{L}_{\sss}^{r,\Psi^{(m)}}[\sigma^{\epsilon}]\prod_{i=1}^{m}\Vert f_i\Vert_{H^{p_i}(\rn)}\\
&\lesssim \mathcal{L}_{\sss}^{r,\Psi^{(m)}}[\sigma]\prod_{i=1}^{m}\Vert f_i\Vert_{H^{p_i}(\rn)}.
\end{align*}
In view of this reduction,  in the proof of the main theorem we may actually assume that $\sigma$ is a Schwartz function such that $\wh{\sigma}$ has a compact support. Our estimates will depend only on 
$ \mathcal{L}_{\sss}^{r,\Psi^{(m)}}[\sigma]$ and not on other quantities related to $\sigma$.

With the regularization in Lemma \ref{regularization}, we may apply the following lemma in the case that for at least one $i$ with $1\leq i\leq m$ we have $p_i \le 1$,  so that the $H^{p_i}$-atomic decomposition is applied.
\begin{lemma}[\cite{Gr_Mi_Ng_Tom}]\label{switchsum}
  Let $1\leq l\leq m$, $0<p_1, \dots, p_l\le 1$, and $1<p_{l+1}, \dots, p_m\le \infty$.
Let $\sigma$ be a Schwartz function on $(\rn)^m$ whose Fourier transform has compact support (as   $\sigma^{\epsilon}$ does in Lemma \ref{regularization}).
Suppose that $f_i\in H^{p_i}(\rn)$, $1\le i\le l$, have atomic representations
$f_i = \sum_{k_i=1}^{\infty}\lambda_{i,k_i}a_{i,k_i},$
where $a_{i,k_i}$ are $L^{\infty}$-atoms for $H^{p_i}$ and
$\big( \sum_{k_i=1}^{\infty}{|{\lambda_{i,k_i}}|}^{p_i}\big)^{1/p_i}\leq \Vert{f_i}\Vert_{H^{p_i}(\rn)}$.
Suppose $f_i \in \SS(\rn)$ for $l + 1\le i \le m$.
Then
  \begin{equation*}
   T_{\sigma}\fff(x) = \sum_{k_1=1}^{\infty}\cdots\sum_{k_l=1}^{\infty}\lambda_{1,k_1}\cdots\lambda_{l,k_l}   T_{\sigma}\big(a_{1,k_1},\ldots,a_{l,k_l},f_{l+1},\ldots,f_m\big) (x)
   \end{equation*} for almost all $x\in\rn$.
\end{lemma}

In order to establish an inequality such as (\ref{keypointest}), the vanishing moment condition of $a_{i,k_i}$ will be exploited in the following way.
\begin{lemma}\label{smalllemma}
Suppose that $a\in L^{\infty}_{0}(\rn)$ is a bounded function with a compact support and has   vanishing moments in the sense that there is a $M\in \nn\cup \{0\}$ such that 
\begin{equation}\label{momentcondition}
\int_{\rn}{x^{\alpha}a(x)}dx=0, \quad |\alpha|\leq M.
\end{equation}
Then for any $K\in \SS(\rn)$ and $c_0\in\rn$, we have 
\begin{equation}\label{smallgoal}
\big|K\ast a(x)\big|\lesssim \int_0^1{\int_{\rn}{|y-c_0|^{M+1}\sum_{|\alpha|=M+1}{\big|\partial^{\alpha}K\big(x-c_0-t(y-c_0)\big) \big||a(y)|}}dy}dt.
\end{equation}
\end{lemma}
\begin{proof}
 We recall   Taylor's formula saying that for any $x,y\in\mathbb{R}^n$ and $M\in\nn\cup \{0\}$ we have 
\begin{equation*}
f(x+y)=\sum_{|\alpha|\leq M}{\frac{\partial^{\alpha}f(x)}{\alpha !}y^{\alpha}}+(M+1)\sum_{|\alpha|=M+1}{\frac{1}{\alpha !}\Big( \int_0^1{(1-t)^M\partial^{\alpha}f(x+ty)}dt\Big) y^{\alpha}}.
\end{equation*}
Then (\ref{momentcondition}) yields that the left-hand side of (\ref{smallgoal}) is dominated by a constant times
\begin{align*}
&\sum_{|\alpha|=M+1}{\frac{1}{\alpha !}\int_0^1{(1-t)^M\int_{\rn}{\big| \partial^{\alpha}K\big(x-c_0-t(y-c_0)\big)\big||y-c_0|^{M+1}|a(y)|}dy}dt} 
\end{align*} and this is clearly less than the right-hand side of (\ref{smallgoal}).
\end{proof}

The argument in Lemma \ref{smalllemma} will help us  estimate the $L^{r'}$ norm of the product of $\langle x_1\rangle^{s_1}\cdots \langle x_m\rangle^{s_m}$ and derivatives of $\big( \sigma(2^j\cdot )\wh{\Psi^{(m)}}\big)^{\vee}$ to obtain the quantity $\mathcal{L}_{\sss}^{r,\Psi^{(m)}}[\sigma]$, as the Hausdorff-Young inequality $\Vert F^{\vee}\Vert_{L^{r'}((\rn)^m)}\lesssim \Vert F\Vert_{L^r((\rn)^m)}$ is applicable for $1<r\leq 2$. The following lemma will play a significant role in this. 

\begin{lemma}[\cite{Gr_Mi_Ng_Tom,Mi_Tom}]\label{lem:LInfL2}
Let  $1\leq p\leq q\leq \infty$, and $s_k\geq 0$ for $1\leq k\leq m$.
Let ${\sigma}$ be a function defined on $(\rn)^m$ and $K={\sigma}^{\vee}$ be the inverse Fourier transform of $\sigma$.
Suppose that  $\sigma$ is supported in a ball of a constant radius.
Then for $1\le l\le m$ and any multi-index $\aaa$ in $(\zn)^l$
there exists a constant $C_{\aaa}$ such that
\begin{align*}
&\big\Vert{\langle\cdot_1\rangle^{s_{1}} \cdots\langle\cdot_l\rangle^{s_{l}}\partial^{\aaa}K(\cdot_1,\dots,\cdot_l,y_{l+1},\dots,y_m)}\big\Vert_{L^{q}((\rn)^l)}\\
&\leq C_{\aaa} \big\Vert{\langle \cdot_1\rangle^{s_{1}}\cdots\langle\cdot_l\rangle^{s_{l}}K(\cdot_1,\dots,\cdot_l,y_{l+1},\dots,y_m)}\big\Vert_{L^{p}((\rn)^l)}
\end{align*}
where $\partial^{\aaa}$ denotes $\aaa$ derivatives in the first $l$ variables.
\end{lemma}

We end this section by reviewing the technique of Grafakos and Kalton \cite{Gr_Ka}, which will be very useful in estimating the $L^p$ norm of the sum of functions having a compact support for $0<p\leq 1$.
\begin{lemma}\cite[Lemma 2.1]{Gr_Ka}\label{grkalemma}
  Let $0<p\leq 1$ and $\{f_Q\}_{Q\in \mathcal{J}}$ be a family of nonnegative integrable functions with $\textup{Supp}(f_Q)\subset Q$ for all $Q\in\mathcal{J}$, where $\mathcal{J}$ is a finite or countable family of cubes in $\rn$.
Then we have
 \begin{equation*}
\Big\Vert{\sum_{Q\in\mathcal{J}} f_Q}\Big\Vert_{L^p(\rn)}\lesssim \Big\Vert{\sum_{Q\in\mathcal{J}} \Big(\frac1{{\left\vert{Q}\right\vert}}\int_Q f_Q(y) dy\Big)\chi_{Q}}\Big\Vert_{L^p(\rn)},
\end{equation*}
  where the constant in the inequality depends only on $p$.
\end{lemma}

\section{Proof of Proposition \ref{main21}}\label{proofpropo1}

Let $\theta$ and $\wt{\theta}$ denote Schwartz functions on $\rn$ having the properties   
\begin{align*}
&\textup{Supp}(\wh{\theta})\subset \big\{\xi\in\rn: \tfrac{1}{2000\sqrt{m}}\leq |\xi|\leq \tfrac{1}{1000\sqrt{m}}\big\}  \\ 
& \textup{Supp}(\wh{\wt{\theta}})\subset \big\{\xi\in\rn:  |\xi|\leq \tfrac{1}{100\sqrt{m}}\big\} &
\end{align*}  
and $ \wh{\wt{\theta}}(\xi)=1 $ for $ |\xi|\leq \frac{1}{1000\sqrt{m}}.$
Then it is clear that
\begin{equation}
\big(\theta(\epsilon \;\cdot)\big)\ast \wt{\theta}(x)=\theta(\epsilon x)\qquad  \text{ for any }~ 0<\epsilon\le 1.
\end{equation}
Let $e_1:=(1,0,\dots,0)\in \zn$.
Choose $2/r<\delta\leq 2$ and let $N>0$ be a sufficiently large number to be chosen later.
Recall that our fixed Schwartz function $\phi_j$ satisfies $\textup{Supp}(\wh{\phi_j})\subset \{\xi\in\rn: |\xi|\leq 2^{j+1}\}$ and $\wh{\phi_j}(\xi)=1$ for $|\xi|\leq 2^j$.
Let $\Phi$ be a Schwartz function on $\rn$ whose Fourier transform is equal to $1$ on the ball $\{\xi\in \rn: |\xi|\le 1\}$ and is supported in a larger ball.
Let $N$ be a sufficiently large positive integer and $\Phi_N:=N^{n}\Phi(N \cdot)$.

We define
\begin{equation*}
\mathcal{H}_{(n,\delta)}^{(N)}(x):=\mathcal{H}_{(n,\delta)}(x)\wh{\Phi_N}(x), \qquad x\in\rn
\end{equation*}
and
\begin{equation*}
\sigma^{(N)}(\xxxi):=\wh{\mathcal{H}_{(n,\delta)}^{(N)}}(\xi_1-e_1)\wh{\wt{\theta}}(\xi_1-e_1)\wh{\wt{\theta}}(\xi_2)\cdots \wh{\wt{\theta}}(\xi_m), \qquad \xxxi\in (\rn)^m ,
\end{equation*}
where $\mathcal{H}_{(n,\delta)}$ is defined in (\ref{hdefinition}).

It follows from the support of $\wh{\wt{\theta}}$ that $\sigma^{(N)}$ is supported in $\{\xxxi\in (\rn)^m: \frac{99}{100} \leq |\xxxi| \leq\frac{101}{100}\}$, which implies that $\sigma^{(N)}(2^l\xxxi)\wh{\Psi^{(m)}}(\xxxi)$ vanishes unless $-1\leq l\leq 1$.
Moreover, in view  of Lemma \ref{katoponce} we have
\begin{equation*}
\mathcal{L}_{\sss}^{r,\Psi^{(m)}}[\sigma^{(N)}]\lesssim \max_{-1\leq l\leq 1}{\big\Vert \sigma^{(N)}(2^l\cdot_1,\dots,2^l\cdot_m)\big\Vert_{L_{\sss}^{r}((\rn)^m)}},
\end{equation*} 
which can be estimated, via scaling, by a constant times
\begin{equation}\label{scalingargument}
\big\Vert \sigma^{(N)} \big\Vert_{L_{\sss}^r((\rn)^m)}\lesssim \Big\Vert \wh{\mathcal{H}_{(n,\delta)}^{(N)}}\cdot\wh{\wt{\theta}} \;\Big\Vert_{L_{s_1}^{r}(\rn)}\lesssim \big\Vert \wh{\mathcal{H}_{(n,\delta)}^{(N)}}\big\Vert_{L_{s_1}^{r}(\rn)}, 
\end{equation}  
where we used Lemma \ref{katoponce}  in the last inequality.
We observe that 
\begin{equation*}
(I-\Delta)^{s_1/2}\wh{\mathcal{H}_{(n,\delta)}^{(N)}}(\xi)=\wh{\mathcal{H}_{(n-s_1,\delta)}^{(N)}}(\xi)=\Phi_N\ast \wh{\mathcal{H}_{(n-s_1,\delta)}}(\xi)
\end{equation*}
and  $\wh{\mathcal{H}_{(n-s_1),\delta}}\in L^r(\rn)$, using (\ref{hproperty4}) with $\delta >2/r$ and $s_1=n/r$.
Since $\{\Phi_N\}_{N\in\nn}$ is an approximate identity, we have
\begin{equation*}
\lim_{N\to \infty}\Big\Vert \Phi_N\ast \wh{\mathcal{H}_{(n-s_1,\delta)}}-\wh{\mathcal{H}_{(n-s_1,\delta)}}\Big\Vert_{L^r(\rn)}=0,
\end{equation*} which proves
\begin{equation}\label{lsigmabound}
\limsup_{N\to\infty}\mathcal{L}_{\sss}^{r,\Psi^{(m)}}[\sigma^{(N)}]\lesssim \lim_{N\to \infty}\Big\Vert \Phi_N\ast \wh{\mathcal{H}_{(n-s_1,\delta)}}\Big\Vert_{L^r(\rn)}\leq \big\Vert \wh{\mathcal{H}_{(n-s_1,\delta)}} \big\Vert_{L^r(\rn)}<\infty.
\end{equation}

On the other hand, for $0<\epsilon<1/100$, let
\begin{equation*}
f_1^{(\epsilon)}(x):= \epsilon^{n/p_1}\theta(\epsilon x)e^{2\pi i\langle x,e_1\rangle}, \qquad         f^{(\epsilon)}_j(x):=\epsilon^{n/p_j}{\theta}(\epsilon x), ~~2\leq j\leq m.
\end{equation*} 
Then it is clear, from the Littlewood-Paley theory for Hardy spaces and scaling arguments  that for each $1\leq j\leq m$
\begin{equation}\label{secondest}
\big\Vert f_j^{(\epsilon)}\big\Vert_{H^{p_j}(\rn)}\approx \big\Vert f_j^{(\epsilon)}\big\Vert_{L^{p_j}(\rn)}=\Vert \theta\Vert_{L^{p_j}(\rn)}\lesssim 1,\qquad \text{ uniformly in }~~\epsilon.
\end{equation}
Moreover, we observe that
\begin{equation*}
\Big| T_{\sigma^{(N)}}\big(f_1^{(\epsilon)},\dots,f_m^{(\epsilon)}\big)(x) \Big|=\epsilon^{n/p}\Big|\mathcal{H}_{(n,\delta)}^{(N)}\ast \big( \theta(\epsilon\cdot)\big)(x) \Big| \big|\theta(\epsilon x)\big|^{m-1}
\end{equation*} and this, together with scaling,  yields that
\begin{align*}
\Big\Vert T_{\sigma^{(N)}}\big(f_1^{(\epsilon)},\dots,f_m^{(\epsilon)} \big)\Big\Vert_{L^p(\rn)}=\Big\Vert \big|\theta\big|^{m-1}\Big( \mathcal{H}_{(n,\delta)}^{(N)}\ast\big(\theta(\epsilon\cdot)\big)\Big)(\cdot/\epsilon)\Big\Vert_{L^p(\rn)}.
\end{align*}
Applying (\ref{secondest}) and   Fatou's lemma, we obtain that
\begin{align}
\big\Vert T_{\sigma^{(N)}}\big\Vert_{H^{p_1}\times\cdots\times H^{p_m}\to L^p}&\gtrsim \liminf_{\epsilon\to 0} \Big\Vert T_{\sigma^{(N)}}\big(f_1^{(\epsilon)},\dots,f_m^{(\epsilon)} \big)\Big\Vert_{L^p(\rn)}\nonumber\\
&\geq \Big\Vert \big| \theta\big|^{m-1}\Big| \liminf_{\epsilon\to 0}\int_{\rn}{\theta(x-\epsilon y)\mathcal{H}_{(n,\delta)}^{(N)}(y)}dy\Big|\Big\Vert_{L^p(\rn)}.\label{continueest}
\end{align}
Since 
\begin{equation*}
\big| \theta(x-\epsilon y)\mathcal{H}_{(n,\delta)}^{(N)}(y)\big|\lesssim \mathcal{H}_{(n,\delta)}^{(N)}(y) \quad \text{ uniformly in }~\epsilon>0, x\in\rn
\end{equation*} and 
\begin{equation*}
\big\Vert \mathcal{H}_{(n,\delta)}^{(N)}\big\Vert_{L^1(\rn)}\leq \big\Vert \wh{\Phi_N}\big\Vert_{L^1(\rn)}\lesssim N^n<\infty,
\end{equation*}
the Lebesgue dominated convergence theorem yields 
\begin{equation*}
(\ref{continueest})=\big\Vert | \theta |^m\big\Vert_{L^p(\rn)}\int_{\rn}{\mathcal{H}_{(n,\delta)}^{(N)}(y)}dy\approx \int_{\rn}{\mathcal{H}_{(n,\delta)}(y)\wh{\Phi_N}(y)}dy.
\end{equation*}
Taking $\liminf_{N\to \infty}$, we finally obtain that
\begin{equation*}
\liminf_{N\to\infty}\big\Vert T_{\sigma^{(N)}}\big\Vert_{H^{p_1}\times\cdots\times H^{p_m}\to L^p}\gtrsim \big\Vert \mathcal{H}_{(n,\delta)}\big\Vert_{L^1(\rn)}=\infty
\end{equation*} where we applied the monotone convergence theorem and the fact that $\mathcal{H}_{(n,\delta)}\not\in L^1(\rn)$ for $\delta\leq 2$ because of (\ref{hproperty2}).

This fact combined with (\ref{lsigmabound})  completes the proof.

\section{Proof of Proposition \ref{main22}}

We first consider the case $1\leq l<m$.
Let $\mu_1:=(m^{-1/2},0,\dots,0)\in\rn$.
The condition 
\begin{equation*}
\sum_{k=1}^{l}{\big({s_k}/{n}-{1}/{p_k}\big)}\leq- {1}/{r'}
\end{equation*} is equivalent to
\begin{equation}\label{equivcondition}
s_1+\dots +s_l+n/r' \leq n/p_1+\dots+n/p_l=n/p-\big(n/p_{l+1}+\dots+n/p_m).
\end{equation}
On the other hand, it follows from the condition $s_j>n/r$, $1\leq j\leq m$, that
\begin{equation*}
s_1+\dots+s_l+n/r'>ln/r+n/r',
\end{equation*} which further implies, combined with (\ref{equivcondition}), that
\begin{equation*}
2l/r+2/r'<2/p-\big(2/p_{l+1}+\dots+2/p_m \big).
\end{equation*}
Now we choose $\tau, \tau_{l+1},\dots, \tau_m>0$ such that
\begin{equation*}
\tau_{l+1}>2/p_{l+1},\dots,\tau_m>2/p_m
\end{equation*}
and
\begin{equation}\label{taucondition}
2/r<\tau<2l/r+2/r'<2/p-(\tau_{l+1}+\dots+\tau_m)<2/p-(2/p_{l+1}+\dots+2/p_m).
\end{equation}

Let $\varphi,\wt{\varphi}\in S(\rn)$ be radial functions having the properties that $\varphi\geq 0$, $\varphi(0)\not= 0$, $\textup{Supp}(\wh{\varphi})\subset \{\xi\in\rn: |\xi|\leq \frac{1}{200lm}\}$, $\textup{Supp}(\wh{\wt{\varphi}})\subset \{\xi\in\rn: |\xi|\leq \frac{1}{100m}\}$, and $\wh{\wt{\varphi}}(\xi)=1$ for $|\xi|\leq \frac{1}{200m}$.
In what follows, we denote $\mathcal{H}_{(s_1+\dots+s_l+n/r',\tau)}$ by $\mathcal{H}$ for notational convenience.
We define
\begin{equation*}
K^{(l)}(x):=\mathcal{H}\ast \varphi(x), \quad x\in\rn,
\end{equation*}
and
\begin{equation*}
M^{(l)}(\xi_1,\dots,\xi_l):={(K^{(l)})}^{\vee}\Big(\frac{1}{l}\sum_{k=1}^{l}{(\xi_k-\mu_1)} \Big)  \prod_{j=2}^{l} {\varphi}^{\vee}\Big(\frac{1}{l}\sum_{k=1}^{l}{(\xi_k-\xi_j)} \Big)
\end{equation*}
where $M^{(l)}$ is defined on $(\rn)^l$.
Then the multiplier $\sigma$ on $(\rn)^m$ is defined by
\begin{equation*}
\sigma(\xi_1,\dots,\xi_m):=M^{(l)}(\xi_1,\dots,\xi_l){\wt{\varphi}}^{\vee}(\xi_{l+1}-\mu_1)\cdots {\wt{\varphi}}^{\vee}(\xi_{m}-\mu_1).
\end{equation*}
To investigate the support of $\sigma$ we first look at the support of $M^{(l)}$.
From the support of ${\varphi}^{\vee}$, we have
\begin{equation*}
\big| \xi_1+\dots+\xi_l-l\mu_1\big|\leq \frac{1}{200m},
\end{equation*} and for each $2\leq j\leq l$
\begin{equation}\label{jest}
\big| \xi_1+\dots+\xi_l-l\xi_j\big|\leq \frac{1}{200m}.
\end{equation}
By adding up all of them, we obtain 
\begin{equation}\label{1est}
\big| \xi_1-\mu_1\big|\leq \frac{1}{200m}
\end{equation}
and the sum of (\ref{jest}) and (\ref{1est}) yields that for each $2\leq j\leq l$
\begin{equation*}
\big| \mu_1+\xi_2+\dots+\xi_l-l\xi_j\big|\leq \frac{1}{100m}.
\end{equation*}
Let us call the above estimate $\mathcal{E}(j)$.
Then for  $2\leq j\leq l$, it follows from  $$\mathcal{E}(j)+\sum_{k=2}^{l}\mathcal{E}(k)$$
that 
\begin{equation*}
\big|\xi_j-\mu_1\big|\leq \frac{1}{100m},
\end{equation*} which proves, together with (\ref{1est}), that 
\begin{align}\label{supportm}
 \textup{Supp}(M^{(l)})&\subset \Big\{ (\xi_1,\dots,\xi_l)\in (\rn)^l:   |\xi_j-\mu_1|\leq \frac{1}{100m}, ~ 1\leq j\leq l \Big\}.
\end{align}
Since $\wh{\wt{\varphi}}$ is also supported in $\{\xi\in\rn: |\xi|\leq \frac{1}{100{m}}\}$, it is clear that 
\begin{align*}
\textup{Supp}(\sigma)&\subset  \Big\{ (\xi_1,\dots,\xi_m)\in (\rn)^m:   |\xi_j-\mu_1|\leq \frac{1}{100m}, ~ 1\leq j\leq m \Big\}\\
&\subset \Big\{\xxxi:=(\xi_1,\dots,\xi_m)\in (\rn)^m:    \frac{99}{100}\leq |\xxxi|\leq \frac{101}{100}     \Big\},
\end{align*} 
 which shows that $\sigma(2^l\xxxi )\wh{\Psi^{(m)}}(\xxxi)$ vanishes unless $-1\leq l\leq 1$.
Furthermore, using Lemma \ref{katoponce} and the scaling argument in 
the derivation of (\ref{scalingargument}), we have
\begin{equation*}
\mathcal{L}_{\sss}^{r,\Psi^{(m)}}[\sigma]\lesssim     \sup_{-1\leq l\leq 1}{\Big\Vert \sigma(2^l\cdot)\wh{\Psi^{(m)}}\Big\Vert_{L_{\sss}^{r}((\rn)^m)}}  \lesssim \Vert \sigma\Vert_{L_{\sss}^{r}((\rn)^m)}
\end{equation*}
and this is clearly less than a constant times
\begin{equation*}
 \big\Vert M^{(l)}\big\Vert_{L_{(s_1,\dots,s_l)}^{r}((\rn)^l)}\prod_{j=l+1}^{m}{\big\Vert {\wt{\varphi}}^{\vee}\big\Vert_{L_{s_j}^{r}(\rn)}}\lesssim \big\Vert (I-\Delta_1)^{s_1/2}\cdots (I-\Delta_l)^{s_l/2}M^{(l)}\big\Vert_{L^{r}((\rn)^l)}.
\end{equation*}
We observe that 
\begin{align*}
&\wh{M^{(l)}}(x_1,\dots,x_l)\\
&=\int_{(\rn)^l}{ ({K^{(l)}})^{\vee}\Big(\frac{1}{l}\sum_{k=1}^{l}{(\xi_k-\mu_1)} \Big)\Big[ \prod_{j=2}^{l} {\varphi}^{\vee}\Big(\frac{1}{l}\sum_{k=1}^{l}{(\xi_k-\xi_j)} \Big)\Big] \Big(\prod_{j=1}^{l}e^{-2\pi i\langle x_j,\xi_j\rangle}\Big)}d\xi_1\cdots d\xi_l.
\end{align*}
Using a change of variables with
\begin{equation*}
\zeta_1:=\frac{1}{l}\sum_{k=1}^{l}{(\xi_k-\mu_1)}, \qquad \text{ and }\qquad \zeta_j:=\frac{1}{l}\sum_{k=1}^{l}{(\xi_k-\xi_j)}, \quad 2\leq j\leq l
\end{equation*}
so that
\begin{equation}\label{system}
\xi_1=\zeta_1+\dots+\zeta_l+\mu_1, \qquad \text{ and }\qquad \xi_j=\zeta_1-\zeta_j+\mu_1, \quad 2\leq j\leq l,
\end{equation}
we see that
\begin{align}\label{mlexpression}
\wh{M^{(l)}}(x_1,\dots,x_l)&= le^{-2\pi i\langle \sum_{k=1}^{l}x_k,\mu_1\rangle}\int_{(\rn)^l}{({K^{(l)}})^{\vee}(\zeta_1)\Big(\prod_{j=2}^{l}{{\varphi}^{\vee}(\zeta_j)}\Big)e^{-2\pi i\langle \sum_{k=1}^{l}x_k,\zeta_1\rangle}}\nonumber \\
&\qquad \qquad \qquad \qquad \qquad \qquad \quad \quad \times {\Big(\prod_{j=2}^{l}{e^{-2\pi i\langle x_1-x_j,\zeta_j\rangle} }\Big)    }d\zeta_1\cdots d\zeta_l\nonumber\\
&=l K^{(l)}(x_1+\dots+x_l)\varphi(x_1-x_2)\cdots \varphi(x_1-x_l)e^{-2\pi i\langle x_1+\dots+x_l,\mu_1\rangle}
\end{align} since the Jacobian of the system (\ref{system}) is $l$.
Consequently,
\begin{align*}
&(I-\Delta_1)^{s_1/2}\cdots (I-\Delta_l)^{s_l/2}M^{(l)}(\xi_1,\dots,\xi_l)\\
&=l\int_{(\rn)^l}{ \Big( \prod_{j=1}^{l}\langle x_j \rangle^{s_j}\Big)    K^{(l)}(x_1+\dots+x_l)\varphi(x_1-x_2)\cdots \varphi(x_1-x_l)}\\
&\qquad \qquad \qquad \qquad \qquad {e^{-2\pi i\langle x_1+\dots+x_l,\mu_1\rangle}  e^{2\pi i\langle x_1,\xi_1\rangle}\cdots e^{2\pi i\langle x_l,\xi_l\rangle}  }dx_1\cdots dx_l
\end{align*}
and we perform another change of variables with
\begin{equation*}
y_1:=x_1+\dots+x_l, \qquad \text{ and }\qquad y_j:=x_1-x_j, \quad 2\leq j\leq l,
\end{equation*}
which is equivalent to 
\begin{equation*}
x_1=\frac{1}{l}\sum_{k=1}^{l}{y_k}, \qquad \text{ and }\qquad  x_j=\frac{1}{l}\sum_{k=1}^{l}{(y_k-y_j)}, \quad 2\leq j\leq l,
\end{equation*}
to obtain that the last expression is controlled by a constant times
\begin{align*}
&\int_{(\rn)^l}{\Big\langle \frac{1}{l}\sum_{k=1}^{l}y_k \Big\rangle^{s_1} \Big( \prod_{j=2}^{l}\Big\langle \frac{1}{l}\sum_{k=1}^{l}{(y_k-y_j)} \Big\rangle^{s_j}\Big) K^{(l)}(y_1)\Big( \prod_{j=2}^{l}\varphi(y_j)\Big)}\\
& \qquad \qquad \qquad \times e^{2\pi i\langle y_1,\frac{1}{l}(\xi_1+\dots+\xi_l)-\mu_1\rangle }\Big(\prod_{j=2}^{l}e^{2\pi i\langle y_j,\frac{1}{l}(\xi_1+\dots+\xi_l)-\xi_j\rangle} \Big)       dy_1\cdots dy_l.
\end{align*}
In conclusion, using a change of variables, we have
\begin{align}\label{conclude}
\mathcal{L}_{\sss}^{r,\Psi^{(m)}}[\sigma]&\lesssim \Big\Vert \int_{(\rn)^l}{\Big\langle \frac{1}{l}\sum_{k=1}^{l}y_k \Big\rangle^{s_1} \Big( \prod_{j=2}^{l}\Big\langle \frac{1}{l}\sum_{k=1}^{l}{(y_k-y_j)} \Big\rangle^{s_j}\Big)}\\
& \qquad \qquad \quad \times K^{(l)}(y_1)\Big( \prod_{j=2}^{l}\varphi(y_j)\Big)\Big(\prod_{j=1}^{l}e^{2\pi i\langle y_j,\xi_j\rangle} \Big)       dy_1\cdots dy_l\Big\Vert_{L^r( \xi_1,\dots,\xi_m )}.\nonumber
\end{align}
For sufficiently large $M>0$, let
\begin{equation*}
 \mathcal{N}_{(M)}(y_1,\dots,y_l):=\dfrac{ \Big\langle \frac{1}{l}\sum_{k=1}^{l}y_k \Big\rangle^{s_1} \prod_{j=2}^{l}\Big\langle \frac{1}{l}\sum_{k=1}^{l}{(y_k-y_j)} \Big\rangle^{s_j}}{\langle y_1\rangle^{s_1+\dots+s_l}\prod_{j=2}^{l}{\langle y_j\rangle^{M}} }.
\end{equation*}
Then the right-hand side of  (\ref{conclude}) can be written as
\begin{equation}\label{mainterm}
\Big\Vert T_{\mathcal{N}_{(M)}}\Big((K^{(l)}_{s_1+\dots+s_l})^{\vee}\otimes (\varphi^{(M)})^{\vee}\otimes \cdots\otimes(\varphi^{(M)})^{\vee}\Big) \Big\Vert_{L^r((\rn)^l)}
\end{equation}
where
\begin{equation*}
K_{(s_1+\dots+s_l)}^{(l)}(y_1):=\langle y_1\rangle^{s_1+\dots+s_l}K^{(l)}(y_1), \qquad \varphi^{(M)}(y):=\langle y\rangle^{M}\varphi(y).
\end{equation*} 
Now we need the following lemma whose proof will be provided in Section \ref{prooflemmas}.
\begin{lemma}\label{nmmultiplier}
Let $M>s_1+\dots+s_l+n+2$. Then $\mathcal{N}_{(M)}$ is an $L^r$ multiplier on $(\rn)^l$.
\end{lemma}

 By choosing $M>s_1+\dots+s_l+n+2$ and using Lemma \ref{nmmultiplier} and \ref{katoponce}, we obtain
\begin{align*}
(\ref{mainterm})&\lesssim \big\Vert (K^{(l)}_{(s_1+\dots+s_l)})^{\vee}\otimes (\varphi^{(M)})^{\vee}\otimes \cdots\otimes(\varphi^{(M)})^{\vee}\big\Vert_{L^r((\rn)^l)}\\
&\lesssim \big\Vert (I-\Delta)^{(s_1+\dots+s_l)/2}(K^{(l)})^{\vee}\big\Vert_{L^r(\rn)}= \big\Vert (I-\Delta)^{(s_1+\dots+s_l)/2}\big( \mathcal{H}^{\vee}\varphi^{\vee}\big)\big\Vert_{L^r(\rn)}\\
&\lesssim \big\Vert (I-\Delta)^{(s_1+\dots+s_l)/2}\mathcal{H}_{(s_1+\dots+s_l+n/r',\tau)}^{\vee}\big\Vert_{L^r(\rn)}=\big\Vert\wh{ \mathcal{H}_{(n/r',\tau)}}\big\Vert_{L^r(\rn)}
\end{align*}
and this is finite because of (\ref{hproperty4}) with $\tau>2/r$, which concludes that
\begin{equation*}
\mathcal{L}_{\sss}^{r,\Psi^{(m)}}[\sigma]<\infty.
\end{equation*}

To achieve 
\begin{equation}\label{achieve}
\Vert T_{\sigma}\Vert_{H^{p_1}\times \dots\times H^{p_m}\to L^p}=\infty,
\end{equation}
let
\begin{equation*}
f_1(x)=\dots=f_l(x):=2^n\wt{\varphi}(2x)e^{2\pi i\langle x,\mu_1\rangle},
\end{equation*}
\begin{equation*}
f_j(x):=\mathcal{H}_{(n/p_j,\tau_j)}\ast \varphi(x)e^{2\pi i\langle x,\mu_1\rangle}, \quad l+1\leq j\leq m.
\end{equation*}
Clearly, $\Vert f_j\Vert_{H^{p_j}(\rn)}\lesssim 1$ for $1\leq j\leq l$ and
\begin{equation*}
\Vert f_j\Vert_{H^{p_j}(\rn)}\approx \Vert f_j\Vert_{L^{p_j}(\rn)}\lesssim \big\Vert \mathcal{H}_{(n/p_j,\tau_j)}\big\Vert_{L^{p_j}(\rn)}\lesssim 1, \quad l+1\leq j\leq m
\end{equation*} due to (\ref{hproperty2}) with $\tau_j>2/p_j$, where the pointwise estimate $\mathcal{H}_{(n/p_j,\tau_j)}\ast \varphi(x)\lesssim \mathcal{H}_{(n/p_j,\tau_j)}(x)$ is applied.
On the other hand, using (\ref{supportm}) and the facts that $\varphi\ast \wt{\varphi}=\varphi$ and
\begin{equation*}
\wh{f_j}(\xi)=1 \qquad \text{ for }~~ |\xi-\mu_1|\leq \frac{1}{100m} ~ \text{ and }~  ~ 1\leq j\leq l,
\end{equation*}
we see that
\begin{equation*}
\sigma(\xi_1,\dots,\xi_m)\wh{f_1}(\xi_1)\cdots\wh{f_m}(\xi_m)=M^{(l)}(\xi_1,\dots,\xi_l)\wh{f_{l+1}}(\xi_{l+1})\cdots\wh{f_m}(\xi_{m}),
\end{equation*} which implies that
\begin{align*}
\big| T_{\sigma}\fff (x)\big|&= \big| \big( M^{(l)}\big)^{\vee}(x,\dots,x)\big| \big|f_{l+1}(x)\big|\cdots \big|f_{m}(x)\big|\\
&=l \big| K^{(l)}(lx)\big|\big| \varphi(0)\big|^{l-1} \prod_{j=l+1}^{m}\big|\mathcal{H}_{(n/p_{j},\tau_{j})}\ast \varphi(x)\big|
\end{align*} where we applied (\ref{mlexpression}) and the fact that $K^{(l)}$ is a radial function.
 Now, since 
 \begin{equation*}
 \mathcal{H}_{(s,\gamma)}\ast \varphi(x)\gtrsim \mathcal{H}_{(s,\gamma)}(x) \quad \text{ for any }~ s,\gamma>0,
 \end{equation*} which follows from the fact that $\varphi, \mathcal{H}_{(s,\gamma)}\geq 0$ and (\ref{hproperty1}),
we obtain that
\begin{align*}
\big\Vert T_{\sigma}\fff(x)\big\Vert_{L^p(\rn)}&\gtrsim \Big\Vert \mathcal{H}(l\cdot)\prod_{j=l+1}^{m}{\mathcal{H}_{(n/p_j,\tau_j)}}\Big\Vert_{L^p(\rn)}\\
&\approx\Big\Vert \mathcal{H}_{(s_1+\dots+s_l+n/r',\tau)}\prod_{j=l+1}^{m}{\mathcal{H}_{(n/p_j,\tau_j)}}\Big\Vert_{L^p(\rn)}\\
&= \big\Vert \mathcal{H}_{(s_1+\dots+s_l+n/p_{l+1}+\dots+n/p_m+n/r',\tau+\tau_{l+1}+\dots+\tau_m)} \big\Vert_{L^p(\rn)}.   
\end{align*}
Since $s_1+\dots+s_l+n/p_{l+1}+\dots+n/p_m+n/r'\leq n/p$ due to (\ref{equivcondition}), the last expression is greater than
\begin{equation*}
\big\Vert \mathcal{H}_{(n/p,\tau+\tau_{l+1}+\dots+\tau_m)}\big\Vert_{L^p(\rn)}=\infty
\end{equation*}
because of (\ref{hproperty2}) with $\tau+\tau_{l+1}+\dots+\tau_{m}<p/2$, which follows from (\ref{taucondition}).
This completes the proof of (\ref{achieve}).

When $l=m$, exactly the   same argument is applicable with $2/r<\tau<2m/r+2/r'<{2}/{p}$, $\sigma:=M^{(m)}$, and $f_j(x):=2^d\wt{\varphi}(2x)e^{2\pi i\langle x,\mu_1\rangle}$ for $1\leq j\leq m$. Since the proof is just a repetition, we omit the details.

\section{Proof of Proposition \ref{propo1}}

Let ${\Theta^{(m)}}$ be a Schwartz function on $(\rn)^m$ such that $0\leq \wh{\Theta^{(m)}}\leq 1$, $\wh{\Theta^{(m)}}(\xxxi)=1$ for $2^{-2}m^{-1/2}\leq |\xxxi|\leq 2^2m^{1/2}$, and 
$\textup{Supp}(\wh{\Theta^{(m)}})\subset \big\{ \xxxi \in(\rn)^m: 2^{-3}m^{-1/2}\leq |\xxxi|\leq 2^{3}m^{1/2}\big\}$.
Then using the Littlewood-Paley partition of unity $\big\{2^{jmn}\Psi^{(m)}(2^j\cdot)\big\}_{j\in\mathbb{Z}}$, the  triangle inequality, and Lemma \ref{katoponce}, we first see that
$\mathcal{L}_{\sss}^{r,{\Theta^{(m)}}}[\sigma]\lesssim  \mathcal{L}_{\sss}^{r,\Psi^{(m)}}[\sigma].$
Thus it suffices to prove the estimate
\begin{equation*}
\big\Vert T_{\sigma}\fff \big\Vert_{L^p(\rn)}\lesssim \mathcal{L}_{\sss}^{r,{\Theta^{(m)}}}[\sigma] \prod_{i=1}^{m}\Vert f_i\Vert_{L^{p_i}(\rn)}
\end{equation*}
as $L^{p_i}=H^{p_i}$ for $1<p_i\leq \infty$.
We adopt the notation $\mathcal{L}_{\sss}^r[\sigma]:=  \mathcal{L}_{\sss}^{r,{\Theta^{(m)}}}[\sigma] $ for simplicity.

It follows from (\ref{conditionsss}) that there exists $1<t<r$ such that
\begin{equation*}
s_1,\dots,s_m>n/t>n/r.
\end{equation*}
Since $\sigma(2^j\vec{\; \cdot\;} )\wh{\Theta^{(m)}}$ has a compact support in an annulus of a constant size, independent of $j$, we have
\begin{equation}\label{compactembedding1}
\mathcal{L}_{\sss}^{t}[\sigma]\lesssim \mathcal{L}_{\sss}^{r}[\sigma]
\end{equation} as $1<t<r$. See \cite[Section 5]{Gr2} for more details.

Using the Littlewood-Paley partition of unity $\{\psi_j\}_{j\in\mathbb{Z}}$,  we decompose $\sigma(\xxxi)$ as
\begin{align}\label{sigmadecompo}
\sigma(\xxxi)&=\sum_{j_1,\dots,j_m \in \mathbb{Z}}{\sigma(\xxxi)\widehat{\psi_{j_1}}(\xi_1)\cdots \widehat{\psi_{j_m}}(\xi_m)}\\
  &=\Big(\sum_{j_1\in\mathbb{Z}}\sum_{j_2,\dots,j_m\leq j_1}{\cdots}\Big)+\Big(\sum_{j_2\in\mathbb{Z}}\sum_{\substack{j_1<j_2\\j_3,\dots,j_m\leq j_2}}{\cdots}\Big)+\dots+ \Big(\sum_{j_m\in\mathbb{Z}}\sum_{j_1,\dots,j_{m-1}<j_m}{\cdots}\Big)\nonumber\\
  &=:\sigma^{(1)}(\xxxi)+\sigma^{(2)}(\xxxi)+\dots+\sigma^{(m)}(\xxxi).\label{sigmakappa}
\end{align}
We are only concerned with $\sigma^{(1)}$ appealing to symmetry  for the other cases. 
Thus, our goal is to show that 
\begin{equation*}
\big\Vert T_{\sigma^{(1)}}\fff \big\Vert_{L^p(\rn)}\lesssim \mathcal{L}_{\sss}^{r}[\sigma] \prod_{i=1}^{m}\Vert f_i\Vert_{L^{p_i}(\rn)}.
\end{equation*}

We write \begin{align*}
\sigma^{(1)}(\xxxi)&=\sum_{j\in\mathbb{Z}}\sum_{j_2,\dots,j_m\leq j}{\sigma(\xxxi)\widehat{\psi_j}(\xi_1)\widehat{\psi_{j_2}}(\xi_2)\cdots\widehat{\psi_{j_m}}(\xi_m)}\\
 &=\sum_{j\in\mathbb{Z}}{\sigma(\xxxi)\wh{\Theta^{(m)}}(\xxxi/2^j) \widehat{\psi_j}(\xi_1)\sum_{j_2,\dots,j_m\leq j}\widehat{\psi_{j_2}}(\xi_2)\cdots\widehat{\psi_{j_m}}(\xi_m)}, 
\end{align*} 
since $\wh{\Theta^{(m)}}(\xxxi/2^j)=1$ for $ 2^{j-1}\leq |\xi_1|\leq 2^{j+1}$ and $|\xi_i|\leq 2^{j+1}$ for $2\leq i\leq m$.
Let 
\begin{equation*}
\sigma_j(\xxxi):=\sigma(\xxxi)\wh{\Theta^{(m)}}(\xxxi/2^j). 
\end{equation*}
Then we note that
\begin{equation}\label{mkbound}
\big\Vert \sigma_j(2^j\cdot )\big\Vert_{L_{\sss}^t((\rn)^m)}\leq \mathcal{L}_{\sss}^t[\sigma]
\end{equation}  
and 
\begin{equation*}
\sigma^{(1)}(\xxxi)=\sum_{j\in\mathbb{Z}}{\sigma_j(\xxxi) \widehat{\psi_j}(\xi_1)\sum_{j_2,\dots,j_m\leq j}\widehat{\psi_{j_2}}(\xi_2)\cdots\widehat{\psi_{j_m}}(\xi_m)}.
\end{equation*}
We further decompose $\sigma^{(1)}$ as
\begin{equation*}
\sigma^{(1)}(\xxxi)=\sigma^{(1)}_{low}(\xxxi)+\sigma^{(1)}_{high}(\xxxi)
\end{equation*} where
\begin{equation*}
\sigma^{(1)}_{low}(\xxxi):=\sum_{j\in\mathbb{Z}}\sigma_j(\xxxi)\widehat{\psi_j}(\xi_1)\sum_{\substack{j_2,\dots,j_m\leq j\\ \max_{2\leq i\leq m}{(j_i)}\geq j-3-\lfloor \log_2{m}\rfloor}}{\widehat{\psi_{j_2}}(\xi_2)\cdots\widehat{\psi_{j_m}}(\xi_m)},
\end{equation*}
\begin{equation}\label{highdef}
\sigma^{(1)}_{high}(\xxxi):=\sum_{j\in\mathbb{Z}}\sigma_j(\xxxi)\widehat{\psi_j}(\xi_1)\sum_{j_2,\dots,j_m\leq j-4-\lfloor \log_2m\rfloor}{\widehat{\psi_{j_2}}(\xi_2)\cdots\widehat{\psi_{j_m}}(\xi_m)}.
\end{equation}
We refer to $T_{\sigma_{low}^{(1)}}$ as the low frequency part and $T_{\sigma_{high}^{(1)}}$ as the high frequency part of $T_{\sigma^{(1)}}$ $\big($due to the Fourier supports of the summands in $T_{\sigma_{low}^{(1)}}\fff$ and $T_{\sigma_{high}^{(1)}}\fff$$\big)$.

\subsection{Low frequency part}\label{lowfrequencypart}
To establish the estimate for the operator $T_{\sigma^{(1)}_{low}}$,
we first observe that
\begin{equation*}
T_{\sigma^{(1)}_{low}}\fff(x)=\sum_{j\in\mathbb{Z}}\sum_{\substack{j_2,\dots,j_m\leq j\\ \max_{2\leq i\leq m}{(j_i)}\geq j-3-\lfloor \log_2{m}\rfloor}} {T_{\sigma_j}\big( (f_1)_j,(f_2)_{j_2},\dots,(f_m)_{j_m}\big)(x)}
\end{equation*}
where $(g)_l:=\psi_l\ast g$ for $g\in \SS(\rn)$ and $l\in\zz$.
 It suffices to treat only the sum over $j_3,\dots,j_m\leq j_2$ and $j-3-\lfloor \log_2{m}\rfloor \leq j_2\leq j$, and we will actually prove that
\begin{equation}\label{goal1}
\Big\Vert \sum_{j\in\mathbb{Z}}\sum_{\substack{j-3-\lfloor \log_2{m}\rfloor \leq j_2\leq j\\ j_3,\dots,j_m\leq j_2  }}T_{\sigma_j}\big((f_1)_j,(f_2)_{j_2},\dots,(f_m)_{j_m} \big)\Big\Vert_{L^p(\rn)}\lesssim \mathcal{L}_{\sss}^r[\sigma]\prod_{i=1}^{m}\Vert f_i\Vert_{L^{p_i}(\rn)}. 
\end{equation}
Let $\phi$ be the Schwartz function defined in (\ref{phidef}) and  $\phi_j:=2^{jn}\phi(2^j\cdot)$ for $j\in\mathbb{Z}$. Let $(g)^l:=\phi_l\ast g$ for $g\in\SS(\rn)$. 
Then we can write
\begin{equation*}
\sum_{j_3,\dots,j_n\leq j_2}T_{\sigma_j}\big((f_1)_j,(f_2)_{j_2},\dots,(f_m)_{j_m} \big)(x)=T_{\sigma_j}\big((f_1)_j,(f_2)_{j_2},(f_3)^{j_2},\dots,(f_m)^{j_2} \big)(x).
\end{equation*}
Since the sum over $j_2$ in the left-hand side of (\ref{goal1}) is a finite sum over $j_2$ near $j$, we may consider only the case $j_2=j$ and thus our claim is 
\begin{equation}\label{mainmainclaim}
\Big\Vert \sum_{j\in\mathbb{Z}}{T_{\sigma_j}\big((f_1)_j,(f_2)_j,(f_3)^j,\dots,(f_m)^j \big)}\Big\Vert_{L^p(\rn)}\lesssim \mathcal{L}_{\sss}^r[\sigma]\prod_{i=1}^{m}{\Vert f_i\Vert_{L^{p_i}(\rn)}}.
\end{equation}

Using Lemma \ref{keyestilemma}, (\ref{mkbound}), and (\ref{compactembedding1}), we obtain the pointwise estimate 
\begin{align}\label{keyestimate}
&\big|T_{\sigma_j}\big((f_1)_j,(f_2)_j,(f_3)^j,\dots,(f_m)^j \big)(x) \big|\lesssim \mathcal{L}_{\sss}^r[\sigma]\Big( \prod_{i=1}^{2}\mathfrak{M}_{s_i,2^j}^{t}(f_i)_j(x)\Big)\Big(\prod_{i=3}^{m}\mathfrak{M}_{s_i,2^j}^{t}(f_i)^j(x)\Big).
\end{align}
Since $s_1,s_2>n/t=n/\min{(p_1,2,t)}=n/\min{(p_2,2,t)}$, it follows from Lemma \ref{equivalence} that for any dyadic cube $Q\in\mathcal{D}$ there exists    measurable  proper subsets $S_{Q}^1$ and $S_{Q}^2$ of $Q$ such that $|S_Q^1|,|S_Q^2|>\frac{3}{4}|Q|$ and 
\begin{equation}\label{equivX}
\Vert f_i\Vert_{X^{p_i}}\approx \Big\Vert \Big\{ \sum_{Q\in\mathcal{D}_j}{\Big( \inf_{y\in Q}{\mathfrak{M}_{s_i,2^j}^t(f_i)_j(y) }\Big)\chi_{S_Q^i}}\Big\}_{j\in\mathbb{Z}}\Big\Vert_{L^{p_i}(\ell^2)}, \quad i=1,2.
\end{equation}
Note that $|S_Q^{1}\cap S_Q^{2}|\geq \frac{1}{2}|Q|$ and thus, for any $\tau>0$
\begin{equation}\label{intersectionchi}
\chi_Q(x)\lesssim_{\tau}\mathcal{M}_{\tau}\big(\chi_{S_Q^{(1)}\cap S_Q^{(2)}}\big)(x)\chi_Q(x),
\end{equation} using the argument in (\ref{inversechi}).
Clearly, the constant in the inequality is independent of $Q$.
Now we choose $\tau<\min{(1,p)}$, and apply (\ref{keyestimate}), (\ref{intersectionchi}), and (\ref{maximal1}), as in the proof of Lemma \ref{bmoboundlemma}, to obtain
\begin{align*}
&\Big\Vert \sum_{j\in\mathbb{Z}}{T_{\sigma_j}\big((f_1)_j,(f_2)_j,(f_3)^j,\dots,(f_m)^j \big)}\Big\Vert_{L^p(\rn)}\\
&\lesssim \mathcal{L}_{\sss}^r[\sigma]\Big\Vert \sum_{j\in\mathbb{Z}}{\sum_{Q\in \mathcal{D}_j}{\Big(\prod_{i=1}^{2}{\mathfrak{M}_{s_i,2^j}^t(f_i)_j}\Big)\Big(\prod_{i=3}^{m}{\mathfrak{M}_{s_i,2^j}^t(f_i)^j}\Big)  \chi_Q  }}\Big\Vert_{L^{p}(\rn)}\\
&\lesssim \mathcal{L}_{\sss}^r[\sigma]  \Big\Vert \sum_{j\in\mathbb{Z}}{\sum_{Q\in \mathcal{D}_j}{\Big(\prod_{i=1}^{2}\Omega_{s_i,t}^{Q,1}f_i\Big)\Big( \prod_{i=3}^{m} \Omega_{s_i,t}^{Q,2}f_i\Big)  \mathcal{M}_{\tau}{(\chi_{S_Q^{1}\cap S_Q^{2}})  }}}\Big\Vert_{L^{p}(\rn)}\\
&\lesssim \mathcal{L}_{\sss}^r[\sigma]  \Big\Vert \sum_{j\in\mathbb{Z}}{\sum_{Q\in \mathcal{D}_j}{\Big(\prod_{i=1}^{2}\Omega_{s_i,t}^{Q,1}f_i\Big)\Big( \prod_{i=3}^{m} \Omega_{s_i,t}^{Q,2}f_i\Big)   {\chi_{S_Q^{1}}  }{\chi_{S_Q^{2}}  }}}\Big\Vert_{L^{p}(\rn)} ,
\end{align*}
where $\Omega_{s,t}^{Q,1}g:=\inf_{y\in Q}{\mathfrak{M}_{s,2^j}^t(g)_j(y)}$ and $\Omega_{s,t}^{Q,2}g:=\inf_{y\in Q}{\mathfrak{M}_{s,2^j}^t(g)^j(y)}$ for all $Q\in\mathcal{D}_j$.

Using H\"older's inequality and the fact that $\Omega_{s_i,t}^{Q,2}f_i\leq \mathfrak{M}_{s_i,t}^t(f_i)^j(x)$ for all $x\in Q\in\mathcal{D}_j$,  the $L^p$ norm in the last displayed expression is bounded by a constant times
\begin{align*}
& \Big( \prod_{i=1}^{2}\Big\Vert \Big\{\sum_{Q\in\mathcal{D}_j}{   \Omega_{s_i,t}^{Q,1}f_i  \cdot \chi_{S_Q^{i}}    }\Big\}_{j\in\mathbb{Z}}\Big\Vert_{L^{p_i}(\ell^2)}\Big) \Big(\prod_{i=3}^{m}{\big\Vert \big\{ \mathfrak{M}_{s_i,2^j}^{t}(f_i)^j\big\}_{j\in\mathbb{Z}}\big\Vert_{L^{p_i}(\ell^{\infty})}}\Big)\\
&\lesssim \Vert f_1\Vert_{X^{p_1}(\rn)}\Vert f_2\Vert_{X^{p_2}(\rn)}\prod_{i=3}^{m}{\Vert f_i\Vert_{H^{p_i}(\rn)}} , 
\end{align*}
where the inequality follows from Lemmas \ref{maximal2}, \ref{equivalence}, and the definition of Hardy space $H^{p_i}$. 
Since $H^{p_i}=L^{p_i}\subset X^{p_i}$ when $1<p_i\leq \infty$, we finally obtain (\ref{mainmainclaim}).

\subsection{High frequency part}
The proof for the high frequency part relies on the fact that if $\widehat{g_j}$ is supported on $\{\xi\in \rn : C^{-1} 2^{j}\leq |\xi|\leq C2^{j}\}$ for some $C>1$, then
\begin{equation}\label{marshall}
\Big\Vert \Big\{ \psi_j\ast \Big(\sum_{l=j-h}^{j+h}{g_l}\Big)\Big\}_{j\in\mathbb{Z}}\Big\Vert_{L^p(\ell^q)}\lesssim_{h,C} \big\Vert \big\{ g_j\big\}_{j\in\mathbb{Z}}\big\Vert_{L^p(\ell^q)}, \qquad 0<p<\infty
\end{equation} for $h\in \mathbb{N}$. The proof of (\ref{marshall}) is elementary and standard, so it is omitted here. Just use the estimate $|\psi_j\ast g_l(x)|\lesssim_{\sigma} \mathfrak{M}_{\sigma,2^l}g_l(x)$ for $0<\sigma<p,q$ and  $j-h\leq l\leq j+h$, and apply Lemma \ref{maximal2}.

We note that
\begin{equation*}
T_{\sigma_{high}^{(1)}}\fff(x)=\sum_{j\in\mathbb{Z}}{T_{\sigma_j}\big((f_1)_j,(f_2)^{j,m},\dots, (f_m)^{j,m} \big)(x)}
\end{equation*}
where $\phi_l$ is defined as in the previous subsection and  $(f_i)^{j,m}:=\phi_{j-4-\lfloor \log_2{m}\rfloor}\ast f_i$ for $2\leq i\leq m$.
Observe that the Fourier transform of $T_{\sigma_j}\big((f_1)_j,(f_2)^{j,m},\dots,(f_m)^{j,m} \big)$ is supported in $\big\{\xi\in\rn : 2^{j-3}\leq |\xi|\leq 2^{j+3} \big\}$ and thus (\ref{marshall}) yields that
\begin{equation}\label{squareexpression}
\big\Vert T_{\sigma_{high}^{(1)}}\fff \big\Vert_{L^p(\rn)} \lesssim  \big\Vert \big\{  T_{\sigma_j}\big((f_1)_j,(f_2)^{j,m},\dots,(f_m)^{j,m} \big)\big\}_{j\in\mathbb{Z}}\big\Vert_{L^p(\ell^{2})}.
\end{equation}
Moreover, it follows from Lemma \ref{keyestilemma}, (\ref{mkbound}), and (\ref{compactembedding1}), that
\begin{equation*}
\big| T_{\sigma_j}\big((f_1)_j,(f_2)^{j,m},\dots,(f_m)^{j,m} \big)(x)\big| \lesssim \mathcal{L}_{\sss}^r[\sigma]\mathfrak{M}_{s_1,2^j}^{t}(f_1)_j(x)\prod_{i=2}^{m}\mathfrak{M}_{s_i,2^j}^{t}(f_i)^{j,m}(x).
\end{equation*}

For $Q\in\mathcal{D}$ let $S_Q^1$ be a  measurable proper subset of $Q$ such that $|S_Q^1|>\frac{3}{4}|Q|$ and (\ref{equivX}) holds for $i=1$ as before, and
 by an analogous argument we  obtain
\begin{align*}
&\big\Vert T_{\sigma_{high}^{(1)}}\fff \big\Vert_{L^p(\rn)}\lesssim \mathcal{L}_{\sss}^r[\sigma]\Big\Vert  \Big\{ \mathfrak{M}_{s_1,2^j}^{t}(f_1)_j\prod_{i=2}^{m}\mathfrak{M}_{s_i,2^j}^{t}(f_i)^{j,m}\Big\}_{j\in\mathbb{Z}} \Big\Vert_{L^p(\ell^2)}\\
&\qquad \lesssim \mathcal{L}_{\sss}^r[\sigma]\Big\Vert  \Big\{ \sum_{Q\in\mathcal{D}_j}{ \Big( \inf_{y\in Q}{\mathfrak{M}_{s_1,2^j}^{t}(f_1)_j(y)}\Big) \Big[\prod_{i=2}^{m}\Big(\inf_{y\in Q}{\mathfrak{M}_{s_i,2^j}^{t}(f_i)^{j,m}(y)\Big) }\Big]\chi_{S_Q^1}} \Big\}_{j\in\mathbb{Z}} \Big\Vert_{L^p(\ell^2)}\\
&\qquad \lesssim \mathcal{L}_{\sss}^r[\sigma]\Big\Vert \Big\{ \sum_{Q\in\mathcal{D}_j}{\Big( \inf_{y\in Q}{\mathfrak{M}_{s_1,2^j}^t(f_1)(y)}\Big)\chi_{S_Q^1}}\Big\}_{j\in\mathbb{Z}}\Big\Vert_{L^{p_1}(\ell^2)} \prod_{i=2}^{m}{\Big\Vert \sup_{j\in\zz}\big|\phi_{j}\ast f_i\big|\Big\Vert_{L^{p_i}(\rn)}}\\
&\qquad \lesssim \mathcal{L}_{\sss}^r[\sigma]\Vert f_1\Vert_{X^{p_1}(\rn)}\prod_{i=2}^{m}{\Vert f_i\Vert_{H^{p_i}(\rn)}}\lesssim \mathcal{L}_{\sss}^r[\sigma]\prod_{i=1}^{m}{\Vert f_i\Vert_{L^{p_i}(\rn)}}.
\end{align*}

\hfill

\section{Proof of Proposition \ref{propo2}}
We  consider only the case $l<m$ as a similar and simpler procedure is applicable to  the case $l=m$.
Let $1\leq l<m$, $0<p_1,\dots,p_l\leq 1$, $p_{l+1}=\dots=p_m=\infty$, and $1/p=1/p_1+\dots+1/p_l$.
 For simplicity we assume that $\Vert f_{l+1}\Vert_{L^{\infty}(\rn)}=\cdots=\Vert f_{m}\Vert_{L^{\infty}(\rn)}=1$. Then the aim is to show that
  \begin{equation}\label{equ:TSigSmP}
    \big\Vert{T_{\sigma}\fff}\big\Vert_{L^p(\rn)} \lesssim \mathcal{L}_{\sss}^{r,\Psi^{(m)}}[\sigma] \prod_{i=1}^{l} \Vert{f_i}\Vert_{H^{p_i}(\rn)}.
  \end{equation}
  Let $f_i\in H^{p_i}(\rn)$ for $1\leq i\leq l$.
  Using  atomic representations, we write
  $$
  f_i = \sum_{k_i=1}^{\infty}\lambda_{i,k_i}a_{i,k_i},\quad 1\le i\le l,
  $$
  where $a_{i,k_i}$ are $L^\infty$-atoms for $H^{p_i}$ satisfying
  \begin{equation}\label{atomproperty1}
  \textup{Supp}(a_{i,k_i})\subset Q_{i,k_i},\quad \Vert{a_{i,k_i}}\Vert_{L^{\infty}(\rn)}\le {|{Q_{i,k_i}}|}^{-1/p_i},\quad \int_{Q_{i,k_i}}x^{\alpha}a_{i,k_i}(x)dx=0
  \end{equation}
  for ${\left\vert{\alpha}\right\vert}<N_i$ with $N_i$ large enough,
and \begin{equation}\label{atomproperty2}
\Big(\sum_{k_i=1}^{\infty}{\left\vert{\lambda_{i,k_i}}\right\vert}^{p_i}\Big)^{1/p_i} \lesssim\Vert{f_i}\Vert_{H^{p_i}(\rn)}.
\end{equation}
By the regularization in Lemma \ref{regularization}, we can assume that $\sigma$ is a Schwartz function whose Fourier transform has a compact support in $(\rn)^m$.
 Then Lemma \ref{switchsum} yields that
  \begin{equation}\label{exchangesumoperator}
  T_{\sigma}\fff(x) = \sum_{k_1=1}^{\infty}\cdots\sum_{k_l=1}^{\infty}\lambda_{1,k_1}\ldots \lambda_{l,k_l} T_{\sigma}\big(a_{1,k_1},\ldots,a_{l,k_l},f_{l+1},\ldots,f_m\big)(x)
  \end{equation}
  for a.e. $x\in\mathbb R^n.$

  For a cube $Q $ we denote by $Q^*$ its concentric dilate by a factor $10\sqrt{n}.$
   Now we can split $T_{\sigma}\fff$ into two parts and estimate
  \begin{equation*}
  \big|T_{\sigma}\fff(x)\big|\leq G_1(x) +G_2(x),
  \end{equation*}
  where
  \begin{align*}
  G_1&:= \sum_{k_1=1}^{\infty}\cdots\sum_{k_l=1}^{\infty}{|{\lambda_{1,k_1}}|}\cdots{|{\lambda_{l,k_l}}|}{\big|{T_{\sigma}(a_{1,k_1},\ldots,a_{l,k_l},f_{l+1},\ldots,f_m)}\big|}\chi_{Q^*_{1,k_1}\cap\cdots\cap Q^*_{l,k_l}},\\
   G_2&:= \sum_{k_1=1}^{\infty}\cdots\sum_{k_l=1}^{\infty}{|{\lambda_{1,k_1}}|}\cdots {|{\lambda_{l,k_l}}|} {\big|{T_{\sigma}(a_{1,k_1},\ldots,a_{l,k_l},f_{l+1},\ldots,f_m)}
  \big|}\chi_{(Q^*_{1,k_1}\cap\cdots\cap Q^*_{l,k_l})^c}.
  \end{align*}

The first part $G_1$ can be dealt with via Lemma \ref{grkalemma}.
Suppose that $Q^*_{1,k_1}\cap\cdots\cap Q^*_{l,k_l} \neq \emptyset$, as if this intersection is empty we are done.
From these cubes, choose a cube that has the minimum side length, and denote it by $R_{k_1,\ldots,k_l}$.
Then
\begin{equation*}
 Q^*_{1,k_1}\cap\cdots\cap Q^*_{l,k_l} \subset R_{k_1,\ldots,k_l}\subset Q^{\ast\ast}_{1,k_1}\cap\cdots\cap Q^{\ast\ast}_{l,k_l}, 
 \end{equation*}
where $Q^{**}_{i,k_i}:=(Q_{i,k_i}^*)^*$ denotes a dilation of $Q^{*}_{i,k_i}$ .
We shall prove
\begin{equation}\label{avgest}
 \frac{1}{|{R_{k_1,\ldots,k_l}}|}\int_{R_{k_1,\ldots,k_l}}{\big|{T_{\sigma}\big(a_{1,k_1},\ldots,a_{l,k_l},f_{l+1},\ldots,f_m\big)(y)}\big|} dy  \lesssim \mathcal{L}_{\sss}^{r,\Psi^{(m)}}[\sigma]\prod_{i=1}^l{\left\vert{Q_{i,k_i}}\right\vert}^{-1/{p_i}}. 
\end{equation}
To verify this, we may assume, without loss of generality, $R_{k_1,\ldots,k_l}=Q_{1,k_1}^{*}$.
Using the Cauchy-Schwarz inequality, the left-hand side of (\ref{avgest}) is majored by
\begin{align*}
\frac{1}{|Q_{1,k_1}^*|^{1/2}}\Big\Vert T_{\sigma}\big(a_{1,k_1},\ldots,a_{l,k_l},f_{l+1},\ldots,f_m\big)\Big\Vert_{L^2(\rn)}
\end{align*}
and this is less than a constant multiple of 
\begin{equation*}
\mathcal{L}_{\sss}^{r,\Psi^{(m)}}[\sigma]\frac{1}{|Q_{1,k_1}^*|^{1/2}}\Vert{a_{1,k_1}}\Vert_{L^2(\rn)}\prod_{i=2}^l\Vert{a_{i,k_i}}\Vert_{L^{\infty}(\rn)}\lesssim  \mathcal{L}_{\sss}^{r,\Psi^{(m)}}[\sigma]\prod_{i=1}^l{\left\vert{Q_{i,k_i}}\right\vert}^{-1/p_i}
\end{equation*} 
in view of Proposition \ref{propo1}.
This proves \eqref{avgest}.

We now apply Lemma \ref{grkalemma}, the estimate \eqref{avgest}, and the H\"older inequality to obtain
\begin{align*}
\Vert G_1\Vert_{L^p(\rn)}&\lesssim \Big\Vert \sum_{k_1=1}^{\infty}\cdots\sum_{k_l=1}^{\infty} \Big(\prod_{i=1}^{l} |{\lambda_{i,k_i}}| \Big)\chi_{R_{k_1,\ldots,k_l}} \frac{1}{|{R_{k_1,\ldots,k_l}}|}\\
   &\qquad\qquad\times\int_{R_{k_1,\ldots,k_l}}\big|{T_{\sigma}  \big(a_{1,k_1},\ldots,a_{l,k_l},f_{l+1},\ldots,f_m\big)(y) }\big| dy \Big\Vert_{L^p(\rn)}\\
   &\lesssim \mathcal{L}_{\sss}^{r,\Psi^{(m)}}[\sigma] \Big\Vert{ \sum_{k_1=1}^{\infty}\cdots\sum_{k_l=1}^{\infty}\Big(\prod_{i=1}^{l} |{\lambda_{i,k_i}}||{Q_{i,k_i}}|^{-1/p_i}\chi_{Q^{\ast\ast}_{i,k_i}}\Big) \Big\Vert_{L^p(\rn)}}\\
   &\leq \mathcal{L}_{\sss}^{r,\Psi^{(m)}}[\sigma] \prod_{i=1}^{l}\Big\Vert{\sum_{k_i=1}^{\infty}|{\lambda_{i,k_i}}| |{Q_{i,k_i}}|^{-1/p_i}\chi_{Q^{**}_{i,k_i}}}\Big\Vert_{L^{p_i}(\rn)},
    \end{align*}
and this clearly implies that
\begin{equation}\label{equ:G1Part} 
\Vert G_1\Vert_{L^p(\rn)}\lesssim \mathcal{L}_{\sss}^{r,\Psi^{(m)}}[\sigma]\prod_{i=1}^{l}\Vert{f_i}\Vert_{H^{p_i}(\rn)}
\end{equation}
as
\begin{equation*}
\Big\Vert{\sum_{k_i=1}^{\infty}|{\lambda_{i,k_i}}| |{Q_{i,k_i}}|^{-1/p_i}\chi_{Q^{**}_{i,k_i}}}\Big\Vert_{L^{p_i}(\rn)}\leq \Big( \sum_{k_i=1}^{\infty}|\lambda_{i,k_i}|^{p_i}\Big)^{1/p_i}\lesssim \Vert f_i\Vert_{H^{p_i}(\rn)}.
\end{equation*}

To obtain the estimate for $G_2$, we need the following lemma whose proof will be given in Section \ref{prooflemmas}.
\begin{lemma}\label{keylemma1}
   Let $1\le l<m$ and $0<p_1, \dots, p_l\le 1$. Let $a_i$, $1\leq i\leq l$, be atoms supported in the cube $Q_i$ such that
 \begin{equation}\label{atomconditions}
   \Vert{a_i}\Vert_{L^{\infty}(\rn)}\le {\left\vert{Q_i}\right\vert}^{-1/p_i},\qquad \int_{Q_i}x^{\alpha}a_i(x)dx=0
 \end{equation}
   for all ${\left\vert{\alpha}\right\vert}< N_i$ with $N_i$ sufficiently large. Let $\Vert f_{l+1}\Vert_{L^{\infty}(\rn)}=\cdots=\Vert f_m\Vert_{L^{\infty}(\rn)}=1$.
Suppose that (\ref{conditionss}) holds for all $J\subset \{1,\dots,l\}$ and  ${\sigma}$ satisfies $\mathcal{L}_{\sss}^{r,\Psi^{(m)}}[\sigma]<\infty$. 
  Then for any nonempty subset $J_0$ of $\{1,\dots,l\}$, there exist nonnegative functions $b_1^{J_0},\ldots,b_{l}^{J_0}$ such that
  \begin{equation*}
  \big\Vert{b_i^{J_0}} \big\Vert_{L^{p_i}(\rn)}\lesssim 1 \quad \text{ for }~ 1\le i\le l,
  \end{equation*} and 
    \begin{equation*}
    \big|{T_{\sigma}(a_1,\ldots,a_l,f_{l+1},\ldots,f_m)(x)}\big| \lesssim \mathcal{L}_{\sss}^{r,\Psi^{(m)}}[\sigma]b_1^{J_0}(x)\cdots b_{l}^{J_0}(x)
    \end{equation*}
    for all $x\in \big(\bigcap_{i\notin J_0}Q_i^*\big)\setminus \big(\bigcup_{i\in J_0}Q_i^*\big)$.
  \end{lemma}

Let $J_0$ be a nonempty subset of $\{1,\dots,l\}$. Then Lemma \ref{keylemma1} ensures the existence of nonnegative functions $b_{1,k_1}^{J_0},\ldots,b_{l,k_l}^{J_0}$ such that
  \begin{equation*}
  \big|T_{\sigma}\big(a_{1,k_1},\ldots,a_{l,k_l},f_{l+1},\ldots,f_m\big)(x)\big| \lesssim \mathcal{L}_{\sss}^{r,\Psi^{(m)}}[\sigma]  b_{1,k_1}^{J_0}(x) \cdots b_{l,k_l}^{J_0}(x)
  \end{equation*}
  for all $x\in \big(\bigcap_{i\notin J_0}Q_{i,k_{i}}^*\big)\setminus \big(\bigcup_{i\in J_0}Q_{i,k_{i}}^*\big)$   and $\Vert{b_{i,k_i}^{J_0}}\Vert_{L^{p_i}(\rn)} \lesssim 1$.

 Now set 
 \begin{equation*}
  b_{i,k_i}:= \sum_{\emptyset \ne J_0\subset\left\{1,2,\ldots,l\right\}}b_{i,k_i}^{J_0}.
  \end{equation*}
  Since it is a finite sum, we first note that  $\Vert{b_{i,k_i}}\Vert_{L^{p_i}(\rn)} \lesssim 1$.
  In addition, we have the pointwise estimate
  \begin{align*}
   \big|{T_{\sigma}\big(a_{1,k_1},\ldots,a_{l,k_l},f_{l+1},\ldots,f_m\big)(x)}\big|  \chi_{(Q^*_{1,k_1}\cap\ldots\cap Q^*_{l,k_l})^c}(x) \nonumber 
  \lesssim \mathcal{L}_{\sss}^{r,\Psi^{(m)}}[\sigma] b_{1,k_1}(x) \cdots b_{l,k_l}(x),
  \end{align*}
 which yields that
  \begin{equation*}
  G_2(x)\lesssim \mathcal{L}_{\sss}^{r,\Psi^{(m)}}[\sigma]\prod_{i=1}^l\Big( \sum_{k_i=1}^{\infty}{|{\lambda_{i,k_i}}|}b_{i,k_i}(x) \Big).
  \end{equation*}
  Then we apply   H\"older's inequality to deduce  that
  \begin{align}\label{equ:G2Part}
  \Vert{G_2}\Vert_{L^p(\rn)}&\lesssim  \mathcal{L}_{\sss}^{r,\Psi^{(m)}}[\sigma] \prod_{i=1}^{l}\Big\Vert \sum_{k_i=1}^{\infty}{|\lambda_{i,k_i}|b_{i,k_i}}\Big\Vert_{L^{p_i}(\rn)} \lesssim \mathcal{L}_{\sss}^{r,\Psi^{(m)}}[\sigma] \prod_{i=1}^{l}\Vert{f_i}\Vert_{H^{p_i}(\rn)}  
  \end{align}
 because
  \begin{equation*}
  \Big\Vert \sum_{k_i=1}^{\infty}{|\lambda_{i,k_i}|b_{i,k_i}}\Big\Vert_{L^{p_i}(\rn)}\leq \Big( \sum_{k_i=1}^{\infty}{|\lambda_{i,k_i}|^{p_i}\Vert b_{i,k_i}\Vert_{L^{p_i}(\rn)}^{p_i}}\Big)^{1/p_i}\lesssim \Big( \sum_{k_i=1}^{\infty}{|\lambda_{i,k_i}|^{p_i}}\Big)^{1/p_i}\lesssim \Vert f_i\Vert_{H^{p_i}(\rn)}.
  \end{equation*}
  
  Combining (\ref{equ:G1Part}) and (\ref{equ:G2Part}), we finally obtain \eqref{equ:TSigSmP} as desired.
This completes the proof.

\hfill

\section{Proof of Proposition \ref{propo3}}
Let $1<r\leq 2$, $1\leq l < \rho \leq m$, and 
\begin{equation}\label{indexset}
\mathrm{I}:=\{1,\dots,l\}, \quad \mathrm{II}:=\{l+1,\dots,\rho\},\quad \mathrm{III}:=\{\rho+1,\dots,m\},\quad  \Lambda:=\mathrm{I}\cup \mathrm{II}\cup \mathrm{III}.
\end{equation}
Assume that $0<p_i\leq 1$ for $i\in\mathrm{I}$, $r\le p_{i}<\infty$ for $i\in\mathrm{II}$,  and ${1}/{p_1}+ \dots + {1}/{p_\rho} = {1}/{p}$.
Let $\Vert f_{i}\Vert_{L^{\infty}(\rn)}=1$ for $i\in\mathrm{III}$.
As in (\ref{sigmadecompo}), we write
\begin{equation*}
\sigma(\xxxi)=\sum_{j_1,\dots,j_m \in \mathbb{Z}}{\sigma(\xxxi)\widehat{\psi_{j_1}}(\xi_1)\cdots \widehat{\psi_{j_m}}(\xi_m)}.
\end{equation*}
If $\max{(j_1,\dots,j_m)}=j_k$, then there are two cases 
\begin{equation}
 j_k-3-\lfloor \log_2m\rfloor \leq \max_{j_i\not= j_k}{(j_i)}\leq j_k \tag{\textbf{Case1}}\label{case1},
\end{equation} 
\begin{equation}
\max_{j_i\not= j_k}{(j_i)}\leq j_k-4-\lfloor \log_2m\rfloor \tag{\textbf{Case2}}\label{case2}.
\end{equation}

For  (\ref{case1}), we utilize the argument in Section \ref{lowfrequencypart}. That is, we need to prove that for $1\leq \kappa_1<\kappa_2\leq m$
\begin{align}\begin{split}\label{case1claim}
&\Big\Vert \sum_{j\in\zz}{\big| T_{\sigma_j}\big( (f_1)^j,\dots,(f_{\kappa_1-1})^j,(f_{\kappa_1})_j,(f_{\kappa_1+1})^j,\dots} \\
&\qquad \qquad \qquad \qquad \qquad\qquad{,(f_{\kappa_2-1})^j,(f_{\kappa_2})_j,(f_{\kappa_2+1})^j,\dots,(f_m)^j \big)\big|}    \Big\Vert_{L^p(\rn)}\\
&\lesssim \mathcal{L}_{\sss}^{r}[\sigma]\prod_{i\in\mathrm{I}\cup\mathrm{II}}{\Vert f_i\Vert_{H^{p_i}(\rn)}},
\end{split}
\end{align} 
where $(g)_j:=\psi_j\ast g$ and $(g)^j:=\phi_j\ast g$ as before.

We remark that (\ref{case2}) is a high frequency part for which $T_{\sigma_{high}^{(\kappa)}}\big(f_1,\dots,f_m \big)$ is written as the sum, over $j_{\kappa}\in\zz$, of terms whose Fourier transform is supported in an annulus of size $2^{j_{\kappa}}$ where $\sigma^{(\kappa)}$ is defined as in (\ref{sigmakappa}) and $\sigma_{high}^{({\kappa})}$ is similar to (\ref{highdef}).
Thus, the square function characterization of $H^p$ will be applied to deal with this case as in (\ref{squareexpression}).
We will actually prove that for each $1\leq {\kappa}\leq m$
\begin{align}\label{case2claim}
&\Big\Vert \Big( \sum_{j\in\zz}{\big|T_{\sigma_j}\big((f_1)^{j,m},\dots,(f_{{\kappa}-1})^{j,m},(f_{\kappa})_j ,(f_{{\kappa}+1})^{j,m},\dots, (f_m)^{j,m}\big) \big|^2} \Big)^{1/2}\Big\Vert_{L^p(\rn)}\nonumber\\
&\lesssim  \mathcal{L}_{\sss}^{r}[\sigma]\prod_{i\in\mathrm{I}\cup\mathrm{II}}{\Vert f_i\Vert_{H^{p_i}(\rn)}}
\end{align} where $(g)^{j,m}:=\phi_{j-4-\lfloor \log_2m\rfloor}\ast g$.

\subsection{Proof of (\ref{case1claim}) for $1\leq {\kappa}_1<{\kappa}_2\leq m$}

Let $\wt{\psi_j}:=\psi_{j-1}+\psi_{j}+\psi_{j+1}$ so that $\wt{\psi_j}\ast \psi_j=\psi_j$ and for each $1\leq {\kappa}_1<{\kappa}_2\leq m$ we define

\begin{equation*}
\sigma_{j,1}^{{\kappa}_1,{\kappa}_2}:=\big(  \textup{ $ \underbrace{  \phi_{j}\otimes\cdots\otimes  \phi_{j} }_{({\kappa}_1-1)\textup{ times }}   $}    \otimes {\psi_j}\otimes \textup{ $  \underbrace{\phi_{j}\otimes\cdots\otimes \phi_{j}}_{({\kappa}_2-{\kappa}_1-1) \textup{ times }} $}  \otimes {\psi_j}\otimes \textup{ $  \underbrace{\phi_{j}\otimes\cdots\otimes \phi_{j}}_{(m-{\kappa}_2) \textup{ times }} $}\big)^{\wedge} \cdot \sigma_j,
\end{equation*} 
\begin{equation*}
\sigma_{j,2}^{{\kappa}_1,{\kappa}_2}:=\big(  \textup{ $ \underbrace{  \phi_{j}\otimes\cdots\otimes  \phi_{j} }_{({\kappa}_1-1)\textup{ times }}   $}    \otimes \wt{\psi_j}\otimes \textup{ $  \underbrace{\phi_{j}\otimes\cdots\otimes \phi_{j}}_{({\kappa}_2-{\kappa}_1-1) \textup{ times }} $}  \otimes \wt{\psi_j}\otimes \textup{ $  \underbrace{\phi_{j}\otimes\cdots\otimes \phi_{j}}_{(m-{\kappa}_2) \textup{ times }} $}\big)^{\wedge} \cdot \sigma_j.
\end{equation*} 
Then both $\sigma_{j,1}^{\kappa_1,\kappa_2}$ and $\sigma_{j,1}^{\kappa_1,\kappa_2}$ can be expressed in the form
\begin{equation*}
{\Xi}({\cdot}/2^j)\big) \cdot \sigma_j
\end{equation*} for some $\Xi\in\SS((\rn)^m)$ whose support is in a ball of a constant radius in $(\rn)^m$.
We observe that, thanks to Lemma \ref{katoponce}, for any $1<t<\infty$
\begin{equation}\label{katogeneral}
\big\Vert \Xi\cdot \sigma_j(2^j\cdot )\big\Vert_{L_{\sss}^{t}((\rn)^m)}\lesssim \big\Vert \sigma_j(2^j\cdot)\big\Vert_{L_{\sss}^{t}((\rn)^m)}\lesssim\mathcal{L}_{\sss}^{t}[\sigma],
\end{equation}  
and
\begin{align}
&T_{\sigma_{j,1}^{\kappa_1,\kappa_2}}\fff \nonumber\\
&=T_{\sigma_j}\big( (f_1)^j,\dots,(f_{{\kappa}_1-1})^j,(f_{{\kappa}_1})_j,(f_{{\kappa}_1+1})^j,\dots(f_{{\kappa}_2-1})^j,(f_{{\kappa}_2})_j,(f_{{\kappa}_2+1})^j,\dots,(f_m)^j \big)\label{1express}\\
&=T_{\sigma_{j,2}^{{\kappa}_1,{\kappa}_2}}\big(f_1,\dots,f_{{\kappa}_1-1},(f_{{\kappa}_1})_j,f_{{\kappa}_1+1},\dots,f_{{\kappa}_2-1},(f_{{\kappa}_2})_j,f_{{\kappa}_2+1},\dots,f_m \big).\label{2express}
\end{align}
Furthermore, if $1\leq \kappa_1<l+1<{\kappa}_2\leq m$, $T_{\sigma_{j,1}^{\kappa_1,\kappa_2}}\fff $ can be also written as
\begin{align}\label{3express}
T_{\sigma_{j,1}^{\kappa_1,\kappa_2}}\fff &=T_{\sigma_{j,2}^{{\kappa}_1,{\kappa}_2}}\big(f_1,\dots,f_{\kappa_1-1},(f_{\kappa_1})_{j},f_{\kappa_1+1},\dots, \\
&\qquad \qquad \qquad ,f_{{l}},(f_{l+1})^{j+1},f_{l+2},\dots,f_{{\kappa}_2-1},(f_{{\kappa}_2})_j,f_{{\kappa}_2+1},\dots,f_m \big)\nonumber
\end{align} since $\phi_{j+1}\ast \phi_{j}=\phi_{j}$. Similarly, for $l+1<{\kappa}_1<{\kappa}_2\leq m$, we have
\begin{align}\label{4express}
T_{\sigma_{j,1}^{\kappa_1,\kappa_2}}\fff &=T_{\sigma_{j,2}^{{\kappa}_1,{\kappa}_2}}\big(f_1,\dots,f_{l},(f_{l+1})^{j+1},f_{l+2},\dots, \\
&\qquad \qquad \qquad ,f_{{\kappa}_1-1},(f_{{\kappa}_1})_j,f_{{\kappa}_1+1},\dots,f_{{\kappa}_2-1},(f_{{\kappa}_2})_j,f_{{\kappa}_2+1},\dots,f_m \big)\nonumber.
\end{align}

Now we write, as in (\ref{exchangesumoperator}),
\begin{equation}\label{atomsepa}
T_{\sigma_{j,1}^{\kappa_1,\kappa_2}}\fff (x)=\sum_{k_1=1}^{\infty}\cdots\sum_{k_l=1}^{\infty}{\lambda_{1,k_1},\dots,\lambda_{l,k_l}T_{\sigma_{j,1}^{\kappa_1,\kappa_2}}\big( a_{1,k_1},\dots,a_{l,k_l},f_{l+1},\dots,f_m\big)(x)       }
\end{equation} where $a_{i,k_i}$ are $L^{\infty}$-atoms for $H^{p_i}$ satisfying (\ref{atomproperty1}) and (\ref{atomproperty2}).
Then we apply the following lemma that will be proved in Section \ref{prooflemmas}.

\begin{lemma}\label{keylemma2}
Let $1<r\leq 2$, $1\leq l<\rho\leq m$, and let $\mathrm{I}$, $\mathrm{II}$, $\mathrm{III}$, and $\Lambda$ be defined as in (\ref{indexset}). 
Suppose that $0<p_i\leq 1$ for $i\in\mathrm{I}$, $r\leq p_{i}<\infty$ for $i\in\mathrm{II}$, and $1/p=1/p_1+\dots+1/p_{\rho}$. 
Let $a_i$, $i\in\mathrm{I}$, be atoms supported in the cube $Q_i$ such that (\ref{atomconditions}) holds for all $|\alpha|<N_i$ with $N_i$ sufficiently large. Suppose that $(\ref{conditionss})$ holds for all $J\subset \mathrm{I}$. Let $f_i\in L^{p_i}(\rn)$ for $i\in\mathrm{II}$ and $\Vert f_{i}\Vert_{L^{\infty}(\rn)}=1$ for $i\in\mathrm{III}$. Then there exist nonnegative functions $b_i$, $i\in\mathrm{I}$, and $F_{i}$, $i\in\mathrm{II}$, on $\rn$ such that
\begin{equation}\label{kl2condition1}
 \sum_{j\in\zz}\big|T_{\sigma_{j,1}^{\kappa_1,\kappa_2}}\big( a_{1},\dots,a_{l},f_{l+1},\dots,f_m\big)(x)\big|  \lesssim \mathcal{L}_{\sss}^{r}[\sigma] \Big(\prod_{i\in\mathrm{I}}b_i(x)\Big)\Big( \prod_{i\in\mathrm{II}}F_i(x)\Big),
\end{equation}
\begin{equation*}
\Vert b_i\Vert_{L^{p_i}(\rn)}\lesssim 1, \qquad \Vert F_i\Vert_{L^{p_i}(\rn)}\lesssim \Vert f_i\Vert_{L^{p_i}(\rn)}.
\end{equation*}
\end{lemma}

Lemma \ref{keylemma2} proves the existence of functions $b_{i,k_i}$ for $i\in\mathrm{I}$, $k_i\in \mathbb Z$, and $F_i$ for $i\in\mathrm{II}$, having the properties that
  \begin{align}\label{property1}
&  \big|{T_{\sigma_{j,1}^{\kappa_1,\kappa_2}}\big(a_{1,k_1},\ldots, a_{l,k_l},f_{l+1}, \ldots, f_{m}\big)(x)}\big| \lesssim \mathcal{L}_{\sss}^{r}[\sigma]\Big( \prod_{i\in\mathrm{I}} b_{i,k_i}(x)\Big) \Big(\prod_{i\in\mathrm{II}}F_i(x)\Big),
  \end{align}
\begin{equation}\label{property2}
\Vert{b_{i,k_i}}\Vert_{L^{p_i}(\rn)}\lesssim 1, \qquad  \Vert F_i\Vert_{L^{p_i}(\rn)}\lesssim \Vert f_i\Vert_{L^{p_i}(\rn)}.
\end{equation}
By using (\ref{1express}) and (\ref{atomsepa}), the left-hand side of (\ref{case1claim}) is less than
\begin{align*}
\Big\Vert \sum_{k_1=1}^{\infty}\cdots \sum_{k_l=1}^{\infty}   |\lambda_{1,k_1}|\cdots |\lambda_{l,k_l}|\sum_{j\in\zz} \big|T_{\sigma_{j,1}^{\kappa_1,\kappa_2}}\big(a_{1,k_1},\dots,a_{l,k_l},f_{l+1},\dots, f_m\big) \big|\Big\Vert_{L^p(\rn)} .
\end{align*} 
Then (\ref{property1}) and  H\"older's inequality yield that the preceding expression is dominated by a constant times
\begin{equation*}
\mathcal{L}_{\sss}^{r}[\sigma]\Big( \prod_{i\in\mathrm{I}}\Big\Vert \sum_{k_i=1}^{\infty}{|\lambda_{i,k_i}|b_{i,k_i}}\Big\Vert_{L^{p_i}(\rn)} \Big)\Big( \prod_{i\in\mathrm{II}}\Vert F_i \Vert_{L^{p_i}(\rn)}\Big).
\end{equation*}
It is obvious that $\Vert F_i\Vert_{L^{p_i}(\rn)}\lesssim \Vert f_i\Vert_{L^{p_i}(\rn)}$, and we also have
\begin{equation*}
\Big\Vert{\sum_{k_i=1}^{\infty} | \lambda_{i,k_i}| b_{i,k_i} }\Big\Vert_{L^{p_i}(\rn)} \leq \Big(\sum_{k_i=1}^{\infty} |{\lambda_{i,k_i}}|^{p_i} \Vert{b_{i,k_i}}\Vert_{L^{p_i}(\rn)}^{p_i}\Big)^{1/p_i} \lesssim \Big(\sum_{k_i=1}^{\infty} |{\lambda_{i,k_i}}|^{p_i}\Big)^{1/p_i} \lesssim \Vert{f_i}\Vert_{H^{p_i}(\rn)}
\end{equation*} 
having used (\ref{property2}). 
This proves (\ref{case1claim}).

\hfill

\subsection{Proof of (\ref{case2claim}) for $1\leq \kappa\leq m$}

Let $\wt{\psi_j}:=\psi_{j-1}+\psi_{j}+\psi_{j+1}$ as above, and for each $1\leq \kappa\leq m$ we define
\begin{equation*}
\sigma_{j,1}^{\kappa}:=\big(  \textup{ $ \underbrace{  \phi_{j-4-\lfloor \log_2m\rfloor}\otimes\cdots\otimes  \phi_{j-4-\lfloor \log_2m\rfloor} }_{(\kappa-1)\textup{ times }}   $}    \otimes {\psi_j}\otimes \textup{ $  \underbrace{\phi_{j-4-\lfloor \log_2m\rfloor}\otimes\cdots\otimes \phi_{j-4-\lfloor \log_2m\rfloor}}_{(m-\kappa) \textup{ times }} $}\big)^{\wedge} \cdot \sigma_j.
\end{equation*} 
\begin{equation*}
\sigma_{j,2}^{\kappa}:=\big(  \textup{ $ \underbrace{  \phi_{j-4-\lfloor \log_2m\rfloor}\otimes\cdots\otimes  \phi_{j-4-\lfloor \log_2m\rfloor} }_{(\kappa-1)\textup{ times }}   $}    \otimes \wt{\psi_j}\otimes \textup{ $  \underbrace{\phi_{j-4-\lfloor \log_2m\rfloor}\otimes\cdots\otimes \phi_{j-4-\lfloor \log_2m\rfloor}}_{(m-\kappa) \textup{ times }} $}\big)^{\wedge} \cdot \sigma_j.
\end{equation*} 
Then the argument in (\ref{katogeneral}) yields that for any $1<t<\infty$
\begin{equation}\label{sigmajkest}
\big\Vert \sigma_{j,1}^{\kappa}(2^j\cdot)\big\Vert_{L_{\sss}^{t}((\rn)^m)}, \big\Vert \sigma_{j,2}^{\kappa}(2^j\cdot)\big\Vert_{L_{\sss}^{t}((\rn)^m)}\lesssim\mathcal{L}_{\sss}^{t}[\sigma].
\end{equation}  
Moreover, we note that
\begin{align}
T_{\sigma_{j,1}^{\kappa}}\fff&=T_{\sigma_j}\big((f_1)^{j,m},\dots,(f_{{\kappa}-1})^{j,m},(f_{\kappa})_j ,(f_{{\kappa}+1})^{j,m},\dots, (f_m)^{j,m}\big)\nonumber\\
&=T_{\sigma_{j,2}^{\kappa}}\big(f_1,\dots,f_{{\kappa}-1},(f_{\kappa})_j,f_{{\kappa}+1},\dots,f_m \big) \label{condexpression1}
\end{align}
and if $l+1<{\kappa}\leq m$, it can be also written as
\begin{equation}\label{secondexpression2}
T_{\sigma_{j,1}^{\kappa}}\fff=T_{\sigma_{j,2}^{\kappa}}\big(f_1,\dots,f_{l},(f_{l+1})^{j+1,m},f_{l+2},\dots,f_{{\kappa}-1},(f_{\kappa})_j,f_{{\kappa}+1},\dots,f_m \big)
\end{equation} since $\phi_{j-3-\lfloor \log_2{m}\rfloor}\ast \phi_{j-4-\lfloor \log_2{m}\rfloor}=\phi_{j-4-\lfloor \log_2{m}\rfloor}$.
Therefore, (\ref{case2claim}) is reduced to
\begin{equation}\label{case2reduction}
\Big\Vert \Big( \sum_{j\in\zz}\big|   T_{\sigma_{j,1}^{\kappa}}\fff    \big|^2\Big)^{1/2}\Big\Vert_{L^p(\rn)}\lesssim \mathcal{L}_{\sss}^r[\sigma]\prod_{i\in\mathrm{I}\cup\mathrm{II}}{\Vert f_i\Vert_{H^{p_i}(\rn)}}.
\end{equation}

Now we write, as in (\ref{exchangesumoperator}),
\begin{equation}\label{case2express}
T_{\sigma_{j,1}^{\kappa}}\fff(x)=\sum_{k_1=1}^{\infty}\cdots\sum_{k_l=1}^{\infty}{\lambda_{1,k_1},\dots,\lambda_{l,k_l}T_{\sigma_{j,1}^{\kappa}}\big( a_{1,k_1},\dots,a_{l,k_l},f_{l+1},\dots,f_m\big)(x)       }
\end{equation} where $a_{i,k_i}$ are $L^{\infty}$-atoms for $H^{p_i}$ satisfying (\ref{atomproperty1}) and (\ref{atomproperty2}).
Then we need the following lemma whose proof will be given in Section \ref{prooflemmas}.
\begin{lemma}\label{keylemma3}
Let $1<r\leq 2$, $1\leq l<\rho\leq m$, and let $\mathrm{I}$, $\mathrm{II}$, $\mathrm{III}$, and $\Lambda$ be defined as in (\ref{indexset}). 
Suppose that $0<p_i\leq 1$ for $i\in\mathrm{I}$, $r\leq p_{i}<\infty$ for $i\in\mathrm{II}$, and $1/p=1/p_1+\dots+1/p_{\rho}$. 
Let $a_i$, $i\in\mathrm{I}$, be atoms supported in the cube $Q_i$ such that (\ref{atomconditions}) holds for all $|\alpha|<N_i$ with $N_i$ sufficiently large. Suppose that $(\ref{conditionss})$ holds for all $J\subset \mathrm{I}$. Let $f_i\in L^{p_i}(\rn)$ for $i\in\mathrm{II}$ and $\Vert f_{i}\Vert_{L^{\infty}(\rn)}=1$ for $i\in\mathrm{III}$. Then there exist nonnegative functions $b_i$, $i\in\mathrm{I}$, and $F_{i}$, $i\in\mathrm{II}$, on $\rn$ such that
\begin{equation}\label{kl3condition1}
\Big( \sum_{j\in\zz}\big|T_{\sigma_{j,1}^{\kappa}}\big( a_{1},\dots,a_{l},f_{l+1},\dots,f_m\big)(x)\big|^2\Big)^{1/2} \lesssim \mathcal{L}_{\sss}^{r}[\sigma] \Big(\prod_{i\in\mathrm{I}}b_i(x)\Big)\Big( \prod_{i\in\mathrm{II}}F_i(x)\Big),
\end{equation}
\begin{equation*}
\Vert b_i\Vert_{L^{p_i}(\rn)}\lesssim 1,\qquad \Vert F_i\Vert_{L^{p_i}(\rn)}\lesssim \Vert f_i\Vert_{L^{p_i}(\rn)}.
\end{equation*}
\end{lemma}

According to the above lemma, we can choose nonnegative functions $b_{i,k_i}$, $i\in\mathrm{I}$, and $F_i$, $i\in\mathrm{II}$, such that
\begin{equation}\label{property11}
\Big( \sum_{j\in\zz}\big|T_{\sigma_{j,1}^k}\big( a_{1,k_1},\dots,a_{l,k_l},f_{l+1},\dots,f_m\big)(x)\big|^2\Big)^{1/2} \lesssim \mathcal{L}_{\sss}^{r}[\sigma] \Big(\prod_{i\in\mathrm{I}}b_{i,k_i}(x)\Big)\Big( \prod_{i\in\mathrm{II}}F_i(x)\Big),
\end{equation}
\begin{equation}\label{property22}
\Vert b_{i,k_i}\Vert_{L^{p_i}(\rn)}\lesssim 1, \qquad \Vert F_i\Vert_{L^{p_i}(\rn)}\lesssim \Vert f_i\Vert_{L^{p_i}(\rn)}.
\end{equation}

Using (\ref{case2express}), a triangle inequality in $\ell^2$, (\ref{property11}), and the H\"older inequality, the left-hand side of (\ref{case2reduction}) is less than
\begin{align*}
&\Big\Vert \sum_{k_1=1}^{\infty}\cdots \sum_{k_l=1}^{\infty} |\lambda_{1,k_1}|\cdots |\lambda_{l,k_l}|\Big(\sum_{j\in\zz}\big|T_{\sigma_{j,1}^{\kappa}}\big( a_{1,k_1},\dots,a_{l,k_l},f_{l+1},\dots,f_m\big)\big|^2 \Big)^{1/2} \Big\Vert_{L^p(\rn)}\\
&\lesssim \mathcal{L}_{\sss}^{r}[\sigma] \Big\Vert \sum_{k_1=1}^{\infty}\cdots\sum_{k_l=1}^{\infty}|\lambda_{1,k_1}|\cdots |\lambda_{l,k_l}|    \Big(\prod_{i\in\mathrm{I}}b_{i,k_i}\Big)\Big( \prod_{i\in\mathrm{II}}F_i\Big)    \Big\Vert_{L^p(\rn)}\\
&\leq \mathcal{L}_{\sss}^{r}[\sigma]\Big(\prod_{i\in\mathrm{I}}\Big\Vert \sum_{k_i=1}^{\infty}{|\lambda_{i,k_i}|b_{i,k_i}}\Big\Vert_{L^{p_i}(\rn)} \Big)\Big(\prod_{i\in\mathrm{II}}\Vert F_i\Vert_{L^{p_i}(\rn)} \Big).
\end{align*} 
This is clearly majored by the right-hand side of (\ref{case2reduction}) and in view of 
 (\ref{property22}) and the proof is concluded.

\hfill

\section{Proof of Proposition \ref{propo4}}\label{proofpropo6}

The proof will be based on the following interpolation method for multilinear multipliers.
\begin{lemma}\label{interpolation}
Let $1<r\leq 2$, $0<p_1^0,\dots,p_m^0\leq \infty$, $0<p_1^1,\dots,p_m^1\leq \infty$, $1/p^0=1/p_1^0+\dots+1/p_m^0$, and $1/p^1=1/p_1^1+\dots+1/p_m^1$.
Let $s_1^0,\dots,s_m^0\geq 0$ and $s_1^1,\dots,s_m^1\geq 0$.
Suppose that 
\begin{equation}\label{boundassumption}
\big\Vert T_{\sigma} \big\Vert_{H^{p_1^l}\times\cdots\times H^{p_m^l}\to L^{p^{l}}}\lesssim \mathcal{L}_{(s_1^l,\dots,s_m^l)}^{r, \Psi^{(m)}}[\sigma], \quad l=0,1.
\end{equation}
Then for any $0<\theta<1$, $0<p,p_1,\dots,p_m\leq \infty$, and $s_1,\dots,s_m\geq 0$ satisfying
\begin{equation*}
1/p=(1-\theta)/p^0+\theta/p^1, \qquad 1/p_k=(1-\theta)/p_k^0+\theta/p_k^1 \quad\text{ for }~1\leq k\leq m,
\end{equation*}
\begin{equation*}
s_k=(1-\theta)s_k^0+\theta s_k^1 \quad \text{ for }~1\leq k\leq m,
\end{equation*} we have
\begin{equation}\label{interpolationresult}
\big\Vert T_{\sigma}\big\Vert_{H^{p_1}\times \cdots \times H^{p_m}\to L^p}\lesssim \mathcal{L}^{r,\Psi^{(m)}}_{(s_1,\dots,s_m)}[\sigma].
\end{equation} 

\end{lemma}
\begin{proof}

Since the proof is more or less standard, we only provide a sketch of it.

Let $\wt{\Psi^{(m)}}:=2^{-mn}\Psi(\vec{\;\cdot\;}/2)+\Psi^{(m)}+2^{mn}\Psi^{(m)}(2\vec{\;\cdot\;})$ so that $\wt{\Psi^{(m)}}\ast \Psi^{(m)}=\Psi^{(m)}$.
We construct a family of multilinear Fourier multipliers $\sigma^z$ as
\begin{align*}
\sigma^z(\xxxi)&:=\frac{(1+\theta)^{mn+1}}{(1+z)^{mn+1}}\sum_{j\in\zz}(I-\Delta_1)^{-(s_1^0(1-z)+s_1^1z)/2}\cdots (I-\Delta_m)^{-(s_m^0(1-z)+s_m^1z)/2}\\
 &\qquad\qquad \qquad \Big( (I-\Delta_1)^{s_1/2}\cdots (I-\Delta_m)^{s_m/2}\big( \sigma(2^j\vec{\;\cdot\;})\wh{\Psi^{(m)}}\big)\Big)(\xxxi/2^j)\wh{\wt{\Psi^{(m)}}}(\xxxi/2^j).
\end{align*}
Note that $\sigma^{\theta}=\sigma$ and it follows from the interpolation theorem for analytic families of operators that
\begin{align*}
\Vert T_{\sigma}\Vert_{H^{p_1}\times\cdots\times H^{p_m}\to L^p}\leq \Big( \sup_{t\in\rr }\mathcal{L}_{(s_1^0,\dots,s_m^0)}^{r,\Psi^{(m)}}[ \sigma^{it}]\Big)^{1-\theta} \Big( \sup_{t\in\rr }\mathcal{L}_{(s_1^1,\dots,s_m^1)}^{r,\Psi^{(m)}}[ \sigma^{1+it}]\Big)^{\theta}
\end{align*} by applying (\ref{boundassumption}).
We refer to \cite{Ca, Ca_To, Fe_St2, Ja_Jo, Str_To} for more details. 

We now observe that for each $l=0,1$, due to compact support conditions and Lemma \ref{katoponce},
\begin{align*}
\mathcal{L}_{(s_1^l,\dots,s_m^l)}^{r,\Psi^{(m)}}[\sigma^{l+it}]&=\sup_{j\in\zz}{\Big\Vert \sigma^{l+it}(2^j\cdot )\wh{\Psi^{(m)}}\Big\Vert_{L^{r}_{(s_1^l,\dots,s_m^l)}((\rn)^m)}}\\
&\lesssim \frac{1}{(1+|t|)^{mn+1}}\sup_{j\in\zz}{\Big\Vert (I-\Delta_1)^{-it(s_1^0-s_1^1)/2}\cdots (I-\Delta_m)^{-it(s_m^0-s_m^1)/2}}\\
&\qquad \qquad \qquad  {(I-\Delta_1)^{s_1/2}\cdots (I-\Delta_m)^{s_m/2}\big(\sigma(2^j\cdot)\wh{\Psi^{(m)}}\big)\Big\Vert_{L^r((\rn)^m)}}\\
&\lesssim \mathcal{L}_{(s_1,\dots,s_m)}^{r,\Psi^{(m)}}[\sigma]
\end{align*}
where we applied the Marcinkiewicz multiplier theorem in the last inequality. 
This proves (\ref{interpolationresult}).
\end{proof}

We now state the following delicate interpolation result whose proof is based on that of \cite[Lemma 3.7]{Gr_Mi_Ng_Tom}. 
\begin{lemma}\label{interpolation2}
Let $1<r\leq 2$, $m\in\nn$ and  $0<p_1,\dots,p_m\leq \infty$.
For $\ppp:=(p_1,\dots,p_m)$ let 
\begin{equation*}
\Gamma_m(\ppp):=\Big\{\sss:\sum_{k\in J}{\big( {s_k}/{n}-{1}/{p_k} \big)}\geq -{1}/{r'}\quad \text{ for any }~ J\subset \{1,\dots,m\}\Big\}
\end{equation*}
and for each $1\leq u\leq m$ define
\begin{equation*}
\Lambda_m^{u}(\ppp):=\big\{\sss:s_u\geq {n}/{p_u}-{n}/{r'}, \quad s_i\geq {n}/{p_i} ~\text{ for all }~i\not= u \big\}.
\end{equation*}
Then $\Gamma_m(\ppp)$ is the convex hull of $\Lambda_m^u(\ppp)$ for $1\leq u\leq m$.
\end{lemma}
\begin{proof}
Let $H_m(\ppp)$ be the convex hull of $\Lambda_m^u(\ppp)$ for $1\leq u\leq m$ and then we need to show that $H_m(\ppp)=\Gamma_m(\ppp)$.

We first note that $\Gamma_m(\ppp)$ is convex as it is the intersection of $2^m-1$ convex sets. Since $\Gamma_m(\ppp)$ contains $\Lambda_m^u(\ppp)$ for all $1\leq u\leq m$, it is clear that
$H_m(\ppp)\subset \Gamma_m(\ppp)$.

We now verify $\Gamma_m(\ppp)\subset H_m(\ppp)$. For this one we restrict the size of $s_i$, $1\leq i\leq m$.
Let $M$ be a sufficiently large number such that $M>mn(1/p_1+\dots+1/p_m)$ and let 
\begin{equation*}
\Gamma_m^M(\ppp):=\Gamma_m(\ppp)\cap \big\{\sss:s_i\leq M ~1\leq i\leq m\big\},
\end{equation*}
\begin{equation*}
\Lambda_m^{u,M}(\ppp):=\Lambda_m^u(\ppp)\cap \big\{\sss:s_i\leq M ~1\leq i\leq m\big\},
\end{equation*}
\begin{equation*}
H_m^{M}(\ppp):=H_m(\ppp)\cap \big\{\sss:s_i\leq M ~1\leq i\leq m\big\},
\end{equation*}
and we actually prove that 
\begin{equation}\label{gammah}
\Gamma_m^M(\ppp)\subset H_m^M(\ppp) \quad \text{ for all }~ 0<p_1,\dots,p_m\leq \infty,
\end{equation}
from which we obtain the desired result by taking $M\to \infty$.
We use an induction argument beginning with the case $m=2$.

When $m=2$, it is trivial because $\Gamma_2^M(\ppp)$ is the convex hull of the five points $(M,M)$, $(n/p_1-n/r',M)$, $(n/p_1-n/r',n/p_2)$, $(n/p_1,n/p_2-n/r')$, and $(M,n/p_2-n/r')$.

Now we fix $m>2$ and assume that (\ref{gammah}) holds with $m$ replaced by $m-1$.
We denote
\begin{equation*}
\Gamma_m^{0,M}(\ppp):=\big\{\sss\in \Gamma_m^M(\ppp):n/p_l-n/r'\leq s_l\leq n/p_l ~ \text{ for all }~ 1\leq l\leq m \big\},
\end{equation*}
\begin{equation*}
\Gamma_m^{l,M}(\ppp):=\big\{\sss\in \Gamma_m^M(\ppp):n/p_l\leq s_l\leq M \big\},\qquad 1\leq l\leq m.
\end{equation*}
It is easy to see that $\Gamma_m^M(\ppp)=\bigcup_{l=0}^{m}\Gamma_m^{l,M}(\ppp)$ and thus it is enough to show that
\begin{equation}\label{gamma0}
\Gamma_m^{0,M}(\ppp)\subset H_m^M(\ppp),
\end{equation}
\begin{equation}\label{gamma1}
\Gamma_m^{l,M}(\ppp)\subset H_m^M(\ppp) \quad \text{ for all }~ 1\leq l\leq m.
\end{equation}

We note that $\Gamma_m^{0,M}(\ppp)$ is the intersection of the two sets
\begin{equation*}
\big\{ \sss:  n/p_l-n/r'\leq s_l\leq n/p_l \quad \text{ for all }~ 1\leq l\leq m\big\}
\end{equation*} and 
\begin{equation*}
\big\{\sss: s_1+\dots+s_m\geq n/p_1+\dots+n/p_m-n/r' \big\},
\end{equation*}  which would be  a standard $m$-simplex with the $m+1$ vertices 
$$(n/p_1,\dots,n/p_m),~ (n/p_1-n/r',n/p_2,\dots,n/p_m),~\dots,~ (n/p_1,\dots,n/p_{m-1},n/p_m-n/r').$$
Since the vertices of the simplex are contained in the convex set $H_m^M(\ppp)$,  (\ref{gamma0}) holds.

To achieve (\ref{gamma1}) we consider only the case $l=m$ as the other cases will follow from a rearrangement.
We additionally define 
\begin{equation*}
\Gamma_{m,1}^{m,M}(\ppp):=\big\{\sss\in \Gamma_m^M(\ppp):s_m=n/p_m \big\},\quad \Gamma_{m,2}^{m,M}(\ppp):=\big\{\sss\in \Gamma_m^M(\ppp):s_m=M \big\}
\end{equation*} and then (\ref{gamma1}) with $l=m$ follows once we prove
\begin{equation}\label{inductionclaim}
\Gamma_{m,1}^{m,M}(\ppp), \Gamma_{m,2}^{m,M}(\ppp)\subset H_m^M(\ppp)
\end{equation}
as $H_{m}^M(\ppp)$ is a convex set.
Therefore, let us prove (\ref{inductionclaim}).
For simplicity, we denote $\ppp^{*m}:=(p_1,\dots,p_{m-1})$ and $\sss^{*m}:=(s_1,\dots,s_{m-1})$ so that $\ppp=(\ppp^{*m},p_m)$ and $\sss=(\sss^{*m},s_m)$.
We observe that
\begin{equation*}
\Gamma_{m,1}^{m,M}(\ppp)=\big\{\sss:\sss^{*m}\in \Gamma_{m-1}^M(\ppp^{*m}),~s_m=n/p_m \big\}.
\end{equation*}
By using the induction hypothesis on $m-1$, we obtain 
\begin{equation*}
\Gamma_{m,1}^{m,M}(\ppp)\subset \big\{\sss:\sss^{*m}\in H_{m-1}^M(\ppp^{*m}),~s_m=n/p_m \big\}
\end{equation*} where the right-hand side is the convex hull of the sets
\begin{equation*}
\big\{\sss:\sss^{*m}\in \Lambda_{m-1}^u(\ppp^{*m}),~s_m=n/p_m \big\}\subset \Lambda_m^u(\ppp), \quad 1\leq u\leq m-1.
\end{equation*}
From the definition of $H_m^M(\ppp)$, it follows that $\Gamma_{m,1}^{m,M}(\ppp)\subset H_m^M(\ppp)$.
Similarly, we have
\begin{equation*}
\Gamma_{m,2}^{m,M}(\ppp)\subset \big\{\sss:\sss^{*m}\in H_{m-1}^M(\ppp^{*m}),~s_m=M \big\}\subset H_m^M(\ppp)
\end{equation*} because $M>n/p_m$ is sufficiently large. This proves (\ref{inductionclaim}).
\end{proof}

Now we prove Proposition \ref{propo4} by induction.

Assume $l=1$ and treat only the case
\begin{equation*}
1<p_1<r,\quad 0<p_2,\dots,p_{\rho}\leq 1, \quad r\leq p_{\rho+1},\dots,p_m\leq \infty.
\end{equation*}
In this case,   condition (\ref{minimal}) is equivalent to
\begin{equation*}
s_1, s_{\rho+1},\dots,s_m>n/r, \quad \text{ and }\quad \sum_{k\in J}{\big({s_k}/{n}-{1}/{p_k}\big)}>-{1}/{r'}
\end{equation*}  for all nonempty subsets $J\subset \{1,\dots,\rho\}$.
Then Lemma \ref{interpolation2} yields that $\sss$ satisfying the above conditions belongs to one of the following sets
\begin{align*}
\mathfrak{S}_u&:=\{\sss: s_u>n/p_u-n/r', \quad s_i>n/p_i ~\text{ for }~i\not= u, ~ 1\leq i \leq \rho \} \\
&\qquad \qquad \qquad \qquad \qquad \qquad \qquad \qquad \cap \{\sss:s_1,s_{\rho+1},\dots,s_m>n/r\}, \quad 1\leq u\leq \rho, \\
\mathfrak{S}_0&:=\{\sss=\theta_1\sss_1+\dots+\theta_{\rho}\sss_{\rho}: \theta_1+\dots+\theta_{\rho}=1,~ 0<\theta_i<1,~\sss_i\in\mathfrak{S}_i,~1\leq i\leq \rho  \}.
\end{align*}
It suffices to show that for $1\leq u\leq \rho$,  $\sss\in \mathfrak{S}_u$ implies (\ref{boundresult})  because the case when $\sss\in\mathfrak{S}_0$ can be proved by using Lemma \ref{interpolation} at most $\rho-1$ times.
If $\sss\in\mathfrak{S}_1$, then the assumptions in Lemma \ref{interpolation} hold with 
\begin{equation*}
(p_1^0,\dots,p_m^0)=(1,p_2,\dots,p_m), \quad (s_1^0,\dots,s_m^0)=(s_1,\dots,s_m)
\end{equation*} and
\begin{equation*}
(p_1^1,\dots,p_m^1)=(r,p_2,\dots,p_m), \quad (s_1^0,\dots,s_m^0)=(s_1,\dots,s_m),
\end{equation*}
due to Proposition \ref{propo1}, \ref{propo2}, and \ref{propo3}, and now (\ref{boundresult}) follows from Lemma \ref{interpolation}. Note that $s_1>n/r=n-n/r'$.

If $\sss\in \mathfrak{S}_u$ for $2\leq u\leq \rho$, then we choose $0<\theta<1$ such that 
\begin{equation*}
s_1>n/p_1=(1-\theta) n+\theta n/r.
\end{equation*} 
We also select $t^0>n$ and $t^1>n/r$ satisfying $s_1=(1-\theta)t^0+\theta t^1$.
Then we interpolate between the two cases 
\begin{equation*}
(p_1^0,\dots,p_m^0)=(1,p_2,\dots,p_m), \quad (s_1^0,\dots,s_m^0)=(t^0,s_2,\dots,s_m)
\end{equation*}
 and 
 \begin{equation*}
 (p_1^1,\dots,p_m^1)=(r,p_2,\dots,p_m), \quad (s_1^1,\dots,s_m^1)=(t^1,s_2,\dots,s_m)
 \end{equation*} 
 using Lemma \ref{interpolation}.
 Here, the assumptions in Lemma \ref{interpolation} with the above two cases follow from Proposition \ref{propo1}, \ref{propo2}, and \ref{propo3}.
 This finally yields (\ref{boundresult}).

We now consider the cases $l\geq 2$ and suppose, by induction, that the claimed assertion holds for $|\mathfrak{L}|=l-1$.
Without loss of generality, we may assume that
$1<p_1,\dots,p_l<r$, $0<p_{l+1},\dots,p_{\rho}\leq 1$, and $r\leq p_{\rho+1},\dots,p_m\leq \infty$, and accordingly, 
we have
\begin{equation*}
s_1,\dots,s_l,s_{\rho+1},\dots,s_m>n/r, \quad \text{ and } \quad \sum_{k\in J}{\big({s_k}/{n}-{1}/b{p_k} \big)}>-{1}/{r'}
\end{equation*} for any nonempty subset $J\subset \{1,\dots,\rho\}$. Similarly as in the case $l=1$, we need to handle only the case that for $1\leq u\leq \rho$
\begin{equation*}
s_1,\dots,s_l,s_{\rho+1},\dots,s_m>n/r,\qquad  s_u>n/p_u-n/r', ~s_i>n/p_i ~\text{ for }~ i\not= u, ~1\leq i\leq \rho.
\end{equation*}
Since $l\geq 2$, we may choose $1\leq v\leq l$ such that $v\not= u$. 
Clearly, 
\begin{equation}\label{svcondition}
s_v>n/p_v~ (>n/r)
\end{equation} since $1<p_v<r$, and $s_u>\max{\big( n/p_u-n/r',n/r \big)}$.
Let $0<\theta<1$ be the number satisfying $1/p_v=(1-\theta)+\theta/r$ and then there exist $t^0>n$ and $t^1>n/r$ so that $s_v=(1-\theta)t^0+\theta t^1$ because of (\ref{svcondition}).
We apply the induction hypothesis to obtain the boundedness with 
\begin{equation*}
(p_1^0,\dots,p_m^0)=(p_1,\dots,p_{v-1},1,p_{v+1},\dots,p_m), ~ (s_1^0,\dots,s_m^0)=(s_1,\dots,s_{v-1},t^0,s_{v+1},\dots,s_m)
\end{equation*}
and another one with
\begin{equation*}
(p_1^1,\dots,p_m^1)=(p_1,\dots,p_{v-1},r,p_{v+1},\dots,p_m), ~ (s_1^0,\dots,s_m^0)=(s_1,\dots,s_{v-1},t^1,s_{v+1},\dots,s_m).
\end{equation*}
Since these are the assumptions in Lemma \ref{interpolation}, (\ref{boundresult}) holds as a result of the lemma.
This completes the proof of Proposition \ref{propo4}.

\section{Proofs of the key lemmas}\label{prooflemmas}

\subsection{Proof of Lemma \ref{nmmultiplier}}
Let $1\leq l\leq m$. The sufficiently large number $M>0$ shall be chosen later.
We utilize an argument of the Marcinkiewicz multiplier theorem. Indeed, we will actually show that for any multi-indices $\alpha_{(1)}$,\dots, $\alpha_{(l)}$ in $\zz^n$ with $|\alpha_{(j)}|\leq n/2+1$ for $1\leq j\leq l$,
\begin{equation}\label{marcin}
\big|\partial_1^{\alpha_{(1)}}\cdots\partial_l^{\alpha_{(l)}}\mathcal{N}_{(M)}(y_1,\dots,y_l) \big|\lesssim_{\alpha_{(1)},\dots,\alpha_{(l)}}|y_1|^{-|\alpha_{(1)}|}\cdots |y_l|^{-|\alpha_{(l)}|}.
\end{equation}
We first observe that
\begin{align}\label{multiplierpoint}
\big|\partial_1^{\alpha_{(1)}}\cdots\partial_l^{\alpha_{(l)}} \mathcal{N}_{(M)}(y_1,\dots,y_l)\big|& \lesssim \sum_{u=1}^{l}\sum_{\alpha^1_{(u)}+\dots+\alpha^{l+1}_{(u)}=\alpha_{(u)}} \Big|  \partial_1^{\alpha^1_{(1)}}\cdots \partial_l^{\alpha^1_{(l)}} \Big\langle \frac{1}{l}\sum_{k=1}^{l}y_k \Big\rangle^{s_1}\Big|\\
&\qquad \quad  \times  \bigg(\prod_{j=2}^{l}\Big| \partial_1^{\alpha_{(1)}^j}\cdots\partial_l^{\alpha_{(l)}^{j}}\Big\langle \frac{1}{l}\sum_{k=1}^{l}{(y_k-y_j)} \Big\rangle^{s_j}\Big|\bigg)\nonumber\\
&\qquad \quad \times \Big|\partial_1^{\alpha^{l+1}_{(1)}}{\langle y_1\rangle^{-(s_1+\dots+s_l)}} \Big|  \bigg( \prod_{j=2}^{l}\Big| \partial_j^{\alpha_{(j)}^{l+1}} {\langle y_j\rangle^{-M}}\Big| \bigg).\nonumber
\end{align}

Using the argument in \cite[p450]{Gr0}, we obtain that
\begin{align*}
\Big|  \partial_1^{\alpha^1_{(1)}}\cdots \partial_l^{\alpha^1_{(l)}} \Big\langle \frac{1}{l}\sum_{k=1}^{l}y_k \Big\rangle^{s_1}\Big|&\lesssim \Big\langle \frac{1}{l}\sum_{k=1}^{l}y_k \Big\rangle^{s_1-(|\alpha_{(1)}^{1}|+\dots+|\alpha_{(l)}^{1}|)},\\
\Big| \partial_1^{\alpha_{(1)}^j}\cdots\partial_l^{\alpha_{(l)}^{j}}\Big\langle \frac{1}{l}\sum_{k=1}^{l}{(y_k-y_j)}  \Big\rangle^{s_j}\Big|&\lesssim \Big\langle \frac{1}{l}\sum_{k=1}^{l}{(y_k-y_j)}  \Big\rangle^{s_j-(|\alpha_{(1)}^j|+\dots+|\alpha_{(l)}^{j}|)},\\
 \Big|\partial_1^{\alpha^{l+1}_{(1)}}{\langle y_1\rangle^{-(s_1+\dots+s_l)}} \Big|&\lesssim {\langle y_1\rangle^{-(s_1+\dots+s_l+|\alpha_{(1)}^{l+1}|)}},\\
\Big| \partial_j^{\alpha_{(j)}^{l+1}} \langle y_j\rangle^{-M}\Big|&\lesssim {\langle y_j\rangle^{-(M+|\alpha_{(j)}^{l+1}|)}}.
\end{align*}
We choose a positive number $N$ such that  $M>N+n+2>s_1+\dots+s_l+n+2$.
Since
\begin{equation*}
\dfrac{\Big\langle \frac{1}{l}\sum_{k=1}^{l}y_k \Big\rangle^{s_1}  \prod_{j=2}^{l} \Big\langle \frac{1}{l}\sum_{k=1}^{l}{(y_k-y_j)} \Big\rangle^{s_j}}{\langle y_1\rangle^{s_1+\dots+s_l}\prod_{j=2}^{l}{\langle y_j\rangle^{N}} }\lesssim 1,
\end{equation*} 
we finally obtain that the right-hand side of (\ref{multiplierpoint}) is dominated by a constant times the product of
\begin{align*}
I_1:={ \Big\langle \frac{1}{l}\sum_{k=1}^{l}y_k \Big\rangle^{-(|\alpha_{(1)}^{1}|+\dots+|\alpha_{(l)}^{1}|)}},&\qquad I_2:=\prod_{j=2}^{l} {\Big\langle \frac{1}{l}\sum_{k=1}^{l}{(y_k-y_j)} \Big\rangle^{-(|\alpha_{(1)}^{j}|+\dots+|\alpha_{(l)}^j|)}},  \\ 
I_3:={\langle y_1\rangle^{-|\alpha_{(1)}^{l+1}|}},& \qquad I_4:=\prod_{j=2}^{l}{\langle y_j\rangle^{-(M-N+|\alpha_{(j)}^{l+1}|)}}.
\end{align*}

If $|y_1|>2l|y_j|$ for all $2\leq j\leq l$, then
\begin{equation*}
I_1\lesssim { \langle y_1 \rangle^{-|\alpha_{(1)}^{1}|}}\quad\text{and}\quad I_2\lesssim \prod_{j=2}^{l} {\langle y_1\rangle^{-|\alpha_{(1)}^{j}|}},  
\end{equation*} which implies that
\begin{equation*}
I_1\times I_2\times I_3\times I_4 \lesssim {\langle y_1\rangle^{-|\alpha_{(1)}|}}\prod_{j=2}^{l}{\langle y_j\rangle^{-(M-N)}}\lesssim |y_1|^{-|\alpha_{(1)}|}\cdots |y_l|^{-|\alpha_{(l)}|}
\end{equation*}
for $M-N>n+2>n/2+1$.

Now assume that $|y_1|\leq 2l \max{(|y_2|,\dots,|y_l|)}$ and actually, only the case $|y_1|\leq 2l |y_2|$ will be considered. In that case, we see that
\begin{align*}
I_1\times I_2\times I_3\times I_4\leq I_4&\lesssim {\langle y_1\rangle^{-|\alpha_{(1)}|}} {\langle y_2\rangle^{-(M-N-|\alpha_{(1)}|)}}\prod_{j=3}^{l}{\langle y_j\rangle^{-(M-N)}}\\
&\lesssim  |y_1|^{-|\alpha_{(1)}|}\cdots |y_l|^{-|\alpha_{(l)}|}
\end{align*} for $M-N\geq M-N-|\alpha_{(1)}|>n+2-|\alpha_{(1)}|\geq |\alpha_{(j)}|$ for $2\leq j\leq l$.

This proves (\ref{marcin}).

 \subsection{Proof of Lemma \ref{keylemma1}}
 
   Without loss of generality, we assume that $J_0=\{1,\ldots,v\}$ for some $1\le v\leq l$, and $\Vert{f_i}\Vert_{L^{\infty}(\rn)}=1$ for all $l+1\leq i\leq m$. 
   Fix
   \begin{equation*}
   x\in \Big(\bigcap_{i=v+1}^lQ_i^*\Big)\setminus \Big(\bigcup_{i=1}^{v}Q_i^*\Big) .
   \end{equation*}
   (When $v=l,$ just fix $x\in \mathbb R^n\setminus (\bigcup_{i=1}^lQ_i^*)$.)
   Now we   write
    \begin{equation*}
    T_{\sigma}(a_1,\ldots,a_l, f_{l+1},\ldots,f_m)(x) = \sum_{j\in\mathbb Z}g_j(x),
    \end{equation*}
    where
 \begin{equation}\label{gjx}
 g_j(x):= \int_{(\rn)^m}2^{jmn} K_j\big(2^j(x-y_1), \ldots, 2^j(x-y_m)\big)\Big( \prod_{i=1}^{l}a_i(y_i)\Big)\Big( \prod_{i=l+1}^{m} f_{i}(y_{i})\Big) d\yyy
    \end{equation}
    with $K_j:= \big(\sigma(2^j\cdot)\wh{\Psi^{(m)}} \big)^{\vee}$.
    Let $c_i$ be the center of the cube $Q_i$ $(1\le i\le l)$.
    For $1\le i\le v$, since $x\notin Q_i^*$, we have ${|{x-c_i}|}\approx {|{x-y_i}|}$ for all $y_i\in Q_i$.
    Fix $1\le k\le v$ and for $1\leq u\leq w\leq m$ denote
    \begin{equation*}
    K_j^{(u,w)}(x,\yyy):=K_j\big(y_1,\dots,y_{u-1},2^j(x-y_u),\dots,2^j(x-y_{w}),y_{w+1},\dots,y_m\big)
    \end{equation*}
    for convenience of notation.
 We see that
    \begin{align*}
      &\Big( \prod_{i=1}^{v}\langle2^j(x-c_i)\rangle^{s_i}\Big){|{g_j(x)}|}\\
      & \lesssim 2^{jmn} \Big(\prod_{i=1}^l\Vert{a_i}\Vert_{L^{\infty}(\rn)}\Big) \int_{Q_1\times\cdots\times{Q_{l}}\times (\rn)^{m-l}} \Big(\prod_{i=1}^{v}\langle2^j(x-y_i)\rangle^{s_i}\Big){\big|{K_j^{(1,m)}\big(x,\yyy\big)}\big|}d\yyy 
     \end{align*}
  and the integral in the preceding expression is less than
  \begin{align*}
  &\int_{Q_1\times\cdots\times Q_v\times (\rn)^{m-v}}{\Big(\prod_{i=1}^{v}\langle 2^j(x-y_i)\rangle^{s_i} \Big)\big|K_j^{(1,m)}\big(x,\yyy\big) \big|    }d\yyy\\
  &=2^{-jn(m-v)}\int_{Q_1\times\cdots\times Q_v\times (\rn)^{m-v}}{\Big(\prod_{i=1}^{v}\langle 2^j(x-y_i)\rangle^{s_i} \Big) \big|K_j^{(1,v)}\big(x,\yyy\big) \big|    }d\yyy\\
  &\leq 2^{-jn(m-v)}\Big( \prod_{i=1}^{v}{|Q_i|}\Big)\int_{(\rn)^{m-v}}{\Big[\int_{y_k\in Q_k}{|Q_k|^{-1}\langle 2^j(x-y_k)\rangle^{s_k}}} \\
  &\qquad \times {{\Big\Vert  \Big( \prod_{\substack{i=1\\i\not= k}}^{v}\langle y_i\rangle^{s_i}\Big)K_j^{(k,k)}\big(x,\yyy\big)    \Big\Vert_{L^{\infty}(y_1,\dots,y_{k-1},y_{k+1},\dots,y_v)}}dy_k \Big]}dy_{v+1}\cdots dy_m.
  \end{align*}   
Using Lemma \ref{lem:LInfL2}, the integral in the last expression is majored by a constant multiple of
\begin{align*}
&2^{-jn(m-v)}\Big( \prod_{i=1}^{v}{|Q_i|}\Big)\int_{y_k\in Q_k}{|Q_k|^{-1}\langle 2^j(x-y_k)\rangle^{s_k}} \\
  &\qquad  \times {\Big[ \int_{(\rn)^{m-v}}}{{\Big\Vert  \Big( \prod_{\substack{i=1\\i\not= k}}^{v}\langle y_i\rangle^{s_i}\Big)K_j^{(k,k)}\big(x,\yyy\big)    \Big\Vert_{L^{r'}(y_1,\dots,y_{k-1},y_{k+1},\dots,y_v)}} }dy_{v+1}\cdots dy_m \Big]dy_k
\end{align*} 
and this is further estimated by
\begin{align*}
&2^{-jn(m-v)}\Big( \prod_{i=1}^{v}|Q_i|\Big)\int_{y_k\in Q_k}{|Q_k|^{-1}\langle 2^j(x-y_k)\rangle^{s_k}}   \\
&\qquad \qquad \qquad \times  {\Big\Vert   \Big( \prod_{\substack{i=1\\i\not= k}}^{m}\langle y_i\rangle^{s_i}\Big)K_j^{(k,k)}\big(x,\yyy\big)    \Big\Vert_{L^{r'}(y_1,\dots,y_{k-1},y_{k+1},\dots,y_m)}     }dy_k ,
\end{align*} 
by applying   H\"older's inequality, as $s_i>n/r$ for $v+1\leq i\leq m$.
We finally obtain
\begin{equation}\label{noderivativeest}
\Big( \prod_{i=1}^{v}\langle 2^j(x-c_i)\rangle^{s_i}\Big)|g_j(x)| \leq 2^{jvn}\Big(\prod_{i=1}^{v}{|Q_i|^{1-1/p_i}}\Big)\Big(\prod_{i=v+1}^{l}{b_i(x)} \Big)h_j^{(k,0)}(x), 
\end{equation}
where  $b_i(x):= {\left\vert{Q_i}\right\vert}^{-1/{p_i}}\chi_{Q_i^*}(x)$ for $v+1\le i\le l$ and
 \begin{equation*}
    h_j^{(k,0)}(x):=  \dfrac1{{|{Q_k}|}}\int_{Q_k}\langle 2^j(x-y_k)\rangle^{s_k} {\Big\Vert   \Big( \prod_{\substack{i=1\\i\not= k}}^{m}\langle y_i\rangle^{s_i}\Big)K_j^{(k,k)}\big(x,\yyy\big)    \Big\Vert_{L^{r'}(y_1,\dots,y_{k-1},y_{k+1},\dots,y_m)}     }dy_k.
    \end{equation*}
 The functions $b_i$, $v+1\le i\le l$, obviously satisfy the estimate $\Vert{b_i}\Vert_{L^{p_i}(\rn)} \lesssim 1$, and the Minkowski inequality with $1<r'<\infty$ gives
   \begin{align}\label{hr'}
    \big\Vert{h_j^{(k,0)}}\big\Vert_{L^{r'}(\rn)}&\leq \frac{1}{|Q_k|}\int_{Q_k}{\Big(\int_{\rn}{\langle 2^j(x-y_k)\rangle^{r's_k}\Big\Vert \cdots \Big\Vert_{L^{r'}((\rn)^{m-1})}^{r'}}dx \Big)^{1/r'}}dy_k\nonumber\\
    &\lesssim 2^{-jn/r'}\Big\Vert \Big( \prod_{i=1}^{m}{\langle \cdot_i\rangle^{s_i}\Big)K_j}\Big\Vert_{L^{r'}((\rn)^m)}\lesssim 2^{-jn/r'}\mathcal{L}_{\sss}^{r,\Psi^{(m)}}[\sigma]
    \end{align}
    where we made use of  a change of variables and applied the Hausdorff-Young inequality in the 
    preceding inequalities.

    On the other hand, using the vanishing moment condition of $a_k$ and Lemma \ref{smalllemma}, we write
    \begin{align*}
    \big|g_j(x)\big| &\lesssim 2^{jmn}\sum_{|\alpha|=N_k+1}\int_0^1{\int_{(\rn)^m}{\big( 2^j|y_k-c_k|\big)^{N_k+1}}}\Big( \prod_{i=1}^{l}{|a_i(y_i)|}\Big)\\
    &\times {{{\Big|\partial_k^{\alpha}K_j\big(2^j(x-y_1),\dots,2^j(x-y_{k-1}),2^jx_{c_k,y_k}^{t},2^j(x-y_{k+1}),\dots,2^j(x-y_m)\big) \Big|} }d\yyy}dt    
    \end{align*}
    where
    $x_{c_k,y_k}^t:=x-c_k-t(y_k-c_k)$ and $\partial^{\alpha}_{k}K_j (z_1, \dots, z_m):= \partial_{z_k}^{\alpha} K_j (z_1, \dots, z_m)$.
    Notice that $|{x_{c_k,y_k}^t}|\approx |{x-c_k}|$ for $x\not\in Q_k^{\ast}$, $y_k\in Q_k$, and $0<t<1$.
Repeating the preceding argument that is used to establish (\ref{noderivativeest}), we also obtain
    \begin{equation}\label{equ:H1NFunc}
   \Big(  \prod_{i=1}^{v}\langle 2^j(x-c_i)\rangle^{s_i}\Big){|{g_j(x)}|}\lesssim 2^{jvn}\Big(\prod_{i=1}^{v}{|{Q_i}|}^{1-1/{p_i}} \Big)\Big(\prod_{i=v+1}^lb_i(x)\Big) h_j^{(k,1)}(x),
    \end{equation}
    where $b_i(x):=|Q_i|^{-1/p_i}\chi_{Q_i^*}(x)$ as before, and
    \begin{align*}
    h_j^{(k,1)}(x)&:=\big( 2^jl(Q_k)\big)^{N_k+1}\sum_{|\alpha|=N_k+1}  \dfrac1{{|{Q_k}|}}\int_0^1\int_{Q_k}\langle 2^jx_{c_k,y_k}^{t}\rangle^{s_k}\\
       & \quad \times \Big\Vert{\Big( \prod_{\substack{i=1\\i\ne k}}^{m}\langle\cdot_i\rangle^{s_i} \Big)\partial_k^{\alpha}K_j\big(\cdot_1,\ldots,\cdot_{k-1},2^jx_{c_k,y_k}^{t},\cdot_{k+1}, \dots,\cdot_m\big)}\Big\Vert_{L^{r'}((\rn)^{m-1})}dy_k dt.
    \end{align*}
Now Minkowski's inequality and Lemma \ref{lem:LInfL2} yield that
\begin{equation}\label{hr1'}
    \big\Vert{h_j^{(k,1)}}\big\Vert_{L^{r'}(\rn)}\lesssim 2^{-\frac{jn}{r'}}\big(2^jl(Q_k)\big)^{N_k+1}\mathcal{L}_{\sss}^{r,\Psi^{(m)}}[\sigma].
\end{equation} which is the counterpart of (\ref{hr'}) for $h_j^{(k,1)}$.

    Combining \eqref{noderivativeest} and \eqref{equ:H1NFunc}, we obtain
    \begin{equation}\label{equ:GjQ2}
  |g_j(x)| \lesssim 2^{jvn}\Big(\prod_{i=1}^{v}{|Q_i|}^{1-1/{p_i}}\langle2^j(x-c_i)\rangle^{-s_i}\Big)\Big(\prod_{i=v+1}^l b_i(x)\Big)\min\Big( h_j^{(k,0)}(x),h_j^{(k,1)}(x)\Big)
    \end{equation}
    for all $x\in \big(\bigcap_{i=v+1}^lQ_i^*\big)\setminus \big(\bigcup_{i=1}^{v}Q_i^*\big)$ and all $1\le k\le v.$

    Now we will construct nonnegative functions $u_{i,j}$ for $1\leq i \leq v$ such that
\begin{equation*}
|{g_j(x)}|\lesssim \mathcal{L}_{\sss}^{r,\Psi^{(m)}}[\sigma]\Big( \prod_{i=1}^{v}u_{i,j}(x)\Big)\Big( \prod_{i=v+1}^lb_i(x)\Big)
\end{equation*}
for all  $x\in \big(\bigcap_{i=v+1}^lQ_i^*\big)\setminus \big(\bigcup_{i=1}^{v}Q_i^*\big)$ and 
\begin{equation}\label{ujkest}
\Big\Vert{\sum_{j\in\zz}u_{i,j}}\Big\Vert_{L^{p_i}(\rn)}\lesssim 1.
\end{equation}
Then the lemma follows by taking 
\begin{equation}
b_i:=\sum_{j\in\zz}u_{i,j} \quad 1\leq i\leq v.
\end{equation} 
For this, we choose $\lambda_{i}$, $1\le i\le v$, such that
 \begin{equation*}
    0\leq \lambda_i<{1}/{r'}, \quad {s_i}/{n}>{1}/{p_i}-{1}/{r'}+\lambda_i, \quad \sum_{i=1}^{v}\lambda_i= {(v-1)}/{r'}.
 \end{equation*}
This is possible since the second condition in (\ref{conditionss}), with $J\subset \{1,\dots,v\}$, yields
\begin{equation*}
\sum_{i=1}^{v}{\min{\big( 0,{s_i}/{n}-{1}/{p_i} \big)}}>-{1}/{r'},
\end{equation*}
which further implies
 \begin{equation*}
    \sum_{i=1}^v\min\big({1}/{r'},{s_i}/{n}-{1}/{p_i}+{1}/{r'}\big)>{(v-1)}/{r'}.
 \end{equation*}
We set  
\begin{equation}\label{betak}
\alpha_{i}:= {1}/{p_i}-{1}/{r'}+\lambda_i \quad \text{and} \quad \beta_{i}:= 1-r'\lambda_i.
\end{equation}
    Then we have $\alpha_i >0$, $\beta_i >0$, and $\sum_{i=1}^{v}\beta_i =1$.
Letting
    \begin{equation*}
    u_{i,j}(x):= \big( \mathcal{L}_{\sss}^{r,\Psi^{(m)}}[\sigma]\big)^{-\beta_i}2^{jn}{\left\vert{Q_i}\right\vert}^{1-1/{p_i}}\langle2^j(x-c_i)\rangle^{-s_i}\chi_{(Q_i^*)^c}(x)
 \min\Big( h_j^{(i,0)}(x),h_j^{(i,1)}(x)\Big)^{\beta_i}
    \end{equation*} for $1\leq i\leq v$, it is easy to see, from \eqref{equ:GjQ2}, that
   \begin{equation}\label{gjpointest}
    |g_j(x)|\lesssim \mathcal{L}_{\sss}^{r,\Psi^{(m)}}[\sigma]\Big( \prod_{i=1}^{v}u_{i,j}(x)\Big)\Big(\prod_{i=v+1}^lb_i(x)\Big)
    \end{equation}
     for all  $x\in (\cap_{i=v+1}^lQ_i^*)\setminus (\cup_{i=1}^{v}Q_i^*).$
     
It remains to verify (\ref{ujkest}).
    Since $1/{p_i} = \alpha_i+{\beta_i}/{r'},$   H\"older's inequality yields
    \begin{align*}
   \big\Vert{u_{i,j}}\big\Vert_{L^{p_i}(\rn)}&\leq \big(\mathcal{L}_{\sss}^{r,\Psi^{(m)}}[\sigma]\big)^{-\beta_i}2^{jn}{|{Q_i}|}^{ 1- 1/{p_i}}\big\Vert{\langle 2^j(\cdot-c_i)\rangle^{-s_i}\chi_{(Q_i^*)^c}}\big\Vert_{L^{{1}/{\alpha_i}}(\rn)}\\
    &\qquad \qquad \times \min{\Big(\big\Vert h_j^{(i,0)}\big\Vert_{L^{r'}(\rn)}^{\beta_i},  \big\Vert h_j^{i,1} \big\Vert_{L^{r'}(\rn)}^{\beta_i}\Big)}.
    \end{align*}
    We observe that
\begin{equation*}
 \big\Vert{\langle 2^j(\cdot-c_i)\rangle^{-s_i}\chi_{(Q_i^*)^c}}\big\Vert_{L^{1/\alpha_i}(\rn)} \approx 2^{-jn\alpha_i}\min\big(1,\big(2^jl(Q_i)\big)^{-(s_i-\alpha_in)}\big)
\end{equation*}
since ${s_i}/{\alpha_i}>n$.
In addition, it follows from (\ref{hr'}) and (\ref{hr1'}) that
\begin{equation*}
 \min{\Big(\big\Vert h_j^{(i,0)}\big\Vert_{L^{r'}(\rn)}^{\beta_i},  \big\Vert h_j^{(i,1)} \big\Vert_{L^{r'}(\rn)}^{\beta_i}\Big)}\lesssim 2^{-jn\beta_i/r'}\big(\mathcal{L}_{\sss}^{r,\Psi^{(m)}}[\sigma] \big)^{\beta_i}\min{\big(1,\big( 2^jl(Q_i)\big)^{\beta_i(N_i+1)} \big)}.
\end{equation*}
 In conclusion, we have
\begin{equation}\label{uijest}
\big\Vert{u_{i,j}}\big\Vert_{L^{p_i}(\rn)}\lesssim \begin{cases}
{  (2^{j}l(Q_i) )^{- (n/p_i-n) +  \beta_i(N_i+1)}  }, & { \quad \text{if} \quad  2^{j} l(Q_i) \le 1  }\\
{  (2^{j}l(Q_i) )^{- (n/p_i-n) -(s_i- \alpha_i n)}  }, & {  \quad  \text{if}  \quad   2^{j} l(Q_i) > 1 }.
\end{cases}
\end{equation}
We choose $N_i$ sufficiently large so that $- (n/p_i-n) +  \beta_i(N_i+1)>0$, and then (\ref{ujkest}) follows immediately.

The proof of Lemma \ref{keylemma1} is done.

\subsection{Proof of Lemma \ref{keylemma2}}
It follows from (\ref{conditionss}) that there exists $1<t<r$ such that 
\begin{equation}\label{stot2}
s_1,\dots,s_m>n/t>n/r, \qquad \sum_{k\in J}{\big({s_k}/{n}-{1}/{p_k} \big)}>-{1}/{t'}>-{1}/{r'}
\end{equation} for every nonempty subset $J\subset \{1,\dots,l\}$. Then (\ref{compactembedding1}) holds.

For each $J_0\subset \mathrm{I}$, let 
\begin{equation*}
E_{J_0}:=\Big( \bigcap_{i\in \mathrm{I}\setminus J_0}Q_i^*\Big)\setminus \Big( \bigcup_{i\in J_0}{Q_i^*}\Big)
\end{equation*} where $E_{\emptyset}=\bigcap_{i\in \mathrm{I}}Q_i^*$ for $J_0=\emptyset$, and $E_{\mathrm{I}}=\big( \bigcup_{i\in \mathrm{I}}{Q_i^*}\big)^c$ for $J_0=\mathrm{I}$.
Then we see that the left-hand side of (\ref{kl2condition1}) can be decomposed as
\begin{align*}
\sum_{J_0\subset \mathrm{I}}\Big( \sum_{j\in\zz}\big|T_{\sigma_{j,1}^{\kappa_1,\kappa_2}}\big( a_{1},\dots,a_{l},f_{l+1},\dots,f_m\big)(x)\big|\Big)\chi_{E_{J_0}}(x).
\end{align*}
Since it is a finite sum over $J_0$, it suffices to show that for each $J_0\subset \mathrm{I}$, there exist functions $b_i^{J_0}$, $1\leq i\leq l$, and $F_i^{J_0}$, $l+1\leq i\leq \rho$ having the properties that
for $x\in E_{J_0}$
\begin{equation}\label{kl22condition1}
 \sum_{j\in\zz}\big|T_{\sigma_{j,1}^{\kappa_1,\kappa_2}}\big( a_{1},\dots,a_{l},f_{l+1},\dots,f_m\big)(x)\big| \lesssim \mathcal{L}_{\sss}^{r}[\sigma]\Big( \prod_{i\in\mathrm{I}}b_i^{J_0}(x)\Big)\Big( \prod_{i\in\mathrm{II}}F_i^{J_0}(x)\Big),
\end{equation}
\begin{equation}\label{kl22condition2}
\Vert b_i^{J_0}\Vert_{L^{p_i}(\rn)}\lesssim 1,\quad \text{ for }~ i\in\mathrm{I},
\end{equation}
\begin{equation}\label{kl22condition3}
\Vert F_i^{J_0}\Vert_{L^{p_i}(\rn)}\lesssim \Vert f_i\Vert_{L^{p_i}(\rn)}, \quad \text{ for }~ i\in \mathrm{II}.
\end{equation}

\hfill

We first consider the case $J_0=\emptyset$ and divide the proof into six cases  based on the location of $\kappa_1$ and $\kappa_2$. Let $x\in E_{J_0}$.

{\bf Case1 : $\kappa_1,\kappa_2\in \mathrm{I}.$} 
By applying (\ref{2express}), Lemma \ref{keyestilemma}, (\ref{katogeneral}), (\ref{compactembedding1}), and (\ref{maximalcompare}),  we have
\begin{align*}
&\big| T_{\sigma_{j,2}^{\kappa_1,\kappa_2}}\big(a_1,\dots,a_l,f_{l+1},\dots,f_m \big)(x)\big|\\
&\lesssim\mathcal{L}_{\sss}^{r}[\sigma]\mathcal{M}_t(a_{\kappa_1})_j(x)\mathcal{M}_t(a_{\kappa_2})_j(x)\Big( \prod_{i\in\mathrm{I}\setminus \{\kappa_1,\kappa_2\}}\mathcal{M}_ta_i(x)\Big)\Big(\prod_{i\in \mathrm{II}}{\mathcal{M}_tf_i(x)} \Big),
\end{align*} since $\mathcal{M}_tf_i(x)\leq \Vert f_i\Vert_{L^{\infty}(\rn)}=1$ for $i\in\mathrm{III}$.
Now we take the sum over $j\in\zz$ to both sides and apply the Cauchy-Schwarz inequality.
Then (\ref{kl22condition1}) follows from taking
\begin{align*}
b_i^{J_0}(x)&:=\Big(\sum_{j\in\zz} \big( \mathcal{M}_t(a_i)_j(x)\big)^2 \Big)^{1/2}\chi_{Q_i^*(x)},\qquad  i\in \{\kappa_1,\kappa_2\}\\
b_i^{J_0}(x)&:=\mathcal{M}_ta_i(x)\chi_{Q_i^*}(x), \qquad i\in\mathrm{I}\setminus \{\kappa_1,\kappa_2\},\\
F_i^{J_0}(x)&:=\mathcal{M}_tf_i(x), \qquad   i\in \mathrm{II}.
\end{align*}
Moreover, using   H\"older's inequality, (\ref{maximal1}) with $t<2$, and (\ref{littlewood}), we obtain
\begin{equation*}
\Vert b_i^{J_0}\Vert_{L^{p_i}(\rn)}\leq |Q_i^*|^{1/p_i-1/2}\big\Vert \big\{ \mathcal{M}_t(a_i)_j\big\}_{j\in\zz}\big\Vert_{L^2(\ell^2)}\lesssim |Q_i|^{1/p_1-1/2}\Vert a_i\Vert_{L^2(\rn)}\lesssim 1, ~ i\in \{\kappa_1,\kappa_2\},
\end{equation*}
\begin{equation*}
\Vert b_i^{J_0}\Vert_{L^{p_i}(\rn)}\leq |Q_i^*|^{1/p_i-1/2}\big\Vert \mathcal{M}_ta_i\big\Vert_{L^2(\rn)}\lesssim |Q_i|^{1/p_i-1/2}\Vert a_i\Vert_{L^2(\rn)}\lesssim 1, \quad  i\in \mathrm{I}\setminus \{\kappa_1,\kappa_2\},
\end{equation*}
\begin{equation*}
\Vert F_i^{J_0}\Vert_{L^{p_i}(\rn)}\lesssim \Vert f_i\Vert_{L^{p_i}(\rn)},\quad i\in \mathrm{II},
\end{equation*} which completes the proof of (\ref{kl22condition2}) and (\ref{kl22condition3}).

{\bf Case2 : $\kappa_1,\kappa_2\in \mathrm{II}.$}
Similarly, (\ref{kl22condition1}) holds with
\begin{align*}
b_i^{J_0}(x)&:=\mathcal{M}_ta_i(x)\chi_{Q_i^*}(x),\qquad  i\in \mathrm{I},\\
F_{i}^{J_0}(x)&:=\Big(\sum_{j\in\zz}\big(\mathcal{M}_t(f_{i})_j(x)\big)^2\Big)^{1/2}, \qquad i\in\{ \kappa_1,\kappa_2\},\\
F_i^{J_0}(x)&:=\mathcal{M}_tf_i(x), \qquad  i\in \mathrm{II}\setminus \{\kappa_1,\kappa_2\}.
\end{align*}
Obviously, (\ref{kl22condition2}) and (\ref{kl22condition3}) are clear as (\ref{littlewood}) is applied when $i\in \{\kappa_1,\kappa_2\}$.

{\bf Case3 : $\kappa_1,\kappa_2\in \mathrm{III}.$}
In this case, we cannot use the classical Littlewood-Paley theory as $L^{\infty}$ norm is not characterized by $L^{\infty}$ norm of a square function.
Instead, we can benefit from Lemma \ref{bmoboundlemma}, using $\mathfrak{M}_{\sigma,2^j}^t$, not $\mathcal{M}_t$.
By applying (\ref{4express}), Lemma \ref{keyestilemma}, (\ref{katogeneral}), (\ref{compactembedding1}), and (\ref{maximalcompare}), we obtain
\begin{align*}
&\big| T_{\sigma_{j,1}^{\kappa_1,\kappa_2}}\big(a_1,\dots,a_l,f_{l+1},\dots,f_m \big)(x)\big| \lesssim \mathcal{L}_{\sss}^{r}[\sigma]\Big( \prod_{i\in\mathrm{I}} \mathcal{M}_ta_i(x)\Big)  \mathfrak{M}_{s_{l+1},2^j}^t(f_{l+1})^{j+1}(x) \\
&\qquad \qquad \qquad \qquad \qquad\qquad \times   \Big(\prod_{i\in\mathrm{II}\setminus \{l+1\}}{\mathcal{M}_tf_i(x)} \Big)\mathfrak{M}_{s_{\kappa_1},2^j}^{t}(f_{\kappa_1})_j(x)\mathfrak{M}_{s_{\kappa_2},2^j}^{t}(f_{\kappa_2})_j(x).
\end{align*}
Now we take
\begin{align*}
b_i^{J_0}(x)&:=\mathcal{M}_ta_i(x)\chi_{Q_i^*}(x),\qquad  i\in \mathrm{I},\\
F_{l+1}^{J_0}(x)&:=\sum_{j\in\zz}\mathfrak{M}_{s_{l+1},2^j}^t(f_{l+1})^{j+1}(x)\mathfrak{M}_{s_{\kappa_1},2^j}^{t}(f_{\kappa_1})_j(x)\mathfrak{M}_{s_{\kappa_2},2^j}^{t}(f_{\kappa_2})_j(x),\\
F_i^{J_0}(x)&:=\mathcal{M}_tf_i(x), \qquad   i\in \mathrm{II}\setminus \{l+1\}.
\end{align*}
Then (\ref{kl22condition1}), (\ref{kl22condition2}), and (\ref{kl22condition3}) are immediate for $i\not= l+1$, and the case $i=l+1$ follows from Lemma \ref{bmoboundlemma}.

{\bf Case4 : $\kappa_1\in \mathrm{I}$, $\kappa_2\in \mathrm{II}$.}
Using the arguments in {\bf Case1} and {\bf Case2},
we are done with the choices 
\begin{align*}
b_{\kappa_1}^{J_0}(x)&:=\Big(\sum_{j\in\zz} \big( \mathcal{M}_t(a_{\kappa_1})_j(x)\big)^2 \Big)^{1/2}\chi_{Q_{\kappa_1}^*}(x),\\
b_i^{J_0}(x)&:=\mathcal{M}_ta_i(x)\chi_{Q_i^*}(x),\qquad   i\in \mathrm{I}\setminus \{\kappa_1\},\\
F_{\kappa_2}^{J_0}(x)&:=\Big(\sum_{j\in\zz}\big(\mathcal{M}_t(f_{\kappa_2})_j(x)\big)^2\Big)^{1/2},\\ 
F_i^{J_0}(x)&:=\mathcal{M}_tf_i(x), \qquad   i\in \mathrm{II}\setminus \{\kappa_2\}.
\end{align*}

{\bf Case5 : $\kappa_1\in \mathrm{I}$, $\kappa_2\in \mathrm{III}$.}
It follows from  (\ref{3express}), Lemma \ref{keyestilemma},  (\ref{katogeneral}), (\ref{compactembedding1}), and (\ref{maximalcompare}) that
(\ref{kl22condition1}) holds with
\begin{align*}
b_{\kappa_1}^{J_0}(x)&:=\Big(\sum_{j\in\zz} \big( \mathcal{M}_t(a_{\kappa_1})_j(x)\big)^2 \Big)^{1/2}\chi_{Q_{\kappa_1}^*}(x),\\
b_i^{J_0}(x)&:=\mathcal{M}_ta_i(x)\chi_{Q_i^*}(x),\qquad  i\in \mathrm{I}\setminus \{\kappa_1\},\\
F_{l+1}^{J_0}(x)&:=\Big(\sum_{j\in\zz}\big( \mathfrak{M}_{s_{l+1},2^j}^t(f_{l+1})^{j+1}(x)\big)^2\big(\mathfrak{M}_{s_{\kappa_2},2^j}^{t}(f_{\kappa_2})_j(x)\big)^2\Big)^{1/2},\\
F_i^{J_0}(x)&:=\mathcal{M}_tf_i(x), \qquad   i\in \mathrm{II}\setminus \{l+1\},
\end{align*}
and it is clear that (\ref{kl22condition1}), (\ref{kl22condition2}), and (\ref{kl22condition3}) hold. Especially, (\ref{kl22condition3}) for $i=l+1$ is due to Lemma \ref{bmoboundlemma}.

{\bf Case6 : $\kappa_1\in \mathrm{II}$, $\kappa_2\in \mathrm{III}$.}
 The similar arguments can be applied with
\begin{align*}
b_i^{J_0}(x)&:=\mathcal{M}_ta_i(x)\chi_{Q_i^*}(x),\qquad  i\in \mathrm{I},\\
F_{\kappa_1}^{J_0}(x)&:=\sum_{j\in\zz}\mathfrak{M}_{s_{\kappa_1},2^j}^t(f_{\kappa_1})_j(x)\mathfrak{M}_{s_{\kappa_2},2^j}^t(f_{\kappa_2})_j(x), \\
F_i^{J_0}(x)&:=\mathcal{M}_tf_i(x), \quad   i\in \mathrm{II}\setminus \{\kappa_1\}.
\end{align*}
Note that Lemma \ref{bmoboundlemma} implies 
$$\Vert F_{\kappa_1}^{J_0}\Vert_{L^{p_{\kappa_1}}(\rn)}\lesssim \Vert f_{\kappa_1}\Vert_{L^{p_{\kappa_1}}(\rn)}\Vert f_{\kappa_2}\Vert_{BMO}\lesssim \Vert f_{\kappa_1}\Vert_{L^{p_{\kappa_1}}(\rn)}.$$

\hfill

Next we consider the case $J_0\not=\emptyset$. 
In this case the proof is based on the idea in the proof of Lemma \ref{keylemma1}.
For notational convenience, let 
\begin{equation}\label{gjdef}
\mathcal{G}_j:=T_{\sigma_{j,1}^{\kappa_1,\kappa_2}}\big( a_{1},\dots,a_{l},f_{l+1},\dots,f_m\big).
\end{equation} 
Here, the notation $\mathcal{G}_j$ does not contain two parameters $\kappa_1$ and $\kappa_2$ as the arguments below are universal for any $1\leq \kappa_1<\kappa_2\leq m$. We note that $\mathcal{G}_j$ plays a similar role as $g_j$ in (\ref{gjx}).

We shall prove that there exist nonnegative functions $u^{J_0}_{i,j}$, $i\in J_0$, such that for $x\in E_{J_0}$ and $j\in \mathbb Z$,
\begin{equation}\label{equ:019}
\big|\mathcal{G}_j(x)\big| \lesssim \mathcal{L}_{\sss}^{r}[\sigma] \Big(\prod_{i\in J_0} u^{J_0}_{i,j}(x)\Big)\Big(\prod_{i\in \mathrm{I} \setminus J_0} |{Q_{i}}|^{-1/p_i} \chi_{Q_{i}^{\ast}}(x)\Big) \Big(\prod_{i\in \mathrm{II} } \mathcal{M}_{t}(f_{i})(x)\Big),
\end{equation}
 and 
\begin{equation}\label{equ:020}
\Vert{u^{J_0}_{i,j}}\Vert_{L^{p_i}(\rn)} \lesssim \min \big( \big(2^{j}\ell (Q_{i})\big)^{\gamma_i}, \big(2^{j}\ell (Q_{i})\big)^{-\delta_i}\big)
\end{equation}
for some $\gamma_i, \delta_i>0$, which are the counterparts of (\ref{gjpointest}) and (\ref{uijest}), respectively.

If we have such functions $u^{J_0}_{i,j}$, then \eqref{kl22condition1} holds with the functions
\begin{equation}\label{bjidef}
b^{J_0}_{i}:=\sum_{j\in \zz}u^{J_0}_{i,j} \quad\text{ for }~ i\in J_0, \qquad  \qquad b^{J_0}_{i}= |{Q_{i}}|^{-1/p_i}\chi_{Q_{i}^*} \quad \text{ for }~ i\in \mathrm{I} \setminus J_0,
\end{equation}
\begin{equation}\label{fiidef}
F^{J_0}_{i}:=\mathcal{M}_t(f_i) \qquad\text{ for }\quad i\in \mathrm{II}.
\end{equation}
The estimate (\ref{kl22condition2}) for  $i\in \mathrm{I} \setminus J_0$ is obvious and when $i\in J_0$ it follows from (\ref{equ:020}).
 In addition, (\ref{kl22condition3}) holds via the $L^{p_i}$-boundedness of $\mathcal{M}_{t}$.

From now on, let us construct $u_{i,j}^{J_0}$ having the properties (\ref{equ:019}) and $(\ref{equ:020})$.
Fix $x\in E_{J_0}$ and write
\begin{equation*}
\mathcal{G}_j(x)=\int_{(\rn)^m}{2^{jmn}K_j\big(2^j(x-y_1),\dots,2^j(x-y_m) \big)\Big( \prod_{i\in\mathrm{I}} a_{i}(y_i)\Big) \Big(\prod_{i\in\mathrm{II}\cup\mathrm{III}}f_{i}(y_i)\Big)}d\yyy
\end{equation*} where $K_j:=\big(\sigma_{j,1}^{\kappa_1,\kappa_2}(2^j\cdot) \big)^{\vee}$.
Let $c_i$ denote the center of the cube $Q_i$ and use the notation
\begin{equation*}
K_j^{(u,w)}(x,\yyy):=K_j\big(y_1,\dots,y_{u-1},2^j(x-y_u),\dots,2^j(x-y_w),y_{w+1},\dots,y_m\big)
\end{equation*}
for simplicity, as before.

Since $|{x-c_i}|\approx |{x-y_i}|$ for $x\not\in Q_{i}^{\ast}$ and $y_{i}\in Q_{i}$, we see that
\begin{align*}
&\Big( \prod_{i\in J_0}\langle 2^j(x-c_i) \rangle^{s_i}\Big)  |{\mathcal{G}_j(x)}|\\
&\lesssim 2^{jmn}\int_{({\rn})^{m}} \Big( \prod_{i\in J_0}\langle 2^j(x-y_i) \rangle^{s_i}\Big)  \big|{K_{j}^{(1,m)}\big(x,\yyy\big)}\big| \Big( \prod_{i\in \mathrm{I}}|{a_{i}(y_i)}|\Big) \Big( \prod_{i\in\mathrm{II} \cup \mathrm{III} }|{f_{i}(y_i)}|\Big) d\yyy\\
&\leq 2^{jmn}\int_{(\rn)^m} \Big( \prod_{i\in J_0}\langle 2^j(x-y_i) \rangle^{s_i}\Big)  \big|{K_{j}^{(1,m)}\big(x,\yyy\big)}\big| \Big( \prod_{i\in \mathrm{I}} |{Q_{i}}|^{-1/p_i}\chi_{Q_{i}}(y_i)\Big) \Big( \prod_{i\in\mathrm{II}} |{f_{i}(y_i)}|\Big) d\yyy.
\end{align*}
We now fix  $k\in J_0$ and estimate the last integral by 
\begin{align*}
&\int_{y_k\in\rn} \Big\Vert{\Big( \prod_{i\in J_0\cup \mathrm{II}} \langle 2^j(x-y_i) \rangle^{s_i} \Big) K_{j}^{(1,m)}\big(x,\yyy\big)}\Big\Vert_{L^{\infty}(\yyy_{J_0\setminus \{k\}})L^{1}(\yyy_{\mathrm{I} \setminus J_0})L^{t'}(\yyy_{\mathrm{II}})L^{1}(\yyy_{\mathrm{III}})}\\
&\qquad \qquad \times \Big\Vert{\prod_{i\in J_0}|{Q_{i}}|^{-1/p_i}\chi_{Q_{i}}(y_i)}\Big\Vert_{L^{1}(\yyy_{J_0\setminus \{k\}})} \Big\Vert{\prod_{i\in \mathrm{I} \setminus J_0}|{Q_{i}}|^{-1/p_i}\chi_{Q_{i}}(y_i)}\Big\Vert_{L^{\infty}(\yyy_{\mathrm{I} \setminus J_0})}\\
&\qquad  \qquad \times \Big\Vert{\prod_{i\in \mathrm{II}}\langle 2^j(x-y_i) \rangle^{-s_i} f_{i}(y_i)}\Big\Vert_{L^{t}(\yyy_{\mathrm{II}})}  dy_{k},
\end{align*}
where we used the notations $\displaystyle \yyy_J:=\otimes_{i\in J}y_i$ for all $J$ (for example, $\yyy_{\mathrm{I}}=(y_1,\dots,y_l)$, $\yyy_{\mathrm{II}}=(y_{l+1},\dots,y_{\rho})$, and so on),
and
\begin{equation*}
\Vert{F(z_1,z_2)}\Vert_{L^{p}(z_1)L^{q}(z_2)}:=\big\Vert  \Vert F(z_1,z_2)\Vert_{L^p(z_1)} \big\Vert_{L^q(z_2)}. 
\end{equation*}
Using a change of variables we write
\begin{align*}
&\Big\Vert{\Big( \prod_{i\in J_0\cup \mathrm{II}}\langle 2^j(x-y_i) \rangle^{s_i}\Big)  K_{j}^{(1,m)}\big(x,\yyy \big) }\Big\Vert_{L^{\infty}(\yyy_{J_0\setminus \{k\}})L^{1}(\yyy_{\mathrm{I} \setminus J_0})L^{t'}(\yyy_{\mathrm{II}})L^{1}(\yyy_{\mathrm{III}})}\\
&=2^{-jn \textup{Card}(\mathrm{I} \setminus J_0)}2^{-(jn/t')\textup{Card}(\mathrm{II})}2^{-jn \textup{Card}(\mathrm{III})}  \\
&\qquad \times \Big\Vert  \langle 2^j(x-y_k)\rangle^{s_k}\Big( \prod_{i\in J_0\cup \mathrm{II}\setminus \{k\}}\langle y_i\rangle^{s_i}\Big) K_j^{(k,k)}\big(x,\yyy \big)\Big\Vert_{L^{\infty}(\yyy_{J_0\setminus \{k\}})L^{1}(\yyy_{\mathrm{I} \setminus J_0})L^{t'}(\yyy_{\mathrm{II}})L^{1}(\yyy_{\mathrm{III}})}.
\end{align*}
Now   H\"older's inequality with $s_i>n/t$ and Lemma \ref{lem:LInfL2} yield
\begin{align*}
&\Big\Vert  \langle 2^j(x-y_k)\rangle^{s_k}\Big( \prod_{i\in J_0\cup \mathrm{II}\setminus \{k\}}\langle y_i\rangle^{s_i}\Big) K_j^{(k,k)}\big(x,\yyy \big)\Big\Vert_{L^{\infty}(\yyy_{J_0\setminus \{k\}})L^{1}(\yyy_{\mathrm{I} \setminus J_0})L^{t'}(\yyy_{\mathrm{II}})L^{1}(\yyy_{\mathrm{III}})}\\
&\lesssim \Big\Vert \langle 2^j(x-y_k)\rangle^{s_k}   {\Big( \prod_{i\in \Lambda\setminus \{k\}} \langle y_i \rangle^{s_i} \Big) K_{j}^{(k,k)}\big(x,\yyy\big) }\Big\Vert_{ L^{\infty}(\yyy_{J_0\setminus \{k\}}) L^{t'}(\yyy_{\Lambda \setminus J_0})}\\
&\lesssim \Big\Vert \langle 2^j(x-y_k) \rangle^{s_k} \Big( \prod_{i\in \Lambda\setminus \{k\}} \langle  y_i \rangle^{s_i} \Big)  K_{j}^{(k,k)}\big(x,\yyy \big) \Big\Vert_{L^{t'}(\yyy_{\Lambda\setminus \{k\}})}.
\end{align*}
Morover, we have
\begin{align*}
\Big\Vert{\prod_{i\in J_0}|{Q_{i}}|^{-1/p_i}\chi_{Q_{i}}(y_i)}\Big\Vert_{L^{1}(\yyy_{J_0\setminus \{k\}})}&\lesssim \Big(\prod_{i\in J_0}|Q_i|^{1-1/p_i} \Big)\chi_{Q_k}(y_k)|Q_k|^{-1},\\
\Big\Vert{\prod_{i\in \mathrm{I} \setminus J_0}|{Q_{i}}|^{-1/p_i}\chi_{Q_{i}}(y_i)}\Big\Vert_{L^{\infty}(\yyy_{\mathrm{I} \setminus J_0})}&\leq \prod_{i\in \mathrm{I}\setminus J_0}{|Q_i|^{-1/p_i}},\\
\Big\Vert{\prod_{i\in \mathrm{II}}\langle 2^j(x-y_i) \rangle^{-s_i} f_{i}(y_i) }\Big\Vert_{L^{t}(y_{\mathrm{II}})}&\lesssim  2^{-(jn/t)\textup{Card}(\mathrm{II})}\prod_{i\in\mathrm{II}}\mathfrak{M}_{s_i,2^j}^tf_i(x)  \\
&\lesssim 2^{-(jn/t)\textup{Card}(\mathrm{II})} \prod_{i\in \mathrm{II}} \mathcal{M}_{t}(f_{i})(x),
\end{align*} 
where the last inequality follows from (\ref{maximalcompare}) with $s_i>n/t$.

Combining the above inequalities, we obtain that for $x\in E_{I_0}$,
\begin{align*}
\Big( \prod_{i\in J_0}\langle 2^j(x-c_i) \rangle^{s_i}\Big)  |{\mathcal{G}_j(x)}| &\lesssim 2^{jn \textup{Card}(J_0)}  H_j^{(k,0)}(x) \Big(\prod_{i\in J_0}|{Q_i}|^{1-1/p_i} \Big) \\
&\qquad \qquad \times\Big( \prod_{i\in \mathrm{I} \setminus J_0}|{Q_i}|^{-1/p_i}\Big)\Big( \prod_{i\in \mathrm{II} }\mathcal{M}_{t}(f_i)(x)\Big)
\end{align*}
where $H_j^{(k,0)}$ is defined as 
 \begin{align*}
    H_j^{(k,0)}(x)&:=  \dfrac1{{|{Q_k}|}}\int_{Q_k}\langle 2^j(x-y_k)\rangle^{s_k} \Big\Vert{\Big( \prod_{i\in\Lambda\setminus \{k\}}\langle y_i\rangle^{s_i} \Big) K_j^{(k,k)}\big(x,\yyy\big)}\Big\Vert_{L^{t'}(\yyy_{\Lambda\setminus \{k\}})}dy_k,
    \end{align*} which is the counterpart of $h_j^{(k,0)}$ in the proof of Lemma \ref{keylemma1}.
Then the argument that led to (\ref{hr'}), with (\ref{compactembedding1}), proves that
\begin{equation}\label{equ:025}
\Vert H_j^{(k,0)} \Vert_{L^{t'}(\rn)}\lesssim 2^{-jn/t'}\mathcal{L}_{\sss}^r[\sigma].
\end{equation}

  On the other hand, applying the vanishing moment condition of $a_k$ and Lemma \ref{smalllemma}, we write
    \begin{align*}
    \big|\mathcal{G}_j(x)\big| &\lesssim 2^{jmn}\sum_{|\alpha|=N_k+1}\int_0^1{\int_{(\rn)^m}{\big( 2^j|y_k-c_k|\big)^{N_k+1}}}\Big(\prod_{i\in I}{|Q_i|^{-1/p_i}\chi_{Q_i}(y_i)}\Big)\\
    &\quad \times {{{\Big|\partial_k^{\alpha}K_j\big(2^j(x-y_1),\dots,2^j(x-y_{k-1}),2^jx_{c_k,y_k}^{t},2^j(x-y_{k+1}),\dots,2^j(x-y_m)\big) \Big|} }}\\
    &\quad \times \Big( \prod_{i\in \mathscr{II}}{|f_i(y_i)|}\Big)d\yyy dt    
    \end{align*}
    where $x_{c_k,y_k}^t:=x-c_k-t(y_k-c_k)$.
   Since $|{x_{c_k,y_k}^t}|\approx |{x-c_k}|$ for $x\not\in Q_k^{\ast}$, $y_k\in Q_k$, and $0<t<1$, arguing as in (\ref{equ:H1NFunc}), we obtain that for $x\in E_{J_0}$,
    \begin{align*}
   &\Big( \prod_{i\in J_0}\langle 2^j(x-c_i)\rangle^{s_i}\Big){|{\mathcal{G}_j(x)}|}\\
   &\lesssim 2^{jmn}\big(2^jl(Q_k) \big)^{N_k+1} \sum_{|\alpha|=N_k+1}\int_0^1 \int_{y_k\in\rn}     \Big\Vert \langle 2^jx_{c_k,y_k}^t \rangle^{s_k} \Big( \prod_{i\in J_0\cup {II}\setminus \{k\}}{ \langle 2^j(x-y_i)\rangle^{s_i}}\Big)\\
   &\qquad    \times   \partial^{\alpha}_k K_j\big(2^j(x-y_1),\dots ,2^j(x-y_{k-1}),2^j x_{c_k,y_k}^t,2^j(x-y_{k+1}),\dots,\\
   &\qquad  \qquad \qquad \qquad \qquad \qquad \qquad \qquad , 2^j(x-y_m)\big)  \Big\Vert_{L^{\infty}(\yyy_{J_0\setminus \{k\}})L^{1}(\yyy_{\mathrm{I} \setminus J_0})L^{t'}(\yyy_{\mathrm{II}})L^{1}(\yyy_{\mathrm{III}})}  \\
   &\qquad \times \Big\Vert{\prod_{i\in J_0}|{Q_{i}}|^{-1/p_i}\chi_{Q_{i}}(y_i)}\Big\Vert_{L^{1}(\yyy_{J_0\setminus \{k\}})} \Big\Vert{\prod_{i\in \mathrm{I} \setminus J_0}|{Q_{i}}|^{-1/p_i}\chi_{Q_{i}}(y_i)}\Big\Vert_{L^{\infty}(\yyy_{\mathrm{I} \setminus J_0})}\\
& \qquad \times \Big\Vert{\prod_{i\in \mathrm{II}}\langle 2^j(x-y_i) \rangle^{-s_i} f_{i}(y_i)}\Big\Vert_{L^{t}(\yyy_{\mathrm{II}})}  dy_{k} dt\\
    & \lesssim 2^{jn \textup{Card}(J_0)} H_j^{(k,1)}(x) \Big(\prod_{i\in J_0}|{Q_i}|^{1-1/p_i} \Big)\Big( \prod_{i\in \mathrm{I} \setminus J_0}|{Q_i}|^{-1/p_i}\Big)\Big( \prod_{i\in \mathrm{II} }\mathcal{M}_{t}(f_i)(x)\Big)
    \end{align*}
 where 
    \begin{align*}
    H_j^{(k,1)}(x)&:=\big( 2^jl(Q_k)\big)^{N_k+1}\sum_{|\alpha|=N_k+1}  \dfrac1{{|{Q_k}|}}\int_0^1\int_{Q_k}\langle 2^jx_{c_k,y_k}^{t}\rangle^{s_k}\\
       & \quad \times \Big\Vert{\prod_{i\in\Lambda\setminus \{k\}}\langle\cdot_i\rangle^{s_i}  \partial_k^{\alpha}K_j\big(\cdot_1,\ldots,\cdot_{k-1},2^jx_{c_k,y_k}^{t},\cdot_{k+1}, \dots,\cdot_m\big)}\Big\Vert_{L^{t'}((\rn)^{m-1})}dy_k dt.
    \end{align*}
Using   Minkowski's inequality, Lemma \ref{lem:LInfL2} and (\ref{compactembedding1}), we deduce  
\begin{equation}\label{equ:027}
\Vert{H_j^{(k,1)}}\Vert_{L^{t'}({\mathbb R}^n)}\lesssim 2^{-jn/t'} \big( 2^jl(Q_k)\big)^{N_k+1}\mathcal{L}_{\sss}^{r}[\sigma].
\end{equation}

So far, we have proved that for $x\in E_{J_0}$ and $k\in J_0$,
\begin{align}\label{gjest}
|{\mathcal{G}_{j}(x)}| &\lesssim  2^{jn \textup{Card}(J_0)}  \Big(\prod_{i\in J_0}\langle 2^j(x-c_i) \rangle^{-s_i} |{Q_i}|^{1-1/p_i }\chi_{(Q_i^*)^c}(x)\Big)\\
&\quad \times \Big( \prod_{i\in \mathrm{I} \setminus J_0}|{Q_i}|^{-1/p_i}\chi_{Q_i^*}(x)\Big)\Big( \prod_{i\in \mathrm{II} }\mathcal{M}_{t}(f_i)(x)\Big) \min{\big(H_j^{(k,0)}(x), H_j^{(k,1)}(x)\big)}. \nonumber
\end{align}

We choose $\{\alpha_i\}_{i\in I_0}$ and $\{\beta_i\}_{i\in I_0}$ as in (\ref{betak}) by replacing $\{1,\dots,v\}$ and $r'$ by $J_0$ and $t'$, respectively, which is possible since 
\begin{equation*}
\sum_{i\in J_0}{\min{\big(0,{s_i}/{n}-{1}/{p_i} \big)}}>-{1}/{t'}
\end{equation*}  
by virtue of   condition (\ref{stot2}).
Then we have
\begin{equation*}
\alpha_i,\beta_i >0, \quad {s_i}/{n}>{1}/{p_i}- {\beta_{k}}/{t'}=\alpha_i, \quad \sum_{i\in J_0}\beta_i =1.
\end{equation*}
Now if we set 
\begin{equation*}
u_{i,j}^{J_0}(x):= \big(\mathcal{L}_{\sss}^{r}[\sigma] \big)^{-\beta_i}2^{jn}|{Q_i}|^{1-1/p_i} \langle 2^j(x-c_i) \rangle^{-s_i} \chi_{(Q_i^*)^{c}}(x) \Big( \min{\big(H_j^{(i,0)}(x), H_j^{(i,1)}(x)\big)} \Big)^{\beta_i},
\end{equation*}
(\ref{equ:019}) is immediate from (\ref{gjest}) since $\sum_{i\in J_0}\beta_i=1$.

It remains to verify (\ref{equ:020}).
 H\"older's inequality with $1/p_i=\beta_i/r'+\alpha_i$ yields that
\begin{align*}
\Vert{u_{i,j}^{J_0}}\Vert_{L^{p_i}(\rn)} &\leq  \big( \mathcal{L}_{\sss}^{r}[\sigma]\big)^{-\beta_i}2^{jn}l(Q_i)^{n(1-1/p_i)} \big\Vert{\langle 2^j(\cdot-c_i) \rangle ^{-s_i}\chi_{(Q_i^*)^{c}}}\big\Vert_{L^{1/\alpha_i}(\rn)}\\
 &\qquad \qquad \times \min{\Big(\Vert H_j^{(i,0)}\Vert_{L^{t'}(\rn)}^{\beta_i},\Vert H_j^{(i,1)}\Vert_{L^{t'}(\rn)}^{\beta_i} \Big)}.
\end{align*}
Since $s_i >\alpha_i n$, we have
\begin{equation*}
\big\Vert{\langle 2^j(\cdot-c_i) \rangle ^{-s_i} \chi_{(Q_i^*)^{c}}}\big\Vert_{L^{1/\alpha_i}(\rn)} \lesssim 2^{-jn\alpha_i}\min{\big(1,\big(2^jl(Q_i)\big)^{-(s_i-\alpha_i n)} \big)},
\end{equation*}
and the estimates \eqref{equ:025} and \eqref{equ:027} prove
\begin{align*}
\min{\Big(\Vert H_j^{(i,0)}\Vert_{L^{t'}(\rn)}^{\beta_i},\Vert H_j^{(i,1)}\Vert_{L^{t'}(\rn)}^{\beta_i} \Big)} \lesssim  2^{-jn\beta_i/t'} \big(\mathcal{L}_{\sss}^{r}[\sigma] \big)^{\beta_i}\min{\big( 1,\big( 2^jl(Q_i)\big)^{\beta_i(N_i+1)}\big)}.
\end{align*}
Thus, 
\begin{equation*}
\Vert u_{i,j}^{J_0}\Vert_{L^{p_i}(\rn)}\lesssim \begin{cases}
\big( 2^jl(Q_i)\big)^{-(n/p_i-n)+\beta_i(N_i+1)},\quad  & \text{ if }~ 2^jl(Q_i)\leq 1\\
 \big( 2^jl(Q_i)\big)^{-(n/p_i-n)-(s_i-\alpha_i n)}, \quad & \text{ if }~ 2^j(Q_i)>1
\end{cases}
\end{equation*} since $1-\alpha_i-\beta_i/t'=1-1/p_i$.
This implies \eqref{equ:020} with $\gamma_j= -(n/p_i-n)+\beta_i(N_i+1)$ and $\delta_i=n/p_i-n+s_i-\alpha_i n$.
We have $\gamma_k,\delta_k>0$ as $N_k$ is sufficiently large and $s_i>\alpha_i n$.

This completes the proof of Lemma \ref{keylemma2}.

\hfill

\subsection{Proof of Lemma \ref{keylemma3}}

The proof is similar to that of Lemma \ref{keylemma2}.
As in the proof of Lemma \ref{keylemma2}, we choose $1<t<r$ such that
\begin{equation*}
s_1,\dots,s_m>d/t>d/r, \qquad \sum_{k\in J}{\big({s_k}/{n}-{1}/{p_k} \big)}>-{1}/{t'}>-{1}/{r'}
\end{equation*}
for every nonempty subset $J\subset \mathrm{I}$, and observe that (\ref{compactembedding1}) holds.

For each $J_0\subset \mathrm{I}$, let 
\begin{equation*}
E_{J_0}:=\Big( \bigcap_{i\in \mathrm{I}\setminus J_0}Q_i^*\Big)\setminus \Big( \bigcup_{i\in J_0}{Q_i^*}\Big)
\end{equation*} 
and we decompose the left-hand side of (\ref{kl3condition1}) as
\begin{align*}
\sum_{J_0\subset \mathrm{I}}\Big( \sum_{j\in\zz}\big|T_{\sigma_{1,j}^{\kappa}} \big( a_{1},\dots,a_{l},f_{l+1},\dots,f_m\big)(x)\big|^2\Big)^{1/2}\chi_{E_{J_0}}(x).
\end{align*}
Since it is a finite sum over $J_0$, we need to prove that for each $J_0\subset I$, there exist nonnegative functions $b_i^{J_0}$, $i\in \mathrm{I}$, and $F_i^{J_0}$, $ i\in \mathrm{II}$ satisfying that for all $x\in E_{J_0}$
\begin{align}\label{kl33condition1}
&\Big( \sum_{j\in\zz}\big|T_{\sigma_{j,1}^{\kappa}}\big( a_{1},\dots,a_{l},f_{l+1},\dots,f_m\big)(x)\big|^2\Big)^{1/2}\lesssim \mathcal{L}_{\sss}^{r}[\sigma]\Big( \prod_{i\in\mathrm{I}}b_i^{J_0}(x)\Big)\Big( \prod_{i\in\mathrm{II}}F_i^{J_0}(x)\Big),
\end{align} 
\begin{equation}\label{kl33condition2}
\Vert b_i^{J_0}\Vert_{L^{p_i}(\rn)}\lesssim 1,\quad \text{ for }~ i\in \mathrm{I},
\end{equation}
\begin{equation}\label{kl33condition3}
\Vert F_i^{J_0}\Vert_{L^{p_i}(\rn)}\lesssim \Vert f_i\Vert_{L^{p_i}(\rn)}, \quad \text{ for }~ i\in \mathrm{II}.
\end{equation}

\hfill

Let us first assume $J_0=\emptyset$. In this case, the proof consists of three cases.

{\bf Case1 : $\kappa \in \mathrm{I}$. } 
Using (\ref{condexpression1}), Lemma \ref{keyestilemma}, (\ref{maximalcompare}), (\ref{sigmajkest}), and (\ref{compactembedding1}), we obtain
\begin{equation*}
\big| T_{\sigma_{j,1}^{\kappa}}\big(a_1,\dots,a_l,f_{l+1},\dots,f_m \big)(x)\big|\lesssim \mathcal{L}_{\sss}^{r}[\sigma]\mathcal{M}_t(a_{\kappa})_j(x)\Big( \prod_{i\in \mathrm{I}\setminus \{\kappa\}}\mathcal{M}_ta_i(x)\Big)\Big(\prod_{i\in\mathrm{II}}{\mathcal{M}_tf_i(x)} \Big),
\end{equation*} where we applied $\mathcal{M}_tf_i(x)\leq \Vert f_i\Vert_{L^{\infty}(\rn)}=1$ for $i\in\mathrm{III}$.
We now take 
\begin{align*}
b_{\kappa}^{J_0}(x)&:=\Big(\sum_{j\in\zz} \big( \mathcal{M}_t(a_{\kappa})_j(x)\big)^2 \Big)^{1/2}\chi_{Q_{\kappa}^*}(x),\\
b_i^{J_0}(x)&:=\mathcal{M}_ta_i(x)\chi_{Q_i^*}(x),\qquad i\in \mathrm{I}\setminus \{\kappa\},\\
F_i^{J_0}(x)&:=\mathcal{M}_tf_i(x), \qquad  i\in \mathrm{II}
\end{align*}
and then (\ref{kl33condition1}) holds.
Furthermore, (\ref{kl33condition2}) and (\ref{kl33condition3}) follow from   H\"older's inequality, (\ref{maximal1}) with $t<2$, and (\ref{littlewood}). 
To be specific, the estimates for $i\in\mathrm{I}\setminus \{\kappa\}$ or for $i\in\mathrm{II}$ are clear, and 
\begin{equation*}
\Vert b_{\kappa}^{J_0}\Vert_{L^{p_{\kappa}}(\rn)}\leq |Q_{\kappa}^*|^{1/p_{\kappa}-1/2}\big\Vert \big\{ \mathcal{M}_t(a_\kappa)_j\big\}_{j\in\zz}\big\Vert_{L^2(\ell^2)}\lesssim |Q_{\kappa}|^{1/p_{\kappa}-1/2}\Vert a_{\kappa}\Vert_{L^2(\rn)}\lesssim 1.
\end{equation*}

{\bf Case2 : $\kappa \in \mathrm{II}$. }  It can be proved in a similar way. Indeed, (\ref{kl33condition1}) holds with
\begin{align*}
b_i^{J_0}(x)&:=\mathcal{M}_ta_i(x)\chi_{Q_i}(x),\qquad i\in\mathrm{I},\\
F_{\kappa}^{J_0}(x)&:=\Big(\sum_{j\in\zz}\big(\mathcal{M}_t(f_{\kappa})_j(x)\big)^2\Big)^{1/2},\\
F_i^{J_0}(x)&:=\mathcal{M}_tf_i(x), \qquad  i\in\mathrm{II}\setminus \{\kappa\}.
\end{align*}
It is also obvious that (\ref{kl33condition2}) and (\ref{kl33condition3}) hold as (\ref{littlewood}) is applied in the case $i=\kappa$.

{\bf Case3 : $\kappa \in \mathrm{III}$. } 
 We utilize Lemma \ref{bmoboundlemma} as we did in {\bf Case3} that appeared in the proof of Lemma \ref{keylemma2}.
Using (\ref{secondexpression2}), Lemma \ref{keyestilemma}, (\ref{maximalcompare}), (\ref{sigmajkest}), and (\ref{compactembedding1}), we obtain that
\begin{align*}
&\big| T_{\sigma_{j,1}^{\sigma+1}}\big(a_1,\dots,a_l,f_{l+1},\dots,f_m \big)(x)\big|\\
 &\lesssim \mathcal{L}_{\sss}^{r}[\sigma] \Big( \prod_{i\in\mathrm{I}}\mathcal{M}_ta_i(x)\Big) \mathfrak{M}_{s_{l+1},2^j}^t(f_{l+1})^{j+1,m}(x) \Big(\prod_{i\in\mathrm{II}\setminus \{l+1\}}{\mathcal{M}_tf_i(x)} \Big)\mathfrak{M}_{s_{\kappa},2^j}^{t}(f_{\kappa})_j(x).
\end{align*}
Now we take
\begin{align*}
b_i^{J_0}(x)&:=\mathcal{M}_ta_i(x)\chi_{Q_i}(x),\qquad i\in\mathrm{I},\\
F_{l+1}^{J_0}(x)&:=\Big(\sum_{j\in\zz}\big(\mathfrak{M}_{s_{l+1},2^j}^t(f_{l+1})^{j+1,m}(x)\big)^2 \big( \mathfrak{M}_{s_{\kappa},2^j}^t(f_{\kappa})_j(x)\big)^2 \Big)^{1/2},\\
F_i^{J_0}(x)&:=\mathcal{M}_tf_i(x), \quad  i\in \mathrm{II}\setminus \{l+1\}.
\end{align*}
Then (\ref{kl33condition1}), (\ref{kl33condition2}), and (\ref{kl33condition3}) are all true for $i\not= l+1$, and (\ref{kl33condition3}) for $i=l+1$ follows from Lemma \ref{bmoboundlemma}.

\hfill

Now we consider the case $J_0\not=\emptyset$.
The proof is immediate from the argument in the proof of Lemma \ref{keylemma2}. 
We define, like (\ref{gjdef}), 
\begin{equation*}
\mathcal{G}_j:=T_{\sigma_{j,1}^{\kappa}}\big(a_1,\dots,a_l,f_{l+1},\dots,f_m\big).
\end{equation*} 
Then (\ref{equ:019}) still holds in the present case with (\ref{equ:020}). Let $b_i^{J_0}$, $i\in \mathrm{I}$, and $F_i^{J_0}$, $i\in \mathrm{II}$, be defined as in (\ref{bjidef}) and (\ref{fiidef}), and apply the embedding $\ell^1\hookrightarrow \ell^2$ to obtain
 that the left-hand side of (\ref{kl33condition1}) is bounded by
\begin{align*}
\sum_{j\in\zz}{\big| \mathcal{G}_j(x)\big|}\lesssim \mathcal{L}_{\sss}^{r}[\sigma]\Big(\prod_{i=1}^{l}b_i^{J_0}(x) \Big)\Big(\prod_{i=l+1}^{\rho}F_{i}^{J_0}(x) \Big),
\end{align*} which proves (\ref{kl33condition1}).
In addition, (\ref{kl33condition2}) and (\ref{kl33condition3}) are obvious from (\ref{kl22condition2}) and (\ref{kl22condition3}), respectively.

This completes the proof.

\medskip
\noindent {\bf Acknowledgment:} We would like to thank Professors M. Mastylo and N. Tomita for providing us important references related to complex interpolation.  We would also like to thank the anonymous referee for 
 his/her careful reading and useful comments.

\end{document}